\documentclass{article}

\usepackage{graphicx}

\usepackage[fleqn]{amsmath}
\usepackage{amsmath}
\usepackage{amsthm}
\usepackage{amssymb}
\usepackage{amsbsy}
\usepackage{amsfonts}
\usepackage{mathrsfs}
\usepackage[all]{xy}
\usepackage{amstext}
\usepackage{amscd}
\usepackage[dvips]{epsfig} 
\usepackage{psfrag}
\usepackage{enumerate}
\usepackage{flafter}

\allowdisplaybreaks

\theoremstyle{plain}
\newtheorem{thm}{Theorem}[section]
\newtheorem{prop}[thm]{Proposition}
\newtheorem{cor}[thm]{Corollary}
\newtheorem{lem}[thm]{Lemma}
\theoremstyle{definition}
\newtheorem{exa}[thm]{Example}

\newtheorem{rem}[thm]{Remark}
\newtheorem{defn}[thm]{Definition}

\newcommand*{\longhookrightarrow}{\ensuremath{\lhook\joinrel\relbar\joinrel\rightarrow}}

\def\Ker{\mathop{\mathrm{Ker}}\nolimits}
\def\Coker{\mathop{\mathrm{Coker}}\nolimits}
\def\Hom{\mathop{\mathrm{Hom}}\nolimits}

\def\F{\mathop{\mathbb{F}}\nolimits}

\def\gr{\mathop{\mathrm{gr}}\nolimits}

\newcommand{\tri}{{ \lhd}}

\newcommand{\SQ}{{\mathsf{Sp}}}
\newcommand{\col}{{\rm Col}}

\newcommand{\lra}{\longrightarrow}
\newcommand{\ra}{\rightarrow}

\newcommand{\Z}{{\Bbb Z}}
\newcommand{\As}{{\rm As }}
\newcommand{\D}{{\cal D}_{g}}

\newcommand{\DD}{{{\cal D}_{g}}}

\newcommand{\M}{{\cal M}}
\newcommand{\CC}{{\cal C }}

\newcommand{\X}{\widetilde{X}  }

\newcommand{\OX}{{\rm O}(X) }

\newcommand{\pc}[2]{\mbox{$\begin{array}{c}
\includegraphics[scale=#2]{#1.eps}
\end{array}$}}
\begin{document}
\large
\begin{center}
{\bf\Large Homotopical interpretation of link invariants from finite quandles }
\end{center}

%\vskip 0.5pc

%\begin{center}
%{\bf\Large 
%and Dijkgraaf-Witten invariants of 3-manifolds}
%\end{center}
\vskip 1.5pc\begin{center}{\Large Takefumi Nosaka}\end{center}\vskip 1pc\begin{abstract}\baselineskip=13pt \noindent This paper demonstrates a topological meaning of quandle cocycle invariants of links with respect to finite connected quandles $X$, from a perspective of homotopy theory:
Specifically, for any prime $\ell$ which does not divide the type of $X$, the $\ell$-torsion of this invariants is equal to a sum of
the colouring polynomial and a $\Z$-equivariant part of the Dijkgraaf-Witten invariant of a cyclic branched covering space.
Moreover, our homotopical approach involves application of computing some third homology groups and second homotopy groups of the classifying spaces of quandles,
from results of group cohomology.
%Furthermore we develop an approach for automorphism groups associated to quandles.
\end{abstract}

\begin{center}

\ \ \ bordism group, orthogonal and symplectic group, mapping class group \ \ \
\end{center}
\large\baselineskip=16pt\section{Introduction.}
%A quandle is a set with a binary operation
A quandle, $X$, is a set with a binary operation $\lhd :X^2 \ra X$ the axioms of which were motivated by knot theory.
The operation is roughly the conjugation of a group;
In analogy of group (co)homology,
Carter et. \cite{CJKLS,CEGS} defined quandle (co)homology,
and used some cocycles to define some knot-invariants, which we call {\it quandle cocycle invariants}.
As mentioned in \cite{CJKLS}, the definition is
%analogous to the Dikgraaf-Witten invariant; The construction can be seen
an analogue of the Dijkgraaf-Witten invariant \cite{DW} of closed oriented 3-manifolds $M$, which is constructed from a finite group $G$ and a 3-cocycle $\kappa \in H^3_{\gr}(G;A)$:
Precisely, the invariant is topologically expressed as the formal sum of pairings
%expressed by
\begin{equation}\label{DWDWder9} \mathrm{DW}_{\kappa }(M ):= \sum_{ f \in \Hom_{\mathrm{gr}} ( \pi_1(M), G ) } \langle f^{*}(\kappa ), [M] \rangle \in \Z[A],
\end{equation}
where $[M ] $ is the orientation 3-class in $ H_3(M ;\Z)$.
While this $ \mathrm{DW}_{\kappa }(M )$ is far from being computable by definition,
the cocycle invariants from quandles can be diagrammatically computable if we find an explicit description of a quandle cocycle.
However, %in general, it has been a difficult problem to compute the (co)homology groups and find such cocycles.
%Furthermore,
since the cocycle invariants are defined from link diagrams, their topological meanings are mysterious.

To solve such topological problems, this paper is inspired by Quillen-Friedlander theory \cite{Fri} in group homology.
Roughly speaking the theory, if a finite group $G$ is of Lie type and the abelianization of $G$ is almost zero,
the group homology $ H_*^{\gr}(G ;\Z) $ in a stable range can be computed from homotopy groups $\pi_*(K(G,1)^+)$.
Here, the space $K(G,1)^+$ is so-called the Quillen plus construction and, in many cases, is homotopic to an infinite loop space, i.e., $\Omega$-spectrum.
For example, the Chern-Simons invariant which is a ploto model of $\mathrm{DW}_{\kappa }(M ) $ can be universally interpreted as a plus construction or as the algebraic $K$-theory (see \cite{DW,JW}).

In this paper, from such a homotopical viewpoint, we shall focus on the quandle {\it homotopy} (knot-)invariant, which is valued in the group ring $\Z[\pi_2(BX)]$.
Here, $BX$ is the rack space $BX$ of a quandle $X$ and was introduced by Fenn, Rourke and Sanderson \cite{FRS1,FRS2}, as an analogy of the classifying spaces of groups.
Further, as is known \cite{FRS2,RS}, the homotopy invariant is universal in the sense that any cocycle invariant is
derived from the homotopy invariant
%However, as is known \cite{RS}, any such cocycle invariants can be derived from
%the quandle homotopy invariant
via the Hurewicz map $\pi_2(BX) \to H_2(BX;A)$ with local system (see \cite[\S 2]{Nos1} for the detailed formula).
Therefore, it is natural to address $\pi_2(BX)$ for the study the required topological meaning of the cocycle invariants.
%$ \Xi_X(L)$ is universal among all the cocycle invariants
%the homotopy group $\pi_2(BX) $ has
%a universal property of

The main theorem demonstrates a topological meaning of the {\it homotopy} invariants as a general statement,
together with computing quantitatively the homotopy groups $\pi_2(BX) $.
% from, more generally, ``connected" quandles,
%To explain this,
%Let us brifuly explain the main theorem \ref{introthm}.
To state the theorem, we set two simple notation:
%terminologies:
A quandle $X$ is said to be {\it connected}, if any $x, y \in X$ admit some $a_1, \dots, a_n \in X$ such that
$(\cdots (x \lhd a_1) \lhd \cdots) \lhd a_n = y$; {\it The type} $t_X $ of $X$ is
the minimal $N$ such that $ x= (\cdots (x \lhd y) \lhd \cdots) \lhd y $ [$N$-times on the right with $y$] for any $x, y \in X$.

\begin{thm}[{Theorem \ref{theorem}} \footnote{After \S \ref{ss1}, we employ reduced notation of $BX$ for simplicity.
For example, we use two groups, $\As(X)$ and $\Pi_2(X)$, such that $\pi_1(BX) \cong \As(X)$ and
$\pi_2(BX) \cong \Z \oplus \Pi_2(X)$ (see \eqref{bxbx}). Furthermore, instead of the homology $H_*(BX)$,
we mainly deal with the quandle homology $H_*^Q(X;\Z)$ introduced in \cite{CJKLS}, and
remark two known isomorphisms $H_2(BX) \cong \Z \oplus H_2^Q(X)$ and $H_3(BX) \cong \Z \oplus H_2^Q(X) \oplus H_3^Q(X)$ (see \eqref{eq.H2RQ} and \eqref{eq.H3RQ}).
}]\label{introthm}Let $X$ be a connected quandle of type $t_X $ and of finite order.
Let $\mathcal{H}_X : \pi_2(BX) \to H_2(BX;\Z)$ be the Hurewicz map as a common sense.
Then there is a homomorphism $\Theta_X $ from $\pi_2(BX)$ to the third group homology $ H_3^{\gr}( \pi_1 (BX);\Z )$ for which the sum
$$ \Theta_X \oplus \mathcal{H}_X : \pi_2(BX) \lra H_3^{\gr}( \pi_1 (BX);\Z) \oplus H_2(BX;\Z) $$
is an isomorphism after localization at $\ell \in \Z$, where $\ell$ is relatively prime to $t_X $.
%\footnote{To be coincide, for any prime $\ell$ that does not divide $t_X $, the map after $\ell$-localization is an isomorphism (possibly $\ell =0$).}.
\end{thm}
\noindent
Moreover, the sum $ \Theta_X \oplus \mathcal{H}_X$ which we call {\it the TH-map} plays an important role as follows:
We will show (Corollary \ref{thm1bts}) that, for any link $L \subset S^3$,
the homotopy invariant $\Xi_X(L) \in \Z[\pi_2(BX )]$ is sent to a sum of two invariants via the TH-map $ \mathcal{H}_X \oplus \Theta_X$:
precisely, the original quandle 2-cycle invariant $\in \Z[H_2(BX)]$ in \cite{CJKLS} and ``a $\Z$-equivariant part" of Dijkgraaf-Witten invariant of $\widehat{C}_{L}^{t_X} $ (see \eqref{DWDWder2} for the definition),
where $\widehat{C}_{L}^{t_X}$ is the $t_X $-fold cyclic covering space of $S^3$ branched over the link $L$.
To summarize, this corollary \ref{thm1bts} is figuratively summarized to the following equality in the group ring $ \Z[\pi_2(BX) ]$ up to $t_X\textrm{-torsion}$:
$$ \left(
\begin{array}{c}
{\rm Quandle \ homotopy} \\
{\rm invariant \ of \ } L \end{array}
\right)
=
\left(
\begin{array}{c}
\Z \textrm{-}{\rm equivariant \ part } \\
{\rm of \ }DW_{\kappa} ( \widehat{C}_{L}^{t_X}) \end{array}
\right) +
\left(
\begin{array}{c}
{\rm Quandle \ }2\textrm{cycle} \\
{\rm invariant \ of \ } L \end{array}
\right) %\ \ \ \ {\rm up \ to \ }t_X\ \textrm{-torsion}.
$$
Here remark that, as is shown \cite{Eis2}, the second term is topologically characterized by longitudes of $X$-colored links (see \S \ref{SS3j3} for details).
In conclusion, via the TH-map, % $ \mathcal{H}_X \oplus \Theta_X $,
the homotopy invariant $\Xi_X(L)$ without $t_X $-torsion is reduced to the two topological invariants as desired.

Further, we should compare with the previous works and explain the type $t_X$ in some details.
Two papers \cite{Kab,HN} tried to extract a topological meaning of the cocycle invariants;
however, as in \cite{FRS2, Cla, Nos2} as some topological approaches to $BX$,
the second homology $ H_2(BX;\Z)$ is an obstacle to study the space $BX$. Accordingly
successful works were only for the simplest quandle of the form $X=\Z/(2m-1)$ with $x \lhd y:= 2y-x$.
Furthermore, as seen in \S \ref{twoto}, the type $t_X$ is often smaller than the order of $ H_2(BX)$ in many cases.
Hence, the above theorem is quite general,
and implies that a minimal obstacle in studying $\pi_2(BX)$ and topological meanings
is the type $t_X $ of $X$, rather than the second homology $ H_2(BX)$.

%In fact, so far, there are only a few topological studies of $BX$ with respect to general quandles $X$
%(However, we refer to ).
%As a result of our theorem,

% see Theorem \ref{} for details). For this, since as mentioned above.

%analougous to group theory, the

%As the main theorem is applicable for several quandles,
In addition, this paper further determine $\pi_2(BX) $ of several quandles $X$:
%we obtain many examples to compute most parts of $\pi_2(BX)$.
``regular Alexander quandles", most ``symplectic quandles over $\F_q $", ``Dehn quandle $\D $", extended quandles and the connected quandles of order $\leq 8$ (see Section \ref{twoto} for the details).
%Theorems \ref{thm1b}, \ref{thm1c}).
%For example, regarding a symplectic quandle over $\F_q $ ``in a certain stable range",
%the $\pi_2(BX)$ is isomorphic to $\Z \oplus K_3(\F_q) $, where $K_3(\F_q) \cong \Z / (q^2 -1 )$ is the Quillen $K$-group of $\F_q$.
Here the computation of the homotopy group $\pi_2(BX)$ follows from computing $ H_2(BX )$ and $H_3^{\gr}( \pi_1 (BX)) $, as in Theorem \ref{introthm}. % in the right side.
In fact, for this computation, another paper \cite{Nos3} gave an algorithm to determine the group $\pi_1(BX)$ and the third homology $H_3^{\rm gr}(\pi_1(BX))$
up to $t_X$-torsion.
% further, we can determine the second homology $H_2(BX) $ in terms of $\pi_1(BX) $, thanks to a method of Eisermann \cite{Eis2} (see Theorem \ref{ehyouj2i3}).
%In summary, we conclude a computation of the $\pi_2(BX)$ up to $t_X $-torsion from computations of $ H_3^{\gr}( \pi_1 (BX)) $ and $H_2(BX)$.

%To be precise we will show that the TH-map is an isomorphism for several quandles:
%So, in general, we hope that the TH-map is an isomorphism for many connected quandles $X$.
%Incidentally, as an application, regarding a ``Dehn quandle $\D $", which is a certain conjugacy class of the mapping class group and is useful for Lefschetz fibrations, we show that $\pi_2(B\D)$ is either $ \Z \oplus \Z/24 $ or $ \Z \oplus \Z/48$ for $ g \geq 7$ (Theorem \ref{hom3ocdr}).

Finally, we emphasize two applications of our study on $\pi_2(BX)$.
First, our approach to $\pi_2(BX)$ involves a new method for computing the {\it third} homology $H_3(BX)$,
and establishes a relation between third {\it quandle} homology and {\it group} homology.
As a general result, with respect to a finite connected quandle $X$, we solve some torsion subgroups of $H_3(BX) $ in terms of group homology (Theorem \ref{thm3homology}).
In application, as examples,
we compute most torsions of the third homology groups $H_3(BX) $ of the regular Alexander quandles,
the symplectic quandles and spherical quandles over $\F_q $ in a stable range (see Sections \ref{s87Al} and \ref{s87asl}).
%In addition, concerning Alexander quandles of the forms $X= \F_q $,
%i.e., a $\Z[T^{\pm 1}]$-module with $x \lhd y := T x + (1-T)y $,
%we will show an isomorphism $H_3(BX) \cong H_3^{\gr}( \pi_1 (BX)) \oplus \bigl(H_2(BX) \wedge H_2(BX)\bigr) $ up to $2 t_X $-torsion (Theorem \ref{thm3homology2}). %in term of the group homology $ H_3^{\gr}( \pi_1 (BX);\Z) $
Here, we should note that
most of known methods for computing $H_3(BX)$ was a result of Mochizuki \cite{Moc2}, where
$X$ is an Alexander quandles on $\F_q$.
Although his presentation of $H^3(BX;\F_q) $ was a little complicated,
% (see \cite{Moc2}).
our result implies that the complexity of $H^3(BX;\F_q) $ stems from that of $ H_3^{\gr}( \pi_1 (BX)) $.

%Furthermore, we study a close relation between the homology $H_3(B\X)$ and $\pi_2(B\X)$ with respect to ``extended quandles $\X$".
%Here such a quandle $\X$ is constructed from a connected quandle $X$ of type $t_X$; see \S \ref{s87ab3e} for the definition.
%Given a finite connected quandle of type $t_X $, we can define such a quandle $\X$ and an epimorphism $\X \ra X$ \cite{Eis2} (see \S \ref{s87ab3e} the definition).
%As we see in Theorem \ref{3hom9}, if $X$ is of finite order, our approach above provides isomorphisms
%$$\pi_2(B \X) \cong H_3(B \X) \cong H_3^{\rm gr}(\As (X)) \oplus \Z \ \ \ \ \ \ \mathrm{ up \ to \ } t_X \textrm{-}\mathrm{ torsion}. $$
%This result and viewpoint from extended quandles $\X$ is of vital importance in the proof of Theorem \ref{introthm} (see \S \ref{SS3j3out}) and in a subsequent paper \cite{Nos4}.

On the other hand, our theorem suggests an approach for computing the Dijkgraaf-Witten invariant.
% with respect to a finite group $G$.
In contract to the simple formula \eqref{DWDWder9}, it is not so easy to compute this invariant exactly.
Actually, most known computations of the invariants are in the cases where $G$ are abelian (see, e.g., \cite{DW,Kab,HN}).
However, as mentioned above, the TH-map $\Theta_X \oplus \mathcal{H}_X$ implies that
we can deal with some $\Z$-equivariant parts of this invariants.
In fact, in the subsequent paper \cite{Nos4} via the quandle cocycle invariants, we will compute the invariants of some knots using Alexander quandles $X$ over $\F_q$,
whose $\pi_1(BX)$ are nilpotent. % (see \eqref{lower} for the lower central series).
As a result, triple Massey products of some Brieskorn manifolds $\Sigma(n,m,l)$ will be calculated; see \cite[\S 5]{Nos4}.

%regarding the 2-cohomology of $X$ as $H^2(BX;A)$.
%Therefore, it is important to know the concrete form of $\pi_2(BX)$
%in order to study such cocycle invariants and the quandle homotopy invariant.
%S. Carter et. al introduced quandle cohomologies with local coefficients $H_*^Q (X;A) $,
%and the quandle cocycle invariants with respect to these 2-cocycles.

%the parings of $\Xi_X(L) $ and . they further proposed an invariant of framed links $L \subset S^3$ valued in the group ring of
%second homotopy group $\Xi_X(L) \in \Z[ \pi_2(BX)] $, which we call quandle homotopy invariant (see \S \ref{} for the definition).
%After that, This evaluation $ \langle \phi , \Xi_X(L) \rangle \in \Z[A]$ is called quandle cocycle invariant,
%and has been much studied, since the evaluation can be relatively computed compared with the original one $\Xi_X(L)$.

This paper is organized as follows.
Section \ref{ss1} reviews the quandle homotopy invariant, and
Section \ref{ss151} states our results.
Section \ref{SS3j3outjlw} constructs the homomorphism $\Theta_X$, and %quandle cocycles with non-abelian coefficients and the signature.
Section \ref{SS3j3} reviews the second quandle homology.
% according to \cite{Eis2}.
Section \ref{ss2se} proves the Theorem \ref{introthm}. %
Section \ref{twoto} computes $\pi_2(BX)$ of some quandles, as examples.
%includingcontains the proofs of Theorems \ref{thm1b}, \ref{thm1c}: precisely, we will compute some concretely.
%Section \ref{s87} determines $\pi_2(BX)$ of the Dehn quandle in a stable range.
Section \ref{s87as} is devoted to computing the third homology $H_3(BX)$. %of the Dehn quandle in a stable range.
%In addition, Appendix \ref{Sin} proposes a calculating for automorphism groups of quandles.
%Appendix \ref{Sin2} computes some second quandle homology groups.
%\subsection*{Acknowledgment}
%The author expresses his gratitude to Tomotada Ohtsuki for helpful suggestions.
%He is particularly grateful to Yuichi Kabaya for making several suggestions for improvement.
%He also thank Naoyuki Monden, Masatoshi Sato, Sakasai and Yoshiro Yaguchi for useful comments on the mapping class group.

\vskip 0.7pc
\noindent
{\bf Conventional notation.}
Throughout this paper, most homology groups are with (trivial) integral coefficients; so we often omit writing coefficients, e.g., $H_n(X)$.
We write $H_n^{\rm gr}(G)$ for the group homology of a group $G$.
Furthermore, we denote a $\Z$-module $M $ localized at a prime $\ell$ by $M_{(\ell)}$.
Moreover, a homomorphism $f:A \ra B$ between abelian groups is said to be
{\it a $[1/N]$-isomorphism}, denoted by $ f: A \cong_{[1/N]} B $,
if the localization of $f$ at $\ell$ is an isomorphism for any prime $\ell$ that does not divide $N$.
In addition, we assume that every manifolds are in $C^{\infty}$-class and oriented,
and that any fields is not of characteristic $2$. %, connected.
%Furthermore, given a binary operation $\lhd: S^2 \ra X$, for $y\in X$, we denote by $\bullet \lhd^n y$ the $n $-times on the right operation with $y$.

\vskip 0.7pc
\baselineskip=16pt
\section{Review of quandles and the quandle homotopy invariants}\label{ss1}
To establish our results in \S \ref{ss151}, we will review quandles in \S \ref{s3s3w},
and quandle homotopy invariants of links in \S \ref{reqhi}.
% and link bordism group consisting of $X$-colorings.
% computations of the homotopy group $\Pi_2(X) $.
\subsection{Review of quandles}\label{s3s3w}

A {\it quandle} is a set, $X$, with a binary operation $\tri : X \times X \ra X$ such that
\begin{enumerate}[(i)]
\item The identity $ a \tri a = a $ holds for any $a \in X $.
\item The map $ (\bullet \tri a ): \ X \ra X$ defined by $x \mapsto x \tri a $ is bijective for any $a \in X$.
\item The identity $(a\tri b)\tri c=(a\tri c)\tri (b\tri c)$ holds for any $a,b,c \in X. $
\end{enumerate}
\noindent
We refer the reader to \S \ref{twoto} for some examples of quandles.
A quandle $X$ is said to be {\it of type $t_X $}, if $t_X >0$ is the minimal $N$ such that $x= x \lhd^N y $ for any $x,y \in X$,
where we denote by $\bullet \lhd^N y$ the $N $-times on the right operation with $y$.
Note that, if $X$ is of finite order, it is of type $t_X $ for some $t_X \in \Z$.
A map $f: X \ra Y$ between quandles is {\it a (quandle) homomorphism}, if $f(a\tri b)=f(a)\tri f(b)$ for any $a,b \in X$.
%Let $\Hom_{\rm Qd}(X,Y)$ denote the set of quandle homorphisms $f: X \ra Y.$

%\begin{exa}\label{conjex} (Conjugacy quandle)
%A union of some conjugacy classes of a group $G$ is a quandle with the conjugacy operation $x\tri y:=y^{-1}xy$ for any $x,y \in G$.
%\end{exa}
% ??As observed above, quandle consists of, figuratively speaking, `operations itself centered at $y \in X$', which can be described as homogenous spaces (see \cite[\S 7]{Joy} for detail).

Next, we review {\it the associated group} \cite{FRS1}, denoted by $\As (X)$.
This group is %the abstract group defined by generators $e_x$ labeled by $x \in X$ modulo the relations $ e_x \cdot e_y = e_y \cdot e_{x\tri y} $ with $x,\ y \in X $.
%To be precise, the group $\As(X)$
is defined by the group presentation
$$ \As(X)= \langle \ e_x \ \ (x \in X )\ \ | \ \ e_{x\tri y}^{-1} \cdot e_y^{-1} \cdot e_x \cdot e_y \ \ (x,y \in X)\ \ \rangle.$$
We fix an action $\As(X)$ on $X$ defined by $ x \cdot e_y:= x \lhd y$ for $x,y \in X$.
Note the equality
\begin{equation}\label{hasz2} e_{x \cdot g} = g^{-1} e_x g \in \As (X) \ \ \ \ \ \ \ (x \in X , \ \ g \in \As (X) ), \end{equation}
by definitions.
The orbits of the above action of $\As(X)$ on $X$ are called {\it connected components of $X$}, denoted by $\mathrm{O}(X).$
If the action of $\As(X)$ on $X$ is transitive, $X$ is said to be {\it connected}.
%For example, it is known \cite[Proposition 1]{LN} that an Alexander quandle $X$ in Example \ref{Alexex} is connected if and only if $(1-T)X=X$.
%Furthermore it can be easily seen that all the quandles in Examples \ref{Sympex} and \ref{sphex} are connected.

Note that the group $\As(X)$ is of infinite order. Actually, there is a splitting epimorphism
\begin{equation}\label{hasz} \varepsilon_X : \As (X) \lra \Z \end{equation}
which sends each generators $e_x $ to $1 \in \Z$. Furthermore, if $X$ is connected,
this $\varepsilon_X$ gives the abelianization $\As (X)_{\mathrm{ab}} \cong \Z$ because of \eqref{hasz2}. The reader should keep in mind this epimorphism $\varepsilon_X$.
%, which is the abelianization of $\As (X)$ [cf.\eqref{epsilon}].
%Note that a homomorphism $X \ra Y$ between quandles induces a group homomorphism $ \As (X) \ra \As(Y)$.

%Furthermore, a quandle $X$ of finite order is called {\it regular}, if
%$X$ is connected, and the order $|X|$ is prime to $ |\mathrm{Inn}(X) |/ |X|.$ For example,
%the symplectic quandles $\mathsf{Sp}_q^g $ over $\F_q$ with $g=1$ are regular.
%In addition, a connected Alexander quandle $X$ with $|X| < \infty$ is regular if and only if the order of $X$ is prime to its type
%, e.g., the Alexander quandle over $ \F_q$.

\subsection{Review; Quandle homotopy invariant of links.}\label{reqhi}
We begin reviewing $X$-colorings.
Let $X$ be a quandle, and $D$ an oriented link diagram of a link $L\subset S^3.$
An $X$-{\it coloring} of $D$
is a map $\CC: \{ \mbox{arcs of $D$} \} \to X$
such that $ \CC(\gamma_k)= \CC(\gamma_i) \lhd \CC(\gamma_j)$ at any crossing of $D$ such as Figure \ref{koutenpn}.
Let $\mathrm{Col}_X(D) $ denote the set of all $X$-colorings of $D$.
As is well known, if two diagrams $D_1$ and $D_2$ are related by Reidemeister moves,
we easily obtain a canonical bijection $\mathrm{Col}_X(D_1) \simeq \mathrm{Col}_X(D_2) $; see, e.g., \cite{Joy,CJKLS}.

\begin{figure}[htpb]
\begin{center}
\begin{picture}(100,60)
\put(-66,50){\large $\gamma_i $}
\put(-13,50){\large $\gamma_j $}
\put(-13,22){\large $\gamma_k $}

\put(-66,33){\pc{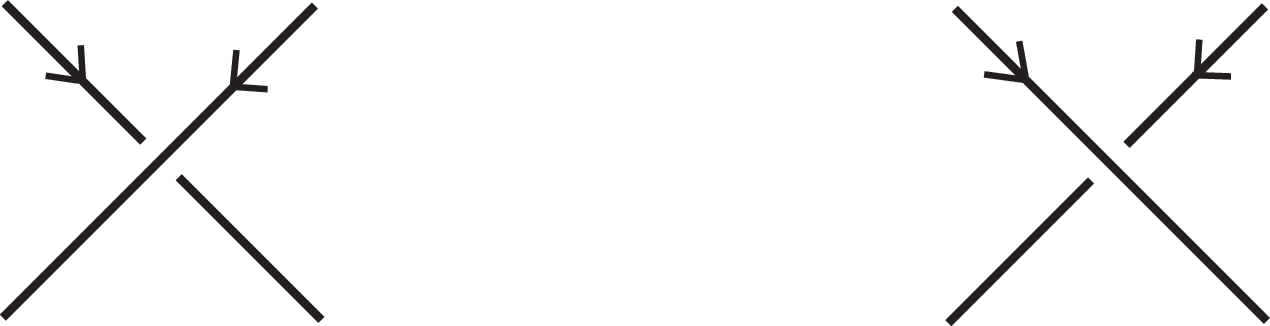}{0.34}}

\put(89,50){\large $\gamma_j $}
\put(141,50){\large $\gamma_k $}
\put(89,22){\large $\gamma_i $}
\end{picture}
\end{center}
\vskip -0.95pc
\vskip -0.95pc
%\vskip -1.7pc
%\vskip -0.5pc
\caption{\label{koutenpn} Positive and negative crossings.}
\end{figure}
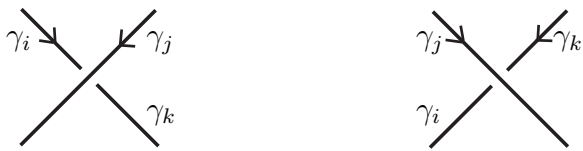

We now study a topological meaning of
an $X$-coloring $\CC $ of $D$.
%, given an $X$-coloring $\CC $ of $D$, we
Let us correspondence each arc $\gamma $ to the generator $\Gamma_{\CC}(\gamma_i):= e_{\CC(\gamma_i)} \in \As (X)$,
which defines a group homomorphism
\begin{equation}\label{Gammacdef} \Gamma_{\mathcal{C}}: \pi_1(S^3 \setminus L)\lra \mathrm{As}(X) \end{equation}
by Wirtinger presentation. The reader should keep this map $\Gamma_{\mathcal{C}} $ in mind, for the sake of later discussion.
Then, we have a map $ \col_X (D) \ra \Hom ( \pi_1(S^3 \setminus L), \mathrm{As}(X))$ which carries $\mathcal{C}$ to $\Gamma_{\mathcal{C}}$, leading to a topological meaning of the set $\mathrm{Col}_X(D) $ as follows:
\begin{prop}[cf. {\cite[Lemma 3.14]{Eis2}} in the knot case]\label{prop.homcol}Let $X$ be a quandle. % and an element $x_0 \in X$.
Let $D$ be a diagram of an oriented link $L$.
We fix a meridian-longitude pair $(\mathfrak{m}_i, \ \mathfrak{l}_i) \in \pi_1(S^3 \setminus L)$ of each link-component which is compatible with the orientation.
Then, the previous map which sends $\mathcal{C} $ to $\Gamma_{\mathcal{C}}$ gives rise to a bijection between the
$\mathrm{Col}_X(D)$ and a set
\begin{equation}\label{1v4}
\{ (x_1, \dots, x_{\# L }, f ) \in X^{\# L} \times \mathrm{Hom}_{\rm gr}(\pi_1(S^3 \setminus L), \As(X) )\ | \ f(\mathfrak{m}_i ) =e_{x_i}, \ \ x_i \cdot f(\mathfrak{l}_i) =x_i \} .
\end{equation}
\end{prop}
\noindent
This proof will appear in Appendix \ref{ds1epr23}, although it is not so difficult.

Next, we briefly recall the quandle homotopy invariant of links (our formula is a modification the formula in \cite{FRS1}).
Let us consider the set, $\Pi_2(X)$, of all $X$-colorings of all diagrams subject to Reidemeister moves and
{\it the concordance relations} illustrated in Figure \ref{pic.21...}.
Then, disjoint unions of $X$-colorings make $\Pi_2(X)$ into an abelian group, which is closely related to a homotopy group $\pi_2(BX)$; see \eqref{bxbx}.
For any link diagram $D$, we have a map $ \Xi_{X,D} : \col_X(D) \ra \Pi_2(X)$ taking $\CC$ to the class $[\CC]$ in $ \Pi_2(X)$.
If $X$ is of finite order and $D$ is a diagram of a link $L$, then the {\it quandle homotopy invariant} of $L$ is defined as the expression
\begin{equation}\label{definv} \Xi_X ( L):=\sum_{\CC \in \col_X(D) } \Xi_{X,D} ( \CC) \in \Z [\Pi_2(X)]. \end{equation}
Moreover, as is known \cite{RS} (see also \cite[\S 2]{Nos1}), the homotopy invariant is universal among all ``the quandle cocycle invariants with local coefficients" (see \cite{CJKLS,CKS,CEGS, Kab} for these definitions).
Hence, to answer what the cocycle invariants are, instead, hereafter we may focus on the study of the homotopy invariant and the abelian group $\Pi_2(X)$ in details.

%Here let us describe the universality when $ X =\DD$ or $X= \D$ and $Y$ is a $X$-set.
%For any quandle homomorphism $f: Q(S^3, L) \ra \DD$,
%we denote by $[f] \in \Pi_2(X)$ the canonical class under the bijection \eqref{bxbxrr}.
%Then any cocycle $\psi: Y\times X \times X \ra A$ and $y \in Y$ admit a homomorphism $H_{\psi} : \Pi_2(X) \ra A $ so that
%\begin{equation}\label{bxbxfd} \langle f^*(\psi), \mu_{Y,y} \rangle = H_{\psi} \bigl( [f] \bigr) \in A, \end{equation}
%(See the equations (6) and (7) in \cite{Nosaka1} for details).

%\vskip -0.7pc
\begin{figure}[htpb]
$$
\begin{picture}(220,66)
\put(-62,378){\pc{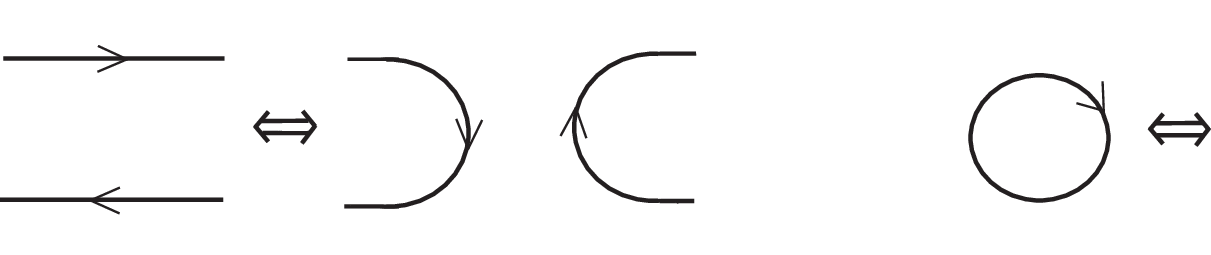}{0.56}}
\put(-31,51){\Large $a$}
\put(-31,-1){\Large $a$}
\put(101,28){\Large $a$}
\put(44.4,28){\Large $a$}
\put(201.2,47){\Large $a$}
\put(275,24){\Huge $\phi$}
\end{picture}$$
%\vskip -0.97pc
\caption{\label{pic.21...}
The concordance relations }
\end{figure}
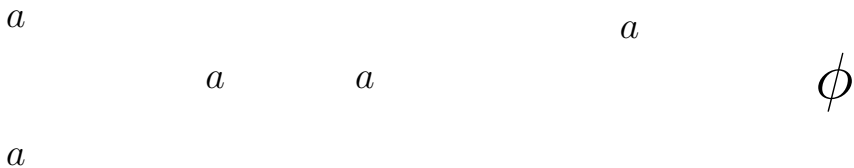

Finally, we briefly review the original quandle cocycle invariant \cite{CJKLS}. Given a finite quandle $X$, we set its quandle homology $H^Q_2 (X)$ with trivial coefficients, which is a quotient
of the free module $\Z \langle X \times X \rangle$; see \S \ref{SS3j3} for the definition.
For an X-coloring $\CC \in \mathrm{Col}_X(D)$, we consider a sum $\sum_{\tau} \epsilon_{\tau} ( \CC(\gamma_i), \CC(\gamma_j) ) \in \Z \langle X \times X \rangle $,
where $\tau$ runs over all crossing
of $D$ as shown in Figure \ref{koutenpn} and the symbol $ \epsilon_{\tau} \in \{ \pm 1 \}$ denotes the sign of the crossing $\tau$.
As is known (see \cite{RS,Nos1}), this sum is a 2-cycle, and homology classes of these sums in $H_2^Q (X)$ are independent of Reidemeister
moves and the concordance relations; Hence we obtain a homomorphism
\begin{equation}\label{h2h2}
\mathcal{H}_X : \Pi_2(X) \lra H^Q_2 (X).
\end{equation}

Using the formula (\ref{definv}), {\it the quandle cycle invariant} of a link $L$, denoted by $\Phi_X(L)$, is then
defined to be the image $\mathcal{H}_X (\Xi_X(L))$ valued in the group ring $\Z[H^Q_2 (X)]$. Namely,
\begin{equation}\label{h2h247}
\Phi_X(L):= \mathcal{H}_X (\Xi_X(L)) \in \Z[H^Q_2 (X)].
\end{equation}
As is known \cite{RS,Nos1,CKS}, given a quandle 2-cocycle $ \phi : X^2 \ra A$, the pairing between this $\phi$ and the cycle invariant
$\Phi_X (L)$ coincides with the original cocycle invariant in [CJKLS, Theorem 4.4].

Although this $\Phi_X(L)$ is constructed from link diagrams,
in \S \ref{SS3j3} we later explain its topological meaning, together with a computation of $H^Q_2 (X)$ following from Eisermann \cite{Eis1,Eis2}.

\section{Results on the quandle homotopy invariants}\label{ss151}
The purpose in this section is to state our results on the group $\Pi_2(X)$.
In \S \ref{s72w}, we will set up a homomorphism $\Theta_X$.
In \S \ref{ager} we state our results on the group $\Pi_2(X) $.
As an application, we will see the computations of third quandle homology groups in \S \ref{s87ab3e}.

\subsection{A key homomorphism $\Theta_{X}$}\label{s72w}
Before describing our main results, we will set up
a homomorphism $\Theta_X$ in Theorem \ref{keykeipp}, which plays a key role in this paper.
Furthermore we discuss Corollary \ref{corthm1a}
which proposes a necessary condition to obtain topological interpretations of
the quandle homotopy invariants. % and, hence, of any quandle cocycle invariants.

To describe the theorem, for any link $L \subset S^3, $
we denote by $p: \widetilde{S^3 \setminus L} \ra S^3 \setminus L$ the $t_X$-fold cyclic covering associated to the homomorphism $ \pi_1( S^3 \setminus L) \ra \Z/t_X $ sending each meridian of $L $ to $1$.
Furthermore, given an $X$-coloring $\CC$
and using the map $ \Gamma_{\CC}$ in \eqref{Gammacdef}, we consider the composite $ p_* \circ \Gamma_{\CC}:\pi_1( \widetilde{S^3 \setminus L}) \ra \As(X)$.
Then
%\begin{equation}\label{000} \mathrm{Proj} \circ \varepsilon_X \circ \Gamma_{\CC }: \pi_1( S^3 \setminus L) \lra \As (X) \lra \Z \lra \Z/t_X. \end{equation}
% given a closed 3-manifold $M $ and a group $G$, we set up a map from the set of all group homomorphisms, $\Hom_{\rm gr}(\pi_1(M),G)$, to the third group homology $H_3^{\rm gr}(G)$ sending
% $f$ to $f_*([M])$. Here $[M] \in H_3(M)$ denotes the fundamental class of $M $.

\begin{thm}\label{keykeipp}Let $X$ be a connected quandle of type $t_X $.
Then, for any diagram $D$ of any link $L \subset S^3$, the composite $ p_* \circ \Gamma_{\CC}:\pi_1( \widetilde{S^3 \setminus L}) \ra \As(X)$ induces a group homomorphism
$$ \theta_{X,D}( \CC ) : \pi_1(\widehat{C}_{L}^{t_X} ) \lra \As(X), $$
where we denote by $\widehat{C}_{L}^{t_X}$ the $t_X $-fold cyclic covering space of $S^3$ branched over the link $L$.

Moreover, consider the pushforward of the fundamental class $[\widehat{C}_{L}^{t_X} ] $ by this $ \theta_{X,D}( \CC ) $,
that is, $ \bigl( \theta_{X,D}( \CC ) \bigr)_* ([\widehat{C}_{L}^{t_X} ]) \in H_3^{\gr}(\As(X))$.
Then, these pushforwards with running over all $X$-colorings of all diagrams $D$ yield an additive homomorphism $\Theta_X : \Pi_2(X) \ra H_3^{\gr}(\As(X))$.
%homomorphism $\Theta_X : \Pi_2(X) \ra H_3^{\gr}( \As(X))$ satisfying the following property:
%admits a map $\theta_{X,D }: \col_X(D) \ra \Hom_{\rm gr}(\pi_1( \widehat{C}_{L}^{t_X}), \As(X))$, where and provides a commutative diagram, functorial in $X$, given by
%Let $G$ denote the kernel $\Ker (\varepsilon_X)$ of the homomorphism in \eqref{hasz}.

\end{thm}
As is seen in the proof in \S \ref{SS3j3outjlw}, we later reconstruct these maps $\theta_{X,D }$ and $\Theta_{X}$ concretely.
By construction, they provide a commutative diagram, functorial in $X$, described as
$$
\xymatrix{ \col_X(D) \ar[rr]^{\!\!\!\!\!\!\!\!\!\!\! \theta_{X,D}} \ar[d]_{\Xi_{X,D}} & &\Hom_{\rm gr} \bigl( \pi_1( \widehat{C}_{L}^{t_X }), \As(X) \bigr) \ar[d]^{(\bullet)_* ([ \widehat{C}_{L}^{t_X} ]) } \\
\Pi_2(X) \ar[rr]^{\!\!\!\!\!\!\!\!\!\!\! \Theta_X } & & H_3^{\gr}( \As(X)) .}
$$
\begin{rem}\label{zdouhenno}
As is seen in \S \ref{SS3j3outjlw}, for any $X$-coloring $\CC \in \mathrm{Col}_X(D)$, the homomorphism
$\theta_{X,D}(\CC ) : \pi_1( \widehat{C}_L^{t_X }) \ra \mathrm{As}(X)$ factors through
the kernel $ \Ker (\varepsilon_X) \subset \As (X) $ in \eqref{hasz},
and
the covering transformation $\Z \curvearrowright \widehat{C}_L^{t_X }$ is $\Z$-equivariant to
the action $\Z \curvearrowright \Ker(\varepsilon_X)$ from the splitting \eqref{hasz}.
%If $ X$ is of finite order, so is the kernel $\Ker(\varepsilon_X)$ [see \cite[Eis3]].
\end{rem}

\noindent
The map plays an important role in this paper. Thus we fix terminologies:
\begin{defn}\label{keykeippdefn}
Let $X$ be a connected quandle of type $t_X$.
We call the map $\Theta_X $ {\it T-map}, and the sum $ \Theta_X \oplus \mathcal{H}_X $ {\it TH-map}.
Here $ \mathcal{H}_X $ is the map $\mathcal{H}_X : \Pi_2(X) \ra H^Q_2 (X)$ defined in \eqref{h2h2}.
\end{defn}

To state Corollary \ref{corthm1a} below,
%For this,
%, we briefly review the Dijkgraaf-Witten invariant \cite{DW}.
%Given a finite group $G$ and a group 3-cocycle $\kappa \in H_{\rm gr}^3(G;A)$, the {\it Dijkgraaf-Witten invariant} of $M $ is defined as a formal sum of some pairings expressed as
%\begin{equation}\label{DWDWder9} \mathrm{ DW}_{\kappa}(M) := \sum_{f \in \Hom(\pi_1(M),G ) } \langle \kappa ,f_*( [M]) \rangle \in \Z [A]. \end{equation}
%%Here $\Z[A]$ is a group ring of $A$.Furthermore,
let us consider the image of the $\theta_{X,D}$ in Theorem \ref{keykeipp}, compared with the Dijkgraaf-Witten invariant \eqref{DWDWder9}.
Strictly speaking, when $X$ is of finite order, we define the formal sum
\begin{equation}\label{DWDWder2} \mathrm{ DW}^{\Z}_{ \As(X)}( \widehat{C}_{L}^{t_X } ) := \sum_{ \mathcal{C} \in \col_X(D)} [\Theta_X ( \Xi_{X,D}(\mathcal{C})) ] = \sum_{\mathcal{C} \in \col_X(D)} \theta_{X,D} (\mathcal{C})_* ([\widehat{C}_{L}^{t_X }]) \in \Z[H_3^{\rm gr}(\As(X))].
\end{equation}
We call it {\it a $\Z$-equivariant part of the
Dijkgraaf-Witten invariant} of branched covering spaces $\widehat{C}_L^{t_X }$.
% as
Then, we obtain the conclude that the quandle homotopy invariant under an assumption is topologically characterized as follows:

\begin{cor}\label{corthm1a}
Let $X$, $\widehat{C}_{L}^{t_X}$, $ \Theta_X $ be as above and let $\ell \in \Z$ be a prime. Let $|X| < \infty$.
% Take the Hurewicz homomorphism $ \mathcal{H}_X : \Pi_2(X) \ra H_2^{Q}(X)$ in \eqref{h2h2}.
%Let $D$ be a diagram of a link $L$, and let $\widehat{C}_{L}^{t}$ be the $t_X $-fold covering of $S^3$ branched over a link $L$.
If the TH-map $\Theta_X \oplus \mathcal{H}_X : \Pi_2(X) \ra H_3^{\gr}( \As(X)) \oplus H_2^{Q}(X)$
is an isomorphism after $\ell$-localization, then the $\ell$-torsion of the quandle homotopy invariant of any link $L$ is
decomposed as
\begin{equation}\label{a}(\Theta_X \oplus \mathcal{H}_X )_{(\ell)} (\Xi_X(L))= \mathrm{ DW}^{\Z}_{ \As(X)}( \widehat{C}_{L}^{t_X } )_{(\ell)} + \Phi_X (L)_{(\ell)} \ \in \Z[\Pi_2(X)_{(\ell )}]. \end{equation}
\end{cor}

In conclusion, under the assumption on the TH-map, we succeed in providing a
topological interpretation of the quandle homotopy invariant $\Xi_X(L)$ as mentioned in the introduction.
Actually, the two invariants in the right hand side of \eqref{a} are topologically defined. % by definition.

\subsection{Results on the TH-maps $ \Theta_X \oplus \mathcal{H}_X $.}\label{ager}
Following Corollary \ref{corthm1a}, it is thus significant to find quandles such that the localized TH-map $(\Theta_X \oplus \mathcal{H}_X)_{(\ell)}$ are isomorphisms.
%This section lists such quandles
%To begin with,
The following main theorem is a general statement. To be specific,
%In summary, we conclude that a certain part of the Dijkgraaf-Witten invariant of $\widehat{C}_{L}^{t}$ is derived from the quandle homotopy invariant $\Xi_X(L)$.

\begin{thm}\label{theorem}
Let $X$ be a connected quandle of type $t_X < \infty $.
If the homology $H_3^{\gr }(\As(X))$ is finitely generated (e.g., if $X$ is of finite order),
then the TH-map is a $[1/t_X]$-isomorphism.
\end{thm}
As a result, combining this theorem with Corollary \ref{corthm1a}, we have obtained the interpretation of some torsion of
the quandle homotopy (cocycle) invariant. Precisely,
\begin{cor}\label{thm1bts}
Let $X$ be a finite connected quandle of type $t_X $.
%Let a prime $\ell$ be relatively prime to the $t_X $.
Then the equality \eqref{a} in Corollary \ref{corthm1a} holds for any prime $\ell $ which does not divide $t_X$.
%: $$ \Xi_X(L)= \mathrm{ DW}^{\Z}_{ \As(X)}( \widehat{C}_{L}^{t} ) + \Phi_X (L) \ \in \Z[\Pi_2(X)_{(\ell )}]. $$
\end{cor}
\noindent
Note that there are many quandles whose types $t_X $ are powers of some prime, e.g., the quandles in Examples \ref{Sympex} and \ref{sphex}.
%, and connected quandles of order $\leq 8$.
In conclusion, for such quandles $X$, we determine most subgroups of $\pi_2(BX)$ by Theorem \ref{theorem}.
\begin{rem}\label{remm1bts}
Corollary \ref{thm1bts} is a strong generalization of
some results in \cite{Kab,HN}.
Indeed, the results dealt with only the Alexander quandle of the from $X=\Z[T^{\pm 1}]/ (2n-1,T+1)$, and
were based on peculiar properties of the quandle.
%Let $X$ be a finite connected quandle of type $t_X $.
%Let a prime $\ell$ be coprime to the $t_X $.% Then the equality \eqref{a} in Corollary \ref{corthm1a} holds.
%: $$ \Xi_X(L)= \mathrm{ DW}^{\Z}_{ \As(X)}( \widehat{C}_{L}^{t} ) + \Phi_X (L) \ \in \Z[\Pi_2(X)_{(\ell )}]. $$
\end{rem}

We moreover address some $t_X $-torsion subgroups of $\Pi_2(X)$.
First we discuss an easy condition of vanishing of these $t_X$-torsions:
\begin{prop}\label{buta}
Let $X$ be a connected quandle of type $t_X$.
If the $t_X$-torsion part of the image $ H_3^{\gr}( \As(X)) \oplus H_2^{Q}(X) $ is zero, then that of the TH-map is also zero.
\end{prop}
%While the easy proof will appear in \S \ref{SS3j3out}, we now observe examples satisfying the assumption.????
% is a general statement and applicable for torsion subgroups.
%Furthermore, for , $\Pi_2(X)$ has no $t_X $-torsion subgroup.

\noindent
As seen in \S \ref{twoto}, there are several such quandles.
In conclusion, %for such quandles,
we obtain a topological meaning of the quandle homotopy
(cocycle) invariants, from the viewpoint of Corollary \ref{corthm1a}.
%, for these quandles, the quandle homotopy invariants are characterized topologically.

\subsection{Application; some computations of third quandle homology}\label{s87ab3e}

\large
\baselineskip=16pt

As an application of computing the (homotopy) group $\Pi_2(X)$,
%From the viewpoint of the (homotopy) group $\Pi_2(X)$,
we develop a new method for computing the third quandle homology $H_3^Q(X) $; see \S \ref{SS3j3} for the definition. %from the homotopy group $\pi_2(BX)$.
%Moreover, we will give explicit computations of $H_3^Q(X)$ of some quandles.
%Statements in this subsection will be proven in \S \ref{s87as}.

To describe our results, we briefly review
{\it the inner automorphism group}, $\mathrm{Inn}(X)$, of a quandle $X$.
Recalling the action of $\As (X)$ on $X$, we thus have a group homomorphism $\psi_X$ from $\As(X)$ to the symmetric group $\mathfrak{S}_{X}$.
The group $\mathrm{Inn}(X) $ is defined to be the image $( \subset \mathfrak{S}_{X} )$.
Hence we have a group extension
\begin{equation}\label{AI} 0 \lra \Ker (\psi_X) \lra \As (X) \xrightarrow{\ \psi_X \ } \mathrm{Inn}(X) \lra 0 \ \ \ \ \ (\mathrm{exact}).
\end{equation}
We should notice that, by the equality \eqref{hasz2}, this kernel $ \Ker (\psi_X)$ is contained in the center.
Further, as was shown \cite{Nos3} (see Theorem \ref{centthm}), if $X$ is of type $t_X $ and connected, then
there is a $[1/t_X]$-isomorphism $\Ker (\psi_X) \cong \Z \oplus H_2^{\rm gr}(\mathrm{Inn}(X))$. %This is easily obtained from the five exact sequence of \eqref{AI} and $H_1^{\rm gr}(\As(X))\cong \Z$.

The following theorem is an estimate on the third homology $ H_3^Q(X)$; the proof will be done in \S \ref{s87as}.
\begin{thm}\label{thm3homology}
Let $X$ be a connected quandle of finite order.
%Let $\mathcal{K} $ be the torsion subgroup of the kernel $\Ker (\psi_X) $.
%Let $ $
Then, there is the following isomorphism up to $2 |\mathrm{Inn}(X)|/|X|$-torsion:
$$H_3^Q(X) \cong H_3^{\rm gr}(\As(X)) \oplus \bigl( \Ker (\psi_X) \wedge \Ker (\psi_X) \bigr). $$
\end{thm}
\noindent Here note from \cite[Lemma 3.7]{Nos3} that the type $t_X $ divides the order $|\mathrm{Inn}(X)|/|X| $.
%; see Lemma \ref{lem11}.
In summary, many torsion subgroups of the third quandle homology are
determined after computing the group homologies of $\As(X)$ and $ \mathrm{Inn}(X)$.
\begin{rem}\label{thm3homologyre}
The isomorphism does not hold in the 2-torsion subgroup.
See Remark \ref{thmqq41} for a counterexample.
Furthermore, as mentioned in the introduction, the most known results on the third $H_3^Q(X)$
are with respect to the Alexander quandles over $\mathbb{F}_q$ and are due to Mochizuki \cite{Moc2}.
So this theorem is expected to be applicable to other quandles.
%has computed the third quandle cohomology with $\F_q$-coefficients, and gave a polynomial-presentation of its basis.
%However, his presentation is a little complicated (see polynomials ``$\Gamma$" in \cite[\S 2.2]{Moc2}).
%Theorem \ref{Ale3homology} implies that the reason is derived from the third group homology of the nilpotent group $\As(X)$.
\end{rem}

%Similarly, using the diagram chasing and the facts of the homologies of the orthogonal groups,
%we can show the third homology of the spherical quandles.

%The proof will appear in \S \ref{118}.
In a subsequent paper \cite{Nos4}, this theorem on the quandle $\X$ will be used to understand the T-map $\Theta_X$
from a viewpoint of complexes of groups and quandles.

\section{The homomorphism $\Theta_X $}\label{SS3j3outjlw}

\baselineskip=16pt

From now on, we will prove the results mentioned in the previous section.

Our purpose in this section is to prove Theorem \ref{keykeipp} and, is to construct a homomorphism from $\Pi_2(X)$ to a bordism group (Lemma \ref{key231}),
which plays a key role in this paper.
The construction is a modification of a certain map in \cite[\S 4,5]{HN}.
%, where we dealt with only a class of ``4-fold symmetric quandles".

For the purpose, we first describe a presentation of the fundamental group $\pi_1(\widehat{C}_{L}^{t}) $,
where $\widehat{C}_{L}^{t}$ denotes the $t $-fold cyclic covering of $S^3$ branched along a link $L$.
Put a link diagram $D$ of $L$.
Let $\gamma_0, \dots, \gamma_n $ be the arcs of this $D$.
Let $ \widetilde{S^3 \setminus L}$ be the $t$-fold cyclic covering space of $ S^3 \setminus L$ associated to
the homomorphism $ \pi_1( S^3 \setminus L)\ra \Z /t $ which sends each $\gamma_i$ to $1$.
For any index $s \in \Z/t $, we take a copy $\gamma_{i, s}$ of the arc $\gamma_i$.
Then, by Reidemeister-Schreier method (see, e.g., \cite[Appendix A]{Rolfsen} and \cite[\S 3]{Kab}), the fundamental group $\pi_1( \widetilde{S^3 \setminus L})$ can be presented by

$$
\begin{array}{ll}
\mathrm{generators:} &\quad \ \gamma_{i,s} \ \ \ (0 \leq i \leq n, \ s \in \Z), \\
&\quad \\
\mathrm{relations:} &\quad \ \gamma_{k,s}= \gamma_{j,s-1}^{-1} \gamma_{i,s-1}\gamma_{j,s} \ \mathrm{for \ each \ crossings \ such \ as \ Figure\ \ref{koutenpn} \ and }, \\
&\quad \ \gamma_{0,0}= \gamma_{0,1}= \cdots = \gamma_{0,t-2}= 1 .
\end{array}$$
Further we can define the inclusion $p_* : \pi_1(\widetilde{S^3 \setminus L}) \hookrightarrow \pi_1(S^3 \setminus L) $ by $\iota (\gamma_{i,s})= \gamma_0^{s-1} \gamma_i \gamma_0^{-s}$,
with a choice of an appropriate base point.
Moreover, the fundamental group $\pi_1(\widehat{C}_{L}^{t})$ is obtained from this presentation by adding the relation $ \gamma_{0,t-1}= 1 $.

Next, given a quandle $X$ of type $t $, we now construct a map \eqref{sanatana} below. %For this end, we note the following lemma, which is often used later.
%Note that the $t_X $-power of the generator $e_x$ lies in the center of $\As(X)$.
%\begin{lem}{[Lemma ??]}\label{daiji} Let $X$ be a connected quandle of type $t $.
%Then, for any $x,y \in X$, we have the identity $ (e_x)^{t}=(e_y)^{t}$ in the center of $\As(X)$. \end{lem}
%\begin{proof}For any $b \in X$, note the equalities $(e_x)^{-t} e_b e_x^{t} = e_{(\cdots (b \lhd x ) \cdots )\lhd x} = e_b $ in $\As(X)$.
%Namely $ (e_x)^{t}$ lies in the center. Furthermore the connectivity admits
%$ g \in \As(X)$ such that $ x \cdot g = y$.
%Hence it follows from \eqref{hasz2} that $(e_x)^{t} = g^{-1} (e_x)^{t} g = (e_{ x \cdot g} ) ^{t}= (e_y)^{t} $ as desired.
%\end{proof}
%\noindent
By the above presentation of $\pi_1(\widehat{C}_{L}^{t}) $, the map $\Gamma_{\mathcal{C}}$
induces a homomorphism $ \widehat{\Gamma}_{\mathcal{C}} : \pi_1(\widehat{C}_{L}^{t})\ra \Ker (\varepsilon_X )$,
% from the set of $X$-colorings to that of some group homomorphisms below \eqref{sanatana}. % (The map \ref{sanatana} was constructed in \cite
where $\varepsilon_X : \As (X) \ra \Z$ is the homomorphism which sends the generators $e_x $ to $1 \in \Z$ [Recall \eqref{hasz}]. %, which is the abelianization of $\As (X)$ [cf.\eqref{epsilon}].
%Then the restriction of the homomorphism $\Gamma_{\mathcal{C}} $ on $\pi_1(\widetilde{S^3 \setminus L})$ factors through the kernel $ \Ker (\varepsilon_X )$.
Precisely,
this $\widehat{\Gamma}_{\mathcal{C}}$ is defined by the formula
$\widehat{\Gamma}_{\mathcal{C}}(\gamma_{i,s})= e_{\CC(\gamma_0)}^{s-1}e_{\CC(\gamma_i)}e_{ \CC(\gamma_0)}^{-s} $ (we use \cite[Lemma 3.5]{Nos3} for the well-definedness).
In summary, we obtain the map
stated in Theorem \ref{keykeipp}:
\begin{equation}\label{sanatana}\theta_{X,D} : \ \mathrm{Col}_{X}(D) \lra
\Hom_{\mathrm{gr}} ( \pi_1(\widehat{C}_{L}^{t}), \Ker (\varepsilon_X ) ), \ \ \ \ \ (\mathcal{C} \longmapsto \widehat{\Gamma}_{\mathcal{C}}). \end{equation}
We here remark that this map depends on the choice of the arc $\gamma_0$; however it does not up to conjugacy of $ \Ker (\varepsilon_X ) $ by construction,
if $X$ is connected.

Finally, in order to state Lemma \ref{key231} below, we briefly recall the {\it oriented bordism group}, $\Omega_n(G)$, of
a group $G$.
Let us consider a pair consisting of a closed connected oriented $n$-manifold $M$ without boundary and a homomorphism $\pi_1(M) \ra G$.
Then the set, $\Omega_n(G)$, is defined to be the quotient set of such pairs ($M $, $\pi_1(M) \ra G$) subject to $G$-bordant equivalence.
Here, such two pairs ($M_i$, $f_i: \pi_1(M_i) \ra G$) are $G$-{\it bordant}, if
there exist an oriented $(n+1)$-manifold $W$ which bounds the connected sum $M_1 \# (- M_2)$ and a homomorphism $\bar{f}:\pi_1 (W) \ra G$ so that
$ f_1 * f_2 = \bar{f} \circ (i_W)_* $, where $i_W: M_1 \# (- M_2) \ra W $ is the natural inclusion.
An abelian group structure is imposed on $\Omega_n(G)$
by connected sum. Note that this $\Omega_n(G)$ agrees with the usual oriented ($SO$-)bordism group of the Eilenberg-MacLane space $K(G,1)$.
%It is well-known that $ \Omega_3(G)$ is isomorphic to the third group homology $H_3^{\rm gr}(G)$.

\begin{lem}\label{key231}
Let $X$ be a connected quandle of type $t $.
Then, by considering all link diagrams $D$, the maps $\theta_{X,D}$ in \eqref{sanatana} give rise to an additive homomorphism
\begin{equation}\label{sanatana2}\Theta_{\Pi \Omega}:\Pi_2(X) \lra \Omega_3( \mathrm{Ker}(\varepsilon_X )).\end{equation} \end{lem}
\begin{proof}[A sketch of the proof]
Since the proof is analogous to \cite[Lemma 5.3 and Proposition 4.3]{HN} essentially,
we will sketch it.
%We consider the maps in \eqref{sanatana} for all link diagrams $D$.
To obtain the homomorphism $\Theta_{\Pi \Omega}$, it suffices to show that the maps take the concordance relations to the bordance ones.

First, to deal with the local move in the right of Figure \ref{pic.21...},
we recall that the $t $-fold cyclic covering of $S^3$ branched over the 2-component trivial link $T_2$ is $S^2 \times S^1 \ra S^3$ (see \cite[\S 10.C]{Rolfsen}).
%By applying $D_2 = T_2$ to Lemma \ref{thmesqq} below,
It suffices to show that any $f: \pi_1(S^2 \times S^1) \ra \mathrm{Ker}(\varepsilon_X ) $ is $G$-bordant.
Indeed, $ f: \pi_1(B^3 \times S^1) \ra \mathrm{Ker}(\varepsilon_X ) $ provides its bordance, where $B^3$ is a ball.

Next, %concerning another local move,
for two $X$-colorings $\CC_1$ and $\CC_2$
% of $L_1$ and $L_2$, respectively. We will show that, if these $\CC_1$ and $\CC_2$ are
related by the left in Figure \ref{saddle6},
we will show that the connected sum $\theta_{X,D} (\CC_1 \# (- \CC_2)^*): \pi_1( \widehat{C}_{L_1 }^{t}\# \widehat{C}_{L_2 }^{t} ) \ra \Ker (\varepsilon_X) $ is null-bordance.
Let $N_{\mathcal{C}_i} \subset S^3$ be a neighborhood around the local move.
Then we put a canonical saddle $\mathcal{F}$ in $ N_{\CC_1} \times [0,1]$ which bounds the four arcs illustrated in Figure \ref{saddle6}.
Define an embedded surface $W \subset S^3 \times [0,1]$ to be $\bigl( (L_1 \setminus N_{\CC_1}) \times [0,1] \bigr) \cup \mathcal{F}$.
%Note the homomorphism $\pi_1 (S^3 \times [0,1] \setminus W) \cong \pi_1(S^3 \setminus L_1 )/ \langle \mathfrak{m}_{\alpha}= \mathfrak{m}_{\beta} \rangle$, where $\mathfrak{m}_{\alpha}$ and $ \mathfrak{m}_{\beta}$ are the meridians of the two arcs in $N_{\CC_1}$.
%Hence the homomorphism $\pi_1(S^3 \setminus L_i ) \ra \Z/m$ gives rise to $\pi_1 (S^3 \times [0,1] \setminus W) \ra \Z / m $.
Then the $t $-fold cyclic covering $\mathcal{W} \ra S^3 \times [0,1]$ branched over $W$
bounds $\widehat{C}_{L_1 }^{t} \sqcup \widehat{C}_{L_2}^{t}$.
Moreover, we can verify that the sum $\theta_{X,D} (\CC_1 \# (- \CC_2)^*)$ extends to a group homomorphism $ \pi_1 (\mathcal{W}) \ra \Ker (\varepsilon_X)$, which gives the desired null-bordance. \end{proof}

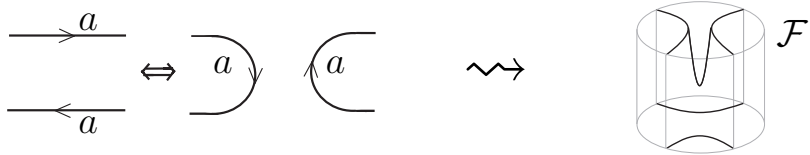
\begin{figure}[htpb]
$$
\begin{picture}(20,41)

\put(126,31){\Large $\mathcal{F}$ }

\put(-170,17){\pc{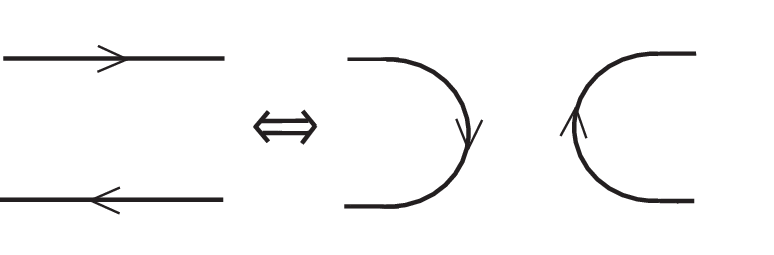}{0.4156}}
\put(-138,37){\Large $a$}
\put(-138,-2){\Large $a$}
\put(-87,21){\Large $a$}
\put(-44.4,21){\Large $a$}

\put(7.4,16){\Huge $\rightsquigarrow $ }

%\put(-173,20){\Large $\mathcal{W}:=$ $ \left\{
%\begin{array}{ll}
%\textrm{the } m\textrm{-fold cyclic covering } &\quad \\
%\textrm{of \ }S^3 \times[0,1] \textrm{ branched along } &\quad
%\end{array}
%\right.$}

%\put(-143,20){\large \ of $S^3\times [0,1]$ branched over}

\put(68,15){\pc{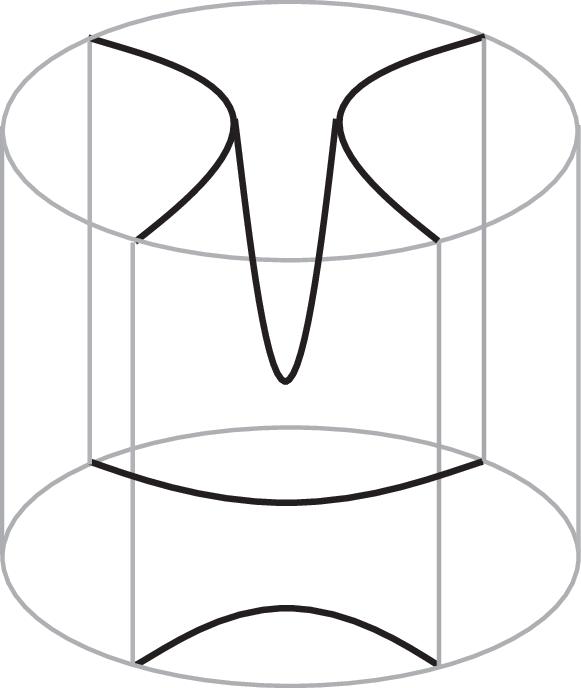}{0.174}}
%\put(140,30){\pc{pic.12.d}{0.45}}
\end{picture}
$$
\caption{\label{saddle6} $\mathcal{F}$ is a saddle in the neighborhood $N_{C_i} \times [0,1]$.}
\end{figure}

%where we use Lemma \ref{thmesqq} below and the presentations of $\pi_1(\widetilde{S^3 \setminus L_i})$ mentioned above. This completes the proof.
%In conclusion, we therefore obtain a bordance of $\theta_{X,D} (\CC_1 \# (- \CC_2)^*)$ as desired.

%\begin{lem}\label{thmesqq}
%Let $X$ be as above.
%Put an $X$-coloring $\CC_i \in \col_X(D_i)$ of a link $L_i$ $(i=1,2)$.
%Let us correspond $\CC_i$ to $f_i :\pi_1( \widehat{C}_{L_i}^{t}) \ra \Ker (\varepsilon_X) $ via the map $\theta_{X,D}$.
%If the arc $\gamma_0$ is colored by the same by $\CC_1$ and $\CC_2$, then the connected sum $\CC_1 \# \CC_2$ along the arc $\gamma_0$ corresponds to the free product $f_1 *f_2 :\pi_1( \widehat{C}_{L_1}^{t} \# \widehat{C}_{L_2}^{t}) \ra \Ker (\varepsilon_X) $.\end{lem}

%\begin{proof} This is almost the same as the proof of \cite[Proposition 4.3]{HN}; so we omit the details. \end{proof}

%\begin{lem}[{\cite[Lemma ]{Nosaka3}}]\label{yokato}
%Let $X$ be a connected quandle of type $t_X $. Then, for any $a,b \in X$, an equality $(e_a)^{t}= (e_b)^{t} \in \As (X)$
%holds in the center in $\As (X) $. \end{lem}

Finally, in order to prove Theorem \ref{keykeipp}, let us review {\it Thom homomorphism} $\mathcal{T}_G: \Omega_n(G) \ra H_n(K(G,1))= H_n^{\mathrm{gr}}(G)$ obtained by assigning to every pair $( M, f: \pi_1(M)\ra G )$ the image of the orientation class
under $f_* : H_n(M) \ra H_n(K(G,1))$. It is widely known that, if $n=3$, the map $\mathcal{T}_G $ is an isomorphism $\Omega_3(G) \cong H_3^{\rm gr} (G)$.
\begin{proof}[Proof of Theorem \ref{keykeipp}]
Let $G$ be $\As(X)$. Put the inclusion $\iota : \mathrm{Ker}(\varepsilon_X) \hookrightarrow \As(X) $.
Consequently, defining the T-map $\Theta_X$ to be the composite $ \mathcal{T}_{\As(X)} \circ \iota_* \circ \Theta_{\Pi \Omega}: \Pi_2(X) \ra H_3^{\mathrm{gr}}(\As(X))$,
we can see that this $\Theta_X$ satisfies the desired properties.
%which is the desired.
%From the construction from the maps $\theta_{X,D}$, the required commutative diagram holds.
% From the definitions of $\mathcal{T}_G $ and $ $, we immediately see
\end{proof}

\section{Preliminaries; quandle homology and cocycle invariant.}\label{SS3j3}
\large
\baselineskip=16pt
%By the exact sequence \eqref{kihon}, the homology of a space $BX$ plays a key role.
As a preliminary, we now review some properties of the (co)homology of the rack space,
and a topological interpretation of the cocycle invariants.
There is nothing new in this section.

We start by reviewing the (action) rack space introduced by Fenn-Rourke-Sanderson \cite[Example 3.1.1]{FRS1}.
Let $X$ be a quandle. We further fix a set $Y$ acted on by $\As(X)$, which is called {\it $X$-set}. For example, the quandle $X$ is itself
an $X$-set, referred as to {\it the primitive $X$-set}, from the canonical action $X \curvearrowleft \As(X)$ mentioned in \S \ref{ss1}.
We further equip a quandle $X$ and the $X$-set $Y$ with their discrete topology.
Let us put a union $ \bigcup_{n \geq 0} \bigl( Y \times ([0,1]\times X)^n \bigr)$, and consider
the relations given by

\vskip -0.994pc
\normalsize
\[(y,t_1,x_1, \dots, x_{j-1},1,x_j, t_{j+1}, \dots, t_n,x_n) \sim (y \cdot e_{x_{j}} , t_1,x_1\tri x_{j}, \dots, t_{j-1}, x_{j-1}\tri x_{j},t_{j+1},x_{j+1},\dots, t_n,x_n), \]
\begin{equation}\label{equ.relation} \normalsize \! (y, t_1,x_1, \dots, x_{j-1},0,x_j,t_{j+1},\dots, t_n,x_n) \sim (y, t_1,x_1, \dots t_{j-1},x_{j-1},t_{j-1},x_{j+1},\dots, t_n, x_{n} ). \notag \end{equation}

\large
\baselineskip=16pt
\noindent
Then the {\it rack space} $B(X,Y)$ is defined to be the quotient space.
When $Y$ is a single point, we denote the space by $BX$ for short.
By construction, we have a cell decomposition of $B(X,Y)$ by
regarding the projection $ \bigcup_{n \geq 0} \bigl( Y \times ([0,1]\times X)^n \bigr) \ra B(X,Y)$ as characteristic maps.
% (see \cite[\S 2]{Nos1} or \cite[\S 2.2]{Nos2} for a detailed picture of the 3-skeleton of $BX$).

Furthermore, we briefly review the rack and quandle (co)homologies (our formula relies on \cite{CEGS}).
Let $X$ be a quandle, and $Y$ an $X$-set.
Let $C_n^R(X,Y)$ be the free right $\Z$-module generated by $Y \times X^n$.
Define a boundary $\partial^R_n : C_n^R(X,Y) \rightarrow C_{n-1}^R(X,Y )$ by
$$ \partial^R_n (y, x_1, \dots,x_n)= \sum_{1\leq i \leq n} (-1)^i\bigl( (y \tri x_i, x_1\tri x_i,\dots,x_{i-1}\tri x_i,x_{i+1},\dots,x_n) -( y,x_1, \dots,x_{i-1},x_{i+1},\dots,x_n) \bigr).$$
The composite $\partial_{n-1}^R \circ \partial_n^R $ is known to be zero. The homology is denoted by $H^R_n(X,Y)$ and is called {\it rack homology}.
As is known, the cellular complex of the rack space $B(X,Y)$ above is isomorphic to the complex $(C_* ^R(X,Y),\partial_*^R )$.

%We can see that the complex $(C_* ^R(X,Y),\partial_* )$ is known to be chain isomorphic to the cellular complex of the rack space $B(X,Y)$.

\begin{rem}\label{hayano}
If $Y$ is the primitive $X$-set $Y=X$,
we have the chain isomorphism $C^R_n (X,X ) \ra C^R_{n+1} (X,{\rm pt} ) $ induced from the identification $X \times X^n \simeq X^{n+1}$; see, e.g., \cite[Proposition 2.1]{Cla}.
In particular, we obtain an isomorphism $H^R_n (X,X) \cong H_{n+1}^R (X, {\rm pt})$. \end{rem}

\noindent
Furthermore, let $C^D_n (X,Y) $ be a submodule of $C^R_n (X,Y)$ generated by $(n+1)$-tuples $(y,x_1, \dots,x_n)$
with $x_i = x_{i+1}$ for some $ i \in \{1, \dots, n-1\}$. It can be easily seen that the submodule $C^D_n (X,Y) $ is a subcomplex of $ C^R_{n} (X,Y). $
Then the {\it quandle homology}, $H^Q_n (X,Y ) $, is defined to be the homology of the quotient complex $C^R_n (X,Y ) /C^D_n (X,Y)$.
%Dually, we can define the cohomology.

We will review some properties of these homologies in the case where $Y$ is a single point (In such a case, we omit the symbol $Y$).
%The quandle homology coincides with the original one in \cite{CJKLS}.
Let us decompose $X $ as $X= \sqcup_{i \in \OX}X_i$ by the connected components. %under the action of $\As (X)$ in Example \ref{reidw}.
The following isomorphisms were shown \cite{LN}:
\begin{equation}
\label{eq.H2RQ} \hspace*{5pc}
H^R_1 (X) \cong \Z^{\OX }, \qquad
H^R_2 (X) \cong H_2^Q (X) \oplus \Z^{\OX }.
\end{equation}

Furthermore, Eisermann \cite{Eis2} gave a computation of the second quandle homologies $ H_2^Q (X)$ with trivial $\Z$-coefficients.
To see this, for $i \in \OX $, define a homomorphism
\begin{equation}\label{epsilon} \displaystyle{ \varepsilon_i : \As (X) \ra \Z } \ \ \ \ \textrm{ by } \left\{ \begin{array}{ll}
\displaystyle{ \varepsilon_i ( e_x)=1 \in \Z }, &\ \mathrm{if} \ \ x \in X_i, \\
\displaystyle{ \varepsilon_i ( e_x)=0\in \Z }, &\ \mathrm{if} \ \ x \in X \setminus X_i.\\
\end{array} \right. \ \end{equation}
Note that the sum $\oplus_{i \in \OX} \varepsilon_i$ yields the abelianization $\As(X)_{\rm ab} \cong \Z^{\OX}$ by \eqref{hasz2}. Furthermore
\begin{thm}[{\cite[Theorem 9.9]{Eis2}}]\label{ehyouj2i3}
Let $X $ be a quandle. Decompose $X= \sqcup_{i \in \OX }X_i$ as the orbits by the action of $\As(X)$.
Fix an element $x_i \in X_i$ for each $i \in \OX$. Let $ \mathrm{Stab}(x_i) \subset \As(X) $ be the stabilizer of $x_i$.
Then the quandle homology $H_2^Q (X )$ is isomorphic to the direct sum of the abelianizations of $\mathrm{Stab}(x_i) \cap \Ker (\varepsilon_i) $: Precisely,
$\bigoplus_{i \in \OX} \bigl( \mathrm{Stab}(x_i) \cap \Ker (\varepsilon_i) \bigr)_{\mathrm{ ab}}$.
\end{thm}
\noindent
Eisermann showed topologically this result by using a certain CW-complex.
However, in \S \ref{kokoroni} we later give another proof as
a slight application of Proposition \ref{propFRS}.% shown \cite{FRS1,FRS2}.
% and a result of Proposition \ref{ehyouji}.

%Incidentally, the author gave another simple proof in \cite[\S 4.1]{Nosaka3}.

Furthermore, we now turn into the study of the cycle invariant $\Phi_X(L)$ mentioned in \eqref{h2h247}, and
briefly explain a topological interpretation of this invariant shown by Eisermann \cite{Eis1,Eis2}. % following Theorem \ref{ehyouj2i3}.
Decompose $X= \sqcup_{i \in \OX}X_i $ as above.
Given an $X$-coloring $\mathcal{C} \in \col_X(D)$,
with respect to a link component of $L$, we fix an arc $\gamma_{j}$ on $D$ for $1 \leq j \leq \# L$.
Let $x_j:=\mathcal{C}(\gamma_j) \in X_j$, and fix a longitude $ \mathfrak{l}_j$ of the component.
Recall from \eqref{Gammacdef} the associated group homomorphism $\Gamma_{\mathcal{C}} : \pi_1(S^3 \setminus L) \ra \mathrm{As}(X)$.
Remark that each longitude $ \mathfrak{l}_j$ commutes with the meridian in the same link component.
Accordingly, $\Gamma_{\mathcal{C}}( \mathfrak{l}_j)$ commutes with $e_{ x_j } $ in $\As(X)$: in other wards, $\Gamma_{\mathcal{C}}( \mathfrak{l}_j) \in \mathrm{Stab}(x_j)$.
Furthermore, since the class of the longitude $ \mathfrak{l}_j $ in $H_1(S^3 \setminus L )$ is zero, $\Gamma_{\mathcal{C}}( \mathfrak{l}_j)$ is contained in the kernel $ \Ker (\varepsilon_j )$ [see \eqref{epsilon}].
Therefore, $ \Gamma_{\mathcal{C}}( \mathfrak{l}_j)$ lies in $\mathrm{Stab}(x_j) \cap \Ker (\varepsilon_j ) $.
Further, consider the class $\Gamma_{\mathcal{C}}( \mathfrak{l}_j)$ in the abelianization of this $\mathrm{Stab}(x_j) \cap \Ker (\varepsilon_j ) $.
In summary, we obtain
\begin{equation}\label{kantanji} \bigl( [\Gamma_{\mathcal{C}}( \mathfrak{l}_1)], \dots, [\Gamma_{\mathcal{C}} ( \mathfrak{l}_{\# L})] \bigr) \in \bigoplus_{1 \leq j \leq \# L} \bigl( \mathrm{Stab}(x_j) \cap \Ker (\varepsilon_j) \bigr)_{\mathrm{ ab} }. \end{equation}
Here note from Theorem \ref{ehyouj2i3} that each direct summand in the right side is contained in $H_2^Q(X) $.
We then put the product $ [\Gamma_{\mathcal{C}}( \mathfrak{l}_1) \cdots \Gamma_{\mathcal{C}} ( \mathfrak{l}_{\# L})] \in H_2^Q(X)$.
By the discussion in \cite[Theorems 3.24 and 3.25]{Eis1}, it can be seen that the product coincides with the
value $\mathcal{H}_X (\mathcal{C})$ in \eqref{h2h2} exactly.
Hence, when $|X| < \infty $, the cycle invariant $\Phi_X(L)$ written in \eqref{h2h247} is reformulated as
\begin{equation}\label{kantan} \Phi_{X}(L)= \sum_{ \mathcal{C} \in \col_X(D)} [ \Gamma_{\mathcal{C}}( \mathfrak{l}_1) \cdots \Gamma_{\mathcal{C}}( \mathfrak{l}_{\# L}) ] \in \Z[ H_2^Q(X) ]. \end{equation}
This $\Phi_{X}(L) $ was called ``colouring polynomials" in \cite[\S 1]{Eis1}. As a result, this
formula suggests an easy computation and a topological meaning of the cycle invariant as desired.

Finally, we observe a relation between this formula \eqref{eq.H2RQ} and the Hurewicz homomorphism of $BX$.
Recall the map $\mathcal{H}_X:\Pi_2(X) \ra H_2^Q(X)$ in (\ref{h2h2}). Using the isomorphisms (\ref{eq.H2RQ}) and (\ref{bxbx}), we consider a composite
$$\pi_2(BX) \cong \Pi_2(X) \oplus \Z^{\mathrm{O}(X)} \xrightarrow{\mathrm{proj}.} \Pi_2(X) \xrightarrow{\mathcal{H}_X} H_2^Q(X) \hookrightarrow H_2^Q(X) \oplus \Z^{\mathrm{O}(X)} \cong H_2(BX).$$
From the definitions of the maps $\mathcal{H}_X$ and the 2-skeleton of the rack space $BX$, we can easily verify that this
composite coincides with the Hurewicz map of $BX$ modulo the direct summand $\Z^{\mathrm{O}(X)}$ (see \cite{RS} and \cite[Proposition 3.12]{Nos1} for details). In conclusion, the formula (\ref{kantan}) enables us to compute the Hurewicz map of
$BX$.

\section{Proof of Theorem \ref{theorem}}\label{ss2se}
This section proves Theorem \ref{theorem}.
Since the proof is ad hoc, the hasty reader may read only the outline in \S \ref{SS3j3out} and skip the details in other subsections.
Following the outline, Section \ref{vanishes} states a $t_X $-vanishing theorem and some properties of quandle coverings.
%since they play a key role in the proof.
In Section \ref{SS4j3out}, we will investigate the homomorphism $\Theta_X$ in terms of relative bordism groups,
and complete the proof.
We will fix notation throughout this section.

\vskip 0.4pc
\noindent
{\bf Notation.} Let $X$ be a connected quandle of type $t_X < \infty$ (possibly, $X$ is of infinite order).
%, and let $p:\X \ra X$ be the universal covering of the $X$.
Furthermore, we write a symbol $\ell $ for a prime which is relatively prime to the integer $t_X $.
\vskip 0.4pc

\subsection{Outline of proofs of Theorem \ref{theorem}}\label{SS3j3out}

\baselineskip=16pt
We roughly outline the proof of Theorem \ref{theorem} to compute $\Pi_2(X) $.

As an approach to the homotopy group $\pi_2 (BX ) $,
the reader should keep in mind the following isomorphism shown by \cite{FRS2} (see also \cite[Theorem 6.2]{Nos1} for the detailed description):
\begin{equation}\label{bxbx}\pi_2 (BX ) \cong \Pi_2(X) \oplus \Z^{\oplus \mathrm{O}(X)}, \end{equation}
where the symbol $\mathrm{O}(X) $ is the set of the connected components of $X$.
According to the isomorphism \eqref{bxbx}, to compute $\Pi_2(X)$, we will change a focus on computing the homotopy group $\pi_2(BX)$ from a standard discussion of ``Postnikov tower on $BX$".
%In general, it is difficult to compute homotopy groups.
To illustrate, let $c : BX \hookrightarrow K (\pi_1(BX),1) $ be an inclusion obtained by killing the higher homotopy groups of $BX$. Notice that the homotopy fiber of $c$ is the universal covering of $BX$.
Thanks to the fact \cite{FRS2} that the action of $\pi_1(BX)$ on $\pi_2 (BX)$ is trivial (see also \cite[Proposition 2.16]{Cla}),
the Leray-Serre spectral sequence of the map $c$ provides an exact sequence
\begin{equation}\label{kihon2} H_3(BX )\stackrel{c_*}{\lra} H_3^{\rm gr}(\pi_1(BX) ) \stackrel{\tau}{\lra} \pi_2(BX ) \stackrel{\mathcal{H}}{\lra} H_2(BX ) \stackrel{c_*}{\lra} H^{\rm gr}_2(\pi_1(BX)) \ra 0 \ \ \ \ \ \ ({\rm exact}), \end{equation}
\noindent where $\mathcal{H}$ is the Hurewicz map of $BX $ and the $\tau$ is the transgression (see, e.g., \cite[\S $8.3^{bis}$]{McC}, \cite[\S II.5]{Bro} for details).
%,and $\tau$ denotes the transgressive map.

We now reduce this \eqref{kihon2} to \eqref{kihon3} below.
Recalling the isomorphism $H_2(BX) \cong \Z^{\mathrm{O}(X)} \oplus H_2^Q(X)$ (see \eqref{eq.H2RQ}), %, where the summand $H_2^Q(X)$ is the quandle homology explained in \S \ref{SS3j3}.
the restriction of the Hurewicz map $\mathcal{H}$ on the summand $\Z^{\mathrm{O}(X)} \subset \pi_2 (BX) $
is shown to be an isomorphism \cite[Proposition 3.12]{Nos1}.
Therefore, recalling the isomorphism $\As(X) \cong \pi_1(BX)$, % and $\pi_2 (BX ) \cong \Pi_2(X) \oplus \Z^{\mathrm{O}(X)}$ in \eqref{kihon2},
the sequence \eqref{kihon2} is reformulated as
\begin{equation}\label{kihon3} H_3(BX )\stackrel{c_*}{\lra} H_3^{\rm gr}(\As (X) ) \stackrel{\tau}{\lra} \Pi_2(X ) \stackrel{\mathcal{H}}{\lra} H_2^Q(X ) \stackrel{c_*}{\lra} H^{\rm gr}_2(\As (X)) \ra 0 \ \ \ \ \ \ ({\rm exact}). \end{equation}
Since this paper often uses this sequence, we call it {\it the $P$-sequence (of} $X$).
%In conclusion, by computing the three homologies, we obtain an upper bound of $\Pi_2(BX)$.

Using the $P$-sequence, we outline the proof of Theorem \ref{theorem}.
Let $X$ be connected and of type $t_X < \infty $.
We later see Theorem \ref{centthm} which says that the maps
$c_*: H_n (BX ) \ra H_n^{\rm gr}(\pi_1(BX) ) $ in \eqref{kihon2} are annihilated by $t_X $ for $n \leq 3$.
Thus, the $P$-sequence \eqref{kihon3} becomes a short exact sequence up to $t_X$-torsion (Corollary \ref{cccor2}).
Hence, in order to show that the TH-map $ \Pi_2(X) \ra H_3^{\rm gr}(\As (X) ) \oplus H_2^Q(X ) $ is a [$1/t_X$]-isomorphism as stated in Theorem \ref{theorem},
we shall show that the T-map $\Theta_X: \Pi_2(X) \ra H_3^{\rm gr}(\As (X) )$ constructed in Theorem \ref{keykeipp}
turns out to be a splitting of the exact sequence \eqref{kihon3}.

To this end, we first show the splitting with respect to the extended quandles $\X $ (Proposition \ref{centthm4o1}).
So we will review properties of the $\X$ in \S \ref{vanishes}.
The point is that, using these properties, the transgression $\tau $ in \eqref{kihon3} can be regarded as
an inverse mapping of the T-map $\Theta_{\X} $ in a (relative) bordism theory.
After that, returning connected quandles $X$,
the functoriality of the projection $\X \ra X$ completes the proof of Theorem \ref{theorem}.

Before going to the next subsection, we now immediately prove Proposition \ref{buta}:
\begin{proof}[Proof of Proposition \ref{buta}]
Since the $t_X$-torsion of $ H_3^{\gr}( \As(X)) \oplus H_2^{Q}(X) $ is zero by assumption,
that of $\Pi_2(X)$ vanishes because of \eqref{kihon3}. Hence, that of the TH-map $ \Theta_X \oplus \mathcal{H}_X$ is zero as desired.
\end{proof}
%\section{Proof of Theorem \ref{thm1b}}\label{Sproof}

\subsection{The $t_X $-vanishing of the map $c_*$, and the quandle covering.}\label{vanishes}
Following the preceding outline, we will review the results in \cite{Nos3}: First,
% Using the notation, our objectivity in this subsection is to show the following theorem:
\begin{thm}[\cite{Nos3}]\label{centthm}
Let $X$ be a connected quandle of type $t_X $, and let $t_X < \infty $.
For $n = 2$ and $3$, the induced map $c_*: H_n(BX ) \ra H_n^{\rm gr} (\As(X)) $ in \eqref{kihon2} is
annihilated by $t_X $.

Furthermore, the second group homology $ H_2^{\rm gr} (\As(X)) $ is annihilated by $t_X $.
\end{thm}
This theorem yields two corollaries which are useful for the $P$-sequences as follows.
\begin{cor}\label{cccor2}
Let $X$ be as above,
and $\ell $ be a prime which is relatively prime to the $t_X $.
%be a connected quandle of type $t_X < \infty$. . Let $\ell$ be a prime coprime to $t_X $.
Then the $P$-sequence localized at $\ell$ is reduced to a short exact sequence
$$ 0\lra H_3^{\rm gr}(\As(X))_{(\ell)} \lra \pi_2(BX)_{(\ell)} \xrightarrow{\ \mathcal{H}_{X} \ } H_2(BX)_{(\ell)}\lra 0 \ \ \ \ (\mathrm{exact}). $$
\end{cor}
\begin{cor}\label{cccor3}
Let $X$ and $\ell \in \Z $ be as above. Let $X$ be of finite order. Then the quandle cycle invariant $\Phi_X$ in \eqref{h2h247}
is non-trivial in the $\ell$-torsion part. That is, for any class $ [ O ] \in H_2(BX)_{(\ell)}$,
there exists some $X$-coloring $\mathcal{C}$ of some link such that $\mathcal{H}_X ([\mathcal{C}]) =[O]$.
\end{cor}
\begin{proof}By Corollary \ref{cccor2}, the map $\mathcal{H}_X$ localized at $\ell$ is surjective.
Since the $\Pi_2(X)$ is generated by $X$-colorings of links by definition,
we have $\mathcal{H}_X ([\mathcal{C}]) =[O]$ for some $X$-coloring $\mathcal{C}$. %, the map $c_2: H_2(BX ) \ra H_2^{\rm gr} (\As(X)) $ is a surjection.
\end{proof}
\begin{rem}\label{cccl}
In general, we can similarly see that, for a quandle $X$ with $H_2^{\gr} (\As(X))=0$, % (e.g., case of $H_1^{\gr} (\mathrm{Inn}(X))=0 $; see Proposition \ref{vanithm2}),
any class $ [ O ] \in H_2(BX)$ ensures some $X$-coloring $\mathcal{C}$ such that $\mathcal{H}_X ([\mathcal{C}]) =[O]$.
% in the $\ell$-torsion part. That is, for any class $ [ O ] \in H_2(BX)_{(\ell)}$,
%there exists some $X$-coloring $\mathcal{C}$ of a link such that $\mathcal{H}_X ([\mathcal{C}]) =[O]$.
\end{rem}

Changing the subject, we will mention extended quandles considered in \cite[\S 7]{Joy}.
%, which plays a role to prove Theorem \ref{theorem}.
Recall the epimorphism $\varepsilon_X :\As (X) \ra \Z$ in \eqref{hasz}.
Given a connected quandle $X$ with $a \in X$, we equip the kernel $\Ker (\varepsilon_X) $ with
a quandle operation by setting
$$ g \lhd h := e_a^{-1} g h^{-1} e_a h \ \ \ \ \ \ \ \mathrm{for} \ g,h \in \Ker(\varepsilon_X) . $$
Let us denote the quandle $(\Ker (\varepsilon_X ), \lhd)$ by $\X$, and call {\it the extended quandle (of $X$)}.
We easily see the independence of the choice of $a \in X$ up to quandle isomorphisms.
Using the restricted action $ X \curvearrowleft \Ker (\varepsilon_X) \subset \As (X)$,
the canonical map $p: \X \ra X$ sending $g$ to $a \cdot g$ is known to be
a quandle homomorphism (see \cite[Theorem 7.1]{Joy}), and is called {\it the universal (quandle) covering of }$X$.
%, according to Eisermann \cite[\S 5]{Eis2}. % called `universal covering of $X$'. .
%We can easily see that, when $X$ is finite and of type $t_X $, so is the extended one $\X$.

The previous paper \cite{Nos3} showed their interesting properties, which will be used later.
% in the next section.
\begin{thm}[{\cite[\S 5]{Nos3}}]\label{extthm}
Let $X $ be a connected quandle of type $t_X$, and let $p_* : \As (\X ) \ra \As (X) $ be the epimorphism induced from the covering $p : \X \ra X $.
Fix the identity element $ 1_{\X} \in \Ker (\varepsilon_X ) = \X $.
Then,
\begin{enumerate}[(i)]
\item The quandle $\X$ is connected and is of type $t_X$. If $X$ is of finite order, then so is $\X$.
\item Under the canonical action of $\As (\X)$ on $\X$, the stabilizer $\mathrm{Stab}(1_{\X} )$ of $1_{\X} $ is equal to
$ \Z \times \Ker (p_*)$ in $\As (\X)$. %Concisely $ \mathrm{Stab}(e_a) = \Z \times \Ker (p_*) \subset \As (\X)$.
Furthermore, the summand $\Z$ is generated by $1_{\X} $.
\item The second quandle homology of $\X$ is isomorphic to the abelian kernel of the projection $p_* : \As (\X ) \ra \As (X) $. Namely $H_2^Q(\X) \cong \Ker (p_*)$.
\item The homology $H_2^Q(\X) \cong \Ker (p_*)$ is annihilated by the type $t_X $.
\item The above map $p_* $ induces a $[1/t_X]$-isomorphism $H_3^{\gr} (\As(\X)) \cong H_3^{\gr} (\As(X)) $.
\end{enumerate}
\end{thm}
%\begin{lem}\label{lem11}
%Let $X $ be a finite connected quandle of type $t_X $.
%Then the abelianiztion of $\mathrm{Inn}(X)$ is a quotient of $\Z/m$.
%In particular, when $X $ is regular, the abelianization localized at $\ee_all$ is zaro.
%\end{lem}
%\begin{proof}
%Consider the five terms of the sequence \eqref{aiseq},
%\begin{equation}\label{k1kkl2} H_2(\mathrm{As}(X)) \lra H_2(\mathrm{Inn}(X)) \lra \Ker(\psi_X)\lra H_1(\mathrm{As}(X)) \lra H_1(\mathrm{Inn}(X))\lra 0.
%\end{equation}
%By the map \eqref{}, $ H_1(\mathrm{As}(X)) \cong \Z $ is generated by $e_x$ for some $x \in X$.
%However, the $t_X $-times on the right with $x$ is the identity in $\mathrm{Inn}(X)$.
%Thereby $ H_1(\mathrm{Inn}(X)) $ is a quotient of $\Z/m$. \end{proof}

\subsection{The T-map $\Theta_X $ as a splitting}\label{SS4j3out}
We first prove Theorem \ref{theorem} by using the following key proposition.
\begin{prop}\label{centthm4o1}
Let $p: \X \ra X$ be the universal covering.
If the homology $H_3^{\gr }(\As(X))$ is finitely generated, then the T-map $\Theta_{\X} : \Pi_2(\X )\ra \Omega_{3}(\As(\X))$ in Theorem \ref{keykeipp}
%a $[1/t_X]$-isomorphism.
% modulo $t_X $-torsion and
is a $[1/t_X]$-splitting in the short exact sequence in Corollary \ref{cccor2}.
%For $n = 2$ or $3$, the induced map $c_*: H_n(BX ) \ra H_n^{\rm gr} (\As(X)) $ is annihilated by $t_X $.
\end{prop}
\begin{proof}[Proof of Theorem \ref{theorem}]
Take the $P$-sequences associated to the covering $p: \X \ra X$:
$$
\xymatrix{ H_3^{\gr }(\As(\X))_{(\ell)} \ar[rr]^{\tilde{\tau}_*} \ar[d]_{p_*} & & \Pi_2 (\X)_{(\ell)} \ar[rr] \ar[d]_{p_*} & & H_2^Q(\X)_{(\ell)}\ar[d]_{p_*} \ar[rr] & & H_2^{\gr}( \As(\X))_{(\ell)} = 0 \ar[d]_{p_*} \\
H_3^{\gr}( \As(X))_{(\ell)} \ar[rr]^{\tau_*} & & \Pi_2(X)_{(\ell)} \ar[rr] & & H_2^Q(X)_{(\ell)} \ar[rr] & & H_2 ^{\gr}( \As(X))_{(\ell)}=0 .}
$$
%Since the left maps $c_*$ are annililated by $t_X $ (Theorem \ref{}),
%the delta $\tilde{\delta}$ is an isomorphism localized at $\ell$.
Since the left $p_* $ between group homologies is a $[1/t_X]$-isomorphism (see Theorem \ref{extthm} (iv) and (v)),
%by the Lyndon-Hochschild sequence of $p$. %, which is finitely generated because of the finiteness of $X$.
%Furthermore, note the vanishing $ H_2^Q(\X)_{(\ell)} \cong 0 $ by Corollary \ref{lemwl}.
the TH-map $\Theta_{X} \oplus \mathcal{H}_{X}: \Pi_2(X)_{(\ell)} \ra H_3^{\gr}( \As(X))_{(\ell)} \oplus H_2^Q(X)_{(\ell)}$ is an isomorphism by the functoriality of $\Theta_X$ and Proposition \ref{centthm4o1}.
\end{proof}
%Proposition \ref{centthm4} gives a splitting $ \Pi_2(X )\ra \Omega_{3}(\As(X)) $ in the short exact sequence in Corollary \ref{cccor2}.
%Hence, we have an isomorphism $ \Pi_2(X )\cong H_2^Q(X) \oplus \Omega_{3}(\As(X)) $ modulo $t_X $-torsion.
\noindent
Thus, we shall aim
%Our goal in this subsection is
to prove Proposition \ref{centthm4o1} with respect to
%and hence quandles in the proof are
the extended quandles $\X$.

% There are some studies of a relation between third group homology and second homotopy groups (see, e.g., \cite{Sie,FRS2}).
For the proof, we will review a classical bordism theory (see \cite[\S 1.4]{CF}). Given a space-pair $(Y,A)$ with $A \subset Y$,
consider a continuous map
$$ f: (M, \partial M ) \lra (Y ,A),$$
where $M $ is an oriented compact $n$-manifold. Such two maps $f_1, \ f_2$ are {\it $G$-bordant},
if there exist an oriented compact manifold $W$ of dimension $n+1$ and a map $F:W \ra Y $ for which
\begin{enumerate}[(I)]
\item There is an $n$-dimensional submanifold $M' \subset \partial W$ satisfying $F( \overline{\partial W \setminus M'}) \subset A. $
\item There is a diffeomorphism $ g : (-M_1) \cup M_2 \ra M' $
preserving orientation such that $ (-f_1) \cup f_2 = (F | M') \circ g$.
\end{enumerate}
Then the {\it bordism group} of $(Y,A)$, denoted by $\Omega_n(Y,A)$, is defined to be
the set of all such map $f$ subject to the $G$-bordant relations.
We make this $\Omega_n(Y,A)$ into an abelian group by disjoint union.
Furthermore, $\{ \Omega_n(Y,A) \}_{n \geq 0} $ gives a homology theory (see \cite[\S 1.6]{CF}),
and the isomorphism $\Omega_n(Y,A) \cong H_n(Y,A) $, for $n \leq 3$, is obtained by the Atiyah-Hirzebruch spectral sequence.
Furthermore, if $Y$ is the Eilenberg-MacLane space $K(G,1)$ and $A$ is the empty set, then $ \Omega_n(Y,A)$ is
isomorphic to the group $\Omega_n(G)$ introduced in \S \ref{SS3j3outjlw}.

% Furthermore, we prepare an easy lemma. Fix $a \in X$ and consider the two monomorphism $i_X : \Z \ra \As(X)$ and $i'_X: \Z/m \ra G_X $ sending $n$ to $e_a^n$.
% \begin{lem}\label{centthm4lem} Then the relative $0$-th homologies are zero. In other ward,
% $$ H_0( K(\As(X),1), K(\Z,1))=0 , \ \ \ \ \ H_0( K(G_X,1 ), K(\Z/m,1))=0 . $$ \end{lem}
% \begin{proof} $H_1^{\gr}(\As(X)) = \Z$ is generated by $e_a$; the $i_X$ induces an isomorphism $H_1^{\gr}(\Z) \cong H_1^{\gr}(\As(X))$.
% Hence the relative 0-homology is zero. Similarly, the latter part on $G_X$ can be shown. \end{proof}

Using the bordism groups, we will construct a homomorphism \eqref{mpru44} below. % cribe Lemma \ref{mpru} below.
%Put the covering $p : \X \ra X$.h
%ruct a homomorphism $ \Upsilon: \Pi_2(X) \ra \Omega( K(G_X,1 ), K( \Z/m,1) )$.
Given an $\X$-coloring $\CC$ with $\# L$ link components, we take $t_X $-copies of $\CC$, and denote them by $\CC_j $ for $1 \leq j \leq m$.
Let us fix an arc of each link component of $\CC$, and consider a connected sum of these $\CC_1 , \dots, \CC_m$ (see Figure \ref{S24}).
Denote the resulting link by $\overline{L}$ and the associated $\X $-coloring of $\overline{L} $ by $\overline{\CC}$.
We then set a homomorphism $\Gamma_{\overline{\CC}} : \pi_1(S^3 \setminus \overline{L})\ra \As(\X) $ discussed in \eqref{Gammacdef}.
Note that each meridian of $\overline{L} $ is sent to be $e_g $ for some $g \in \X $ by definition.
Furthermore each longitudes $\mathfrak{l}_j \in \pi_1(S^3 \setminus \overline{L} )$ are sent to be zero.
Actually the formula \eqref{kantanji} says that the $ \Gamma_{\overline{\CC}} (\mathfrak{l}_j) $ lies in $\Ker (\varepsilon_{\X}) \cap \mathrm{Stab}(1_{\X})$,
which is equal to the abelian kernel $\Ker (p_*)$ and is annihilate by $t_X $ (see Theorem \ref{extthm}).
Consequently, the map $\Gamma_{\overline{\CC}}$ sends every boundaries of $S^3 \setminus \overline{L} $ to a 1-cell of $B\X$.
Here note that the 1-skeleton $B \X_1$ is, by definition, a bouquet of circles labeled by elements of $\X$.
In the sequel, considering all $\X$-coloring $\overline{\CC} $ and such homomorphisms $\Gamma_{\overline{\CC}}$ modulo the bordance relations,
the map $ \mathcal{C}\mapsto \Gamma_{\overline{\CC}} $ defines the desired homomorphism
\begin{equation}\label{mpru44}\Upsilon_{\X}: \Pi_2(\X ) \lra \Omega_3 ( K(\As (\X) ,1 ), B\X_1 ).
\end{equation}
Hereafter, we denote by $\Omega_3^{\rm rel}(\X)$ this relative bordism $ \Omega_3( K(\As(\X),1) , B\X_1)$, for simplicity.

\begin{figure}[htpb]
$$
\begin{picture}(20,80)
\put(-220,25){\pc{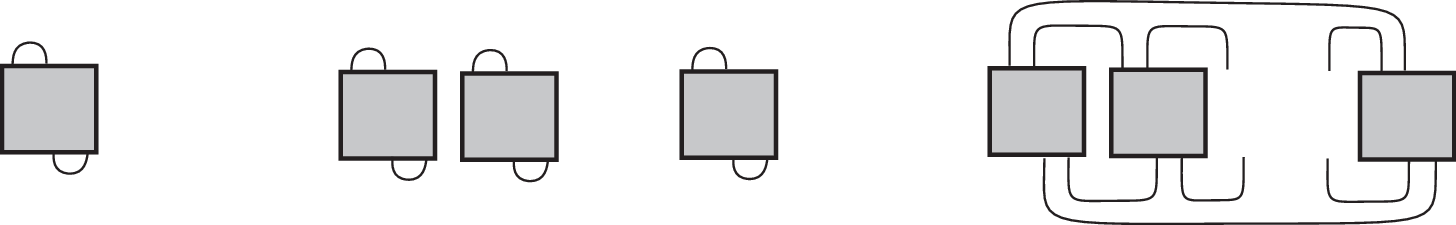}{0.655}}

\put(-158,25){\huge $ \rightsquigarrow $}

\put(-203,50){ $ \gamma_0 $}
\put(-185,10){$ \gamma_1 $}

\put(-96,48){ $ \gamma_0 $}
\put(-78,8){$ \gamma_1 $}

\put(-57,48){ $ \gamma_0 $}
\put(-39,8){$ \gamma_1 $}

\put(13,48){ $ \gamma_0 $}
\put(31,8){$ \gamma_1 $}

\put(-203,28){\Large $ \CC $}

\put(-103,25){\Large $ \CC_1 $}
\put(-63,25){\Large $ \CC_2 $}

\put(8,25){\Large $ \CC_m $}

\put(89,49){\Large $ \overline{\CC} $}

\put(58,25){\huge $ \rightsquigarrow $}

\put(-29,26){\large $\cdots \cdot $}

\put(176, 26){\large $\cdots \cdot \cdot \cdot $}
\put(179, 46){\large $\cdots \cdot $}
\put(185, 6){\large $\cdots$}

\end{picture}
$$
\caption{\label{S24} Construction of $\overline{\CC}$ from $\CC$, when the link components of $\CC$ are two. }
\end{figure}

We now prove Proposition \ref{centthm4o1} by using the following lemma:
\begin{lem}\label{mpru}
%For any connected quandle $X$ of type $t_X $,
The homomorphism $ \Upsilon_{\X }: \Pi_2(\X) \ra \Omega_3^{\rm rel}(\X)$ is surjective up to $t_X $-torsion.
\end{lem}
\begin{proof}[Proof of Proposition \ref{centthm4o1}]
We first explain the following diagram:
$$
\xymatrix{\ar[r]^{\!\!\!\!\!\!\!\!\!\!\!\!\!\!\!\!\!\!\!\!\!\!\!\! 0} & \Omega_3( \As(\X))_{(\ell)} \ar@{^{(}-_{>}}[rr]^{ \!\!\!\!\!\!\! \delta_*} \ar@{=}[d] & & \ \Omega_3^{\rm rel} (\X)_{(\ell)} \ar[r] & \Omega_2( B\X_1 )_{(\ell)}=0 \\
\ar[r]^{\!\!\!\!\!\!\!\!\!\!\!\!\!\!\!\!\!\!\!\!\!\!\!\! 0} & H_3( \As(\X) )_{(\ell)} \ar[rr]^{ \!\!\!\!\!\!\! \widetilde{\tau}_* } & & \Pi_2(\X)_{(\ell)} \ar[u]_{\Upsilon_{\X}} \ar[r] & H_2^Q ( \X )_{(\ell)} =0
}$$
Here the upper sequence is derived by the homology theory $\Omega_n $ with considering the pair $B \X_1 \hookrightarrow K( \As (\X ),1) $, and
the bottom one is obtained from the $P$-sequence of $\X$ with Theorems \ref{centthm} and \ref{extthm}.
Since $ \Pi_2(\X)$ is $[1/t_X]$-isomorphic to the finitely generated module $\Omega_3( \As (\X) )$ by assumption, the localized map of $ \Upsilon_{\X }$ is an isomorphism by Lemma \ref{mpru}.

Therefore, to accomplish the proof, it is sufficient to show the equality
\begin{equation}\label{9999} t_X \cdot \bigl( \delta_* \circ \Theta_{\X}([\mathcal{C}])\bigr) =t_X \cdot \Upsilon_{\X }([\mathcal{C}]) \in \Omega_3^{\rm rel} (\X), \end{equation}
for any $\X$-coloring $\mathcal{C}$. For this, put the resulting link $\overline{L} $ and coloring $\overline{\mathcal{C}} $ explained above.
Furthermore, take the $t_X$-fold cyclic covering $p: C_{\overline{L}}^{t_X} \ra S^3 \setminus \overline{L}$, and consider the natural inclusion $i_{ \mathcal{C}}: C_{\overline{L}}^{t_X} \subset \widehat{C}_{\overline{L}}^{t_X} $ by gluing the 2-handles along the boundary tori.
Here notice that the composite $ \theta_{\X,D} (\overline{\mathcal{C}}) \circ (i_{ \mathcal{C}})_* : \pi_1 (C_{\overline{L}}^{t_X} )\ra \As(X) $ coincides with $ \Gamma_{\overline{\mathcal{C}}} \circ p_* $ by the definition \eqref{sanatana}.
Furthermore notice that the inclusion $i_{ \mathcal{C}}$ gives a bordance relation between the $ \theta_{\X,D} (\overline{\mathcal{C}}) $ and this composite $ \theta_{\X,D} (\overline{\mathcal{C}}) \circ (i_{ \mathcal{C}})_* $.
Since the above map $\delta_*$ comes from the correspondences with $(M,f)$ to $(M,f)$ itself by definition, we thus have
$$ t_X \cdot \delta_* \circ \Theta_{\X}([\mathcal{C}]) = \delta_* \circ \Theta_{\X}([\overline{\mathcal{C}}]) = [ \theta_{\X,D} (\overline{\mathcal{C}})]= [\theta_{\X,D} (\overline{\mathcal{C}}) \circ (i_{ \mathcal{C}})_*]=[ \Gamma_{\overline{\mathcal{C}}} \circ p_*] \in \Omega_3^{\rm rel} (\X), $$
where the first equality is derived from $ t_X[\mathcal{C}]= [\overline{\mathcal{C}} ]\in \Pi_2(\X)$ from the definition of $ \overline{L}$.
We notice $ [ \Gamma_{\overline{\mathcal{C}}} \circ p_*]= t_X [ \Gamma_{\overline{\mathcal{C}} } ]\in \Omega_3^{\rm rel} (\X) $
since the projection $p$ takes the (relative) fundamental class of $C_{\overline{L}}^{t_X}$ to the $t_X$-multiple of that of $S^3 \setminus \overline{L}$.
Hence,
since $\Upsilon_{\X }([\mathcal{C}]) = \Gamma_{\overline{\cal C }}$ by definition,
we have the desired equality \eqref{9999}.
%Recall from the previous proof that the $ \Omega_3^{\rm rel}(\X)_{(\ell)} $ is generated by classes from 3-manifolds with torus boundaries.
%Hence, we can see that the map $ \Theta_{\X} : \Pi_2(\X) \ra \Omega_3( \As(\X))_{(\ell)} $ exnteds to a homomorphism $\overline{\Theta}_{\X } : \Omega_3^{\rm rel}(\X)_{(\ell)} \ra \Omega_3( \As(\X))_{(\ell)} $.
%Since the trivial $t_X$-fold cyclic covering space of any closed 3-manifold $N$ is $t_X$-copies of $N $,
%the composite $\overline{\Theta}_{\X} \circ \delta_* $ is $t_X \mathrm{id}$.
%Therefore, together with the above diagram, the $ \Theta_{\X}$ is an isomorphism and a splitting of $\widetilde{\tau}_* $ modulo $t_X $-torsion as desired.
\end{proof}
To conclude this section, we will work out the proof of Lemma \ref{mpru}.

\begin{proof}[Proof of Lemma \ref{mpru}] %[Proof of Lemma \ref{mpru}]
To begin, we claim that the $\Z_{(\ell)}$-module $\Omega_3^{\rm rel}(\X)_{(\ell)}$ is generated by bordism classes represented by (normal) 3-submanifolds in $S^3$ with torus boundary components.

For this purpose, we first set up the isomorphism \eqref{e94} below.
Here refer to the fact \cite[\S 2.5]{Cla} that the universal covering of $B\X$ is a topological monoid; hence,
it is a based loop space of some space $\mathcal{L}_X$. We therefore have two homotopy fibrations
$$ \Omega \mathcal{L}_X \lra B\X \stackrel{c_*}{\lra} K(\As(\X),1) , \ \ \ \ \ \ B\X \stackrel{c_*}{\lra} K(\As(\X),1) \stackrel{\mathcal{P}_L}{\lra} \mathcal{L}_X .$$
From the right map $\mathcal{P}_L$, we obtain an isomorphism after localization at $\ell$:
\begin{equation}\label{e94} (\mathcal{P}_L)_*: \Omega_3(B\X, K ( \As(\X),1))_{(\ell)} \cong \Omega_3 (\mathcal{L}_X)_{(\ell)}. \end{equation}
However, since the $\mathcal{L}_X$ is 2-connected by definition, the Hurewicz theorem $\pi_3(\mathcal{L}_X) \cong_{[1/t_X] } \Omega_3 (\mathcal{L}_X) $ implies that
this $ \Omega_3 (\mathcal{L}_X)$ is generated by maps $S^3 \ra \mathcal{L}_X$.
Noticing that the map $ (\mathcal{P}_L)_*$ can be regarded as a map coming from collapse of each boundaries of manifolds,
this isomorphism \eqref{e94} implies that generators of the $ \Omega_3(K ( \As(\X),1) , B\X )_{(\ell)}$ are derived from 3-submanifolds in $S^3$.

Next, so as to verify the claim above, consider the inclusions
$$S^1 \subset B\X_1 \subset B\X \stackrel{c}{ \longhookrightarrow }K(\As(\X),1), $$
where the first is obtained by taking the circle labeled by $a \in \X $.
Notice from Theorem \ref{centthm} that these inclusions $S^1 \subset B\X_1 \subset K ( \As(\X),1)$ induce isomorphisms
$$H_2 (S^1)_{(\ell)} \cong H_2( B\X_1 )_{(\ell)} \cong H_2^{\gr }( \As(\X) )_{(\ell)} \cong 0.$$ %with $i= 2$,
Therefore, they yield isomorphisms
$$ \Omega_3( K( \As(\X),1), S^1 )_{(\ell)} \cong \Omega_3^{\rm rel}(\X)_{(\ell)} \cong \Omega_3(K ( \As(\X),1), B\X)_{(\ell)}. $$
Here note that, since the last term is generated by classes from 3-manifolds in $S^3$ as observed above and
$\pi_1(S^1) \cong \Z $ is abelian, the first term is
generated by classes from 3-manifolds in $S^3$ with torus boundaries\footnote{To be precise,
since the first homology of any closed surface is
generated by homology classes from some tori,
given a submanifold $M \subset S^3$ with $f: \pi_1(M) \ra \As(X) $ such that $f(\pi_1(\partial M)) \subset \Z$,
we can obtain another $M' \subset S^3$ with torus boundaries by attaching some 2-handles to $M$, and the maps $f$ extend to $\bar{f}: \pi_1(M') \ra \As(X) $ such that $\bar{f}(\pi_1(\partial M')) \subset \Z$.}.
Hence, so is the $\Omega_3^{\rm rel}(\X)_{(\ell)}$ as claimed.

Finally, to show the required surjectivity of $\Upsilon_{\X}$, we will prove that any generator $O$ of $ \Omega_3^{\rm rel}(\X)_{(\ell)} $
comes from some $\X$-coloring via the bijection \eqref{1v4}. By the previous claim, $ t_X ^{-2}\cdot O$ is represented by
a homomorphism $ f: \pi_1(S^3 \setminus L ) \ra \As (\X)$ for some link $L \subset S^3$.
Furthermore put the resulting link $\overline{L}$ in constructing $\Upsilon_{\X}$ in \eqref{mpru44}.
By repeating the process, we have another $\overline{\overline{L}} $.
Then the $f$ extends to two maps $\overline{f}: \pi_1(S^3 \setminus \overline{L} ) \ra \As (\X) $ and $\overline{\overline{f}}: \pi_1(S^3 \setminus \overline{\overline{L}} ) \ra \As (\X) $ canonically, where
because the class of the latter in $ \Omega_3^{\rm rel}(\X)_{(\ell)} $ equals the generator $O$.
Notice that, for each link component $1 \leq i \leq \# L$, with a choice of meridian element $\mathfrak{m}_i$,
the $ \overline{f} (\mathfrak{m}_i ) \in \As (\X)$ is conjugate to $e_a^{n_i}$ for some $n_i \in \mathbb{N}$, from the definition of $ \Omega_3^{\rm rel}(\X)_{(\ell)} $: % (see Figure \ref{S241}).
Namely, in Wirtinger presentation of $\pi_1( S^3 \setminus \overline{L}) $, each arc $\alpha $ is labeled by $ e_{y_{\alpha }}^{n_i} $ for some $y_{\alpha } \in \X$; see \eqref{hasz2}.
Accordingly, replacing the $i$-th component of the link $\overline{L} $ by $n_i$-parallel copies of the component, we have another link $\overline{L}_{\rm P } $
and, then, can construct a canonical homomorphism $ \overline{f}_{\rm P}: \pi_1(S^3 \setminus \overline{L}_{\rm P}) \ra \As (\X ) $ by which each meridian of $ \overline{L}_{\rm P} $ is sent to $e_y$ for some $y \in \X$ (see Figure \ref{S241}).
%Furthermore, we notice that $ \alpha': \pi_1(S^3 \setminus L') \ra G_{\X} $ send each longitudes to zero from the construction of $\overline{L} $ above (see \ref{} again).
% (see \eqref{} and $\mathrm{Stab}(x_0) \cong \Ker(p_*)$ as an elementary abelian $t_X $-group).
We remark that, this $\overline{f}_{\rm P} $ sends the associated longitude $\mathfrak{l}_i $ of $\overline{L}_{\rm P}$ to $e_y^{t_X n_y}$ for some $ n_y \in \Z $, by the reason similar to the construction of $\Upsilon_{\X} $ in \eqref{mpru44}.
%Indeed, by \eqref{}, this $\overline{f}_{\rm P} (\mathfrak{l}_k )$ is contained in $\mathrm{Stab}(e_z)= \Z \times \Ker (p_*)$, and hence is , since $\Ker (p_*) $ is annihilated by $t_X$.
In particular, we have $ y \cdot \overline{f}_{\rm P} (\mathfrak{l}_i ) = y \cdot e_y^{t_ X n_y} = y \in \X$.
Hence, the bijection \eqref{1v4} admits an $\X$-coloring $ \mathcal{C}_f$ such that $ \Gamma_{ \mathcal{C}_f }= \overline{f}_{\rm P}: \pi_1(S^3 \setminus \overline{L}_{\rm P}) \ra \As (\X ) $.
Consequently, we have the equality $\Upsilon_{\X}([ \mathcal{C}_f ]) = O \in \Omega_3^{\rm rel}(\X)_{(\ell)}$ by construction, which implies the surjectivity.
\end{proof}

\vskip -1.137pc
\begin{figure}[htpb]
$$
\begin{picture}(20,80)
\put(-120,25){\pc{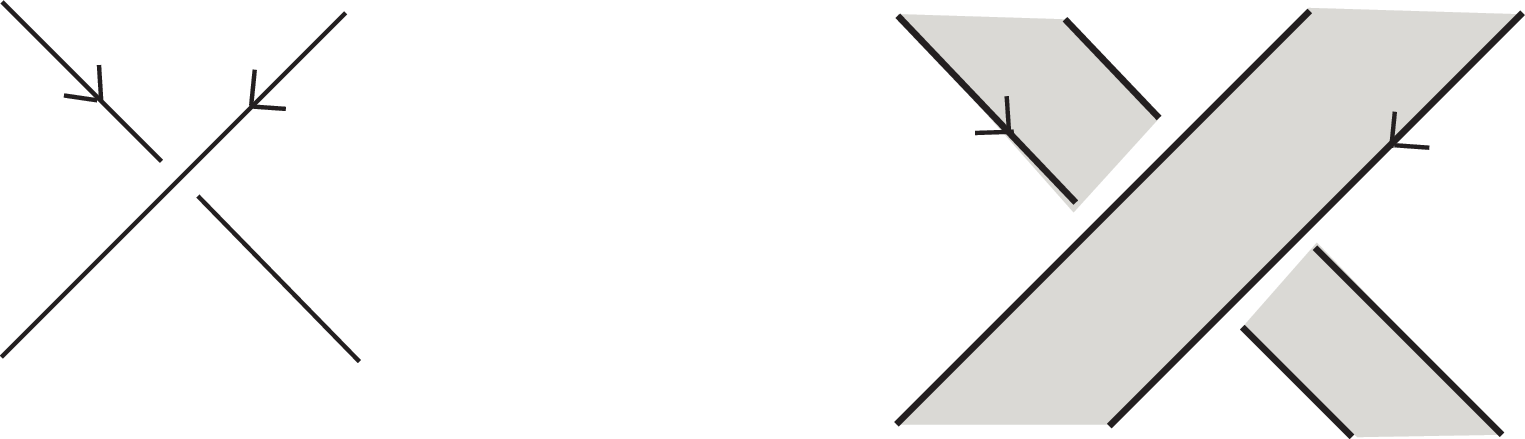}{0.305}}

\put(-125,45){\large $ e_a^{n_i} $}

\put(-67,45){\large $ e_b^{n_j} $}
\put(-67,18){\large $ e_{a \lhd^{n_j} b } ^{n_i} $}

\put(15,45){\large $ e_a $}
\put(99.7,45){\large $ e_b $}
\put(99.7 ,8){\large $ e_{a \lhd^{n_j} b } $}

\put(-29,26){\huge $ \rightsquigarrow $}

\put(70.16,61){\Large $ \overbrace{ \ \ \ \ \ \ \ }^{n_j \textrm{-strands} } $}

\put(7,61){\Large $ \overbrace{ \ \ \ \ \ }^{n_i \textrm{-strands} } $}

% \put(56, 26){\large $\cdot \cdot \cdot $}
% \put(29, 46){\large $\cdots $}
% \put(75, 6){\large $\cdots$}

\end{picture}
$$
\vskip -0.637pc
\caption{\label{S241}Construction from the $ \overline{f} : \pi_1(S^3 \setminus \overline{ L} ) \ra \As (\X) $ to $\overline{f}_{\rm P} : \pi_1(S^3 \setminus \overline{L}_{\rm P} ) \ra \As (\X) $. }
\end{figure}
%\section{Preliminaries to prove Theorems \ref{ }, .}\label{sS34}
%The objectivity is to prove Theorems \ref{ } and \ref{}.
%We will review the quandle homology groups in \S \ref{SS3j3},
%and show Theorem \ref{} in \S \ref{} and
%Theorem \ref{} in \S \ref{}.

\section{Examples; computations of $\Pi_2(X)$. }\label{twoto}
We will compute the group $\Pi_2(X)$ with respect to same quandles.
%symplectic and orthogonal quandles in \S \ref{asssym},
and do $\Pi_2(X)$ of connected quandles of order $\leq 8$ in \S \ref{twoto8}.
%, as stated in Theorems \ref{thm1b}, \ref{thm1c}, respectively.
The fundamental line of the proofs is to more investigate the exact sequence in \eqref{kihon3} that we called the $P$-sequence.
Actually, the proofs result from computations of the terms $H_3^{\rm gr}(\As (X) )$ and $H_2^Q(X )$ with the TH maps concretely including these $t_X $-torsions.
%Their proofs will compute the $t_X $-torsion subgroups of $H_3^{\rm gr}(\As (X) )$ and $H_2^Q(X )$ concretely.
%Actually, we will determine description of $\As(X)$, and
%compute $H_3^{\rm gr}(\As (X) )$ by the group cohomology theory.
%On the other hand, by the method computing $H_2^Q(X )$ by Eisermann \ref{Eis1,Eis2}, we determine $H_2^Q(X )$.
%By the computations of $H_3^{\rm gr}(\As (X) )$ and $H_2^Q(X )$, check the decomposition $ \Pi_2(X) \cong H_3^{\rm gr}(\As (X) ) \oplus H_2^Q(X ) $.
%In this section, we deal with symplectic quandles denoted by $\SQ_{g,l} $.
%Our goal is to determine $\pi_2(B \SQ_{g,l})$.
%In this subsection, we assume Theorem \ref{theorem}.

\subsection{On Alexander quandles}\label{asssymAll}

We start reviewing Alexander quandles.
Any $\Z[T, T^{-1}] $-module $M $ is a quandle with the operation
$ x\lhd y=y +T( x- y) $ for $x,y\in M$, which is called {\it an Alexander quandle}.
% This operation $\bullet \lhd y$ is roughly a $T$-multiple at $y$.
%A direct calculation shows the quandle axioms, and
%We call its quandle an {\it Alexander quandle}.
The type is the minimal $N$ such that $T^N=\mathrm{id}_M$ since $ x\lhd^n y= y+ T^n (x -y) $.
As a typical example, with a choice of an element $\omega \in \F_q \setminus \{ 0,1\}$,
the quandle of the form $X=\F_q [T]/(T -\omega)$ is
called an {\it Alexander quandle on a finite field} $\F_q$.
Furthermore, an Alexander quandle $X$ is said to be {\it regular}, if $X $ is connected and its type is relatively prime to the order $|X|$, e.g.,
the Alexander quandles on $\F_q$ owing to $\omega^{q-1}=1$. %$An Alexander quandle
%Let $X$ be a regular Alexander quandle of finite order and of type $t_X$.
%Assume that the order $|X| $ is relatively prime to its type, say, Alexander quandles over $\F_q$.

We now discuss regular Alexander quandles $X$ of finite order.
\begin{cor}\label{thm1a}Let $X$ be a regular Alexander quandle of finite order.
Then, the TH-map is an isomorphism.
\end{cor}
\begin{proof}
As is known \cite[\S 4.1]{Nos3},
the $\ell$-torsion subgroups of the homologies $ H_2^{Q}(X) $ and $H_3^{\gr }(\As(X)) $ are zero.
Hence, the proof is immediately obtained from Theorem \ref{theorem}.
%is a certain quotient of $ X \otimes_{\Z} X$ (see Proposition \ref{H2dfg31sha}), and
%it can be easily seen that $t_X$-torsion of the group homology $H_*^{\gr }(\As(X)) $ is zero by the lower central series of $\As(X)$ [see \eqref{lower}].
%In particular, by Theorem \ref{theorem}, we immediately have
\end{proof}

\noindent We however remark that the torsion $ \Pi_2( X ) \otimes \Z/p$ of the Alexander quandles on $\F_q$ with $p\neq 2$
were already computed from another direction (see \cite[Appendix]{Nos2}).
Thus, we omit exampling concrete computations.

\subsection{On symplectic and orthogonal quandles over $\F_q $}\label{asssym}
We will show Theorem \ref{thm1b2} that computes
%about computing the
$\Pi_2(X)$ of symplectic and orthogonal quandles $X$ over finite fields.
Let us begin introduce the two quandles in details:

\begin{exa}[{Symplectic quandle}]\label{Sympex} Let $K$ be a field, and $\Sigma_g$ be the closed surface of genus $g$.
Define $X$ to be the first homology with $K $-coefficients away from $0$,
that is, $H^1(\Sigma_g; K ) \setminus \{ 0 \}= K^{2g} \setminus \{ 0 \}.$
Denote the symplectic form by $\langle, \rangle: H^1(\Sigma_g; K)^2 \ra K$.
Then, this set $X $ is made into a quandle by the operation $ x \lhd y := \langle x,y \rangle y +x \in X$ for any $x,y \in X$, and is called {\it a symplectic quandle (over $K $)}.
The operation $ \bullet \lhd y: X\ra X$ is called
{\it the transvection} of $y$, in the common sense.
%A particularly interesting class is the finite filed case $K= \F_q$, and
%we denote the quandle on $(\mathbb{F}_q)^{2g} \setminus \{ 0 \}$ by $\mathsf{Sp}_q^g.$
Note that the quandle $X$ is of type $p=$Char$( K)$ since $ x\lhd^N y= N \langle x,y \rangle y +x $.
When $K$ is a finite field $\F_q$, we denote the quandle by $\mathsf{Sp}_q^g $.
\end{exa}

\begin{exa}[{Spherical quandle}]\label{sphex}
Let $K $ be a field of characteristic $ \neq 2$. Let $ \langle, \rangle : K ^{n+1} \otimes K^{n+1} \ra K $ be the standard symmetric bilinear form.
Consider a set of the form
$$ S^{n}_K := \{\ x \in K^{n+1} \ | \ \langle x, x \rangle =1\ \}. $$
We define the operation $ x \lhd y $ to be $ 2 \langle x,y \rangle y -x \in S^{n}_K $.
This pair $(S^{n}_K, \ \lhd )$ is a quandle of type 2, and is referred to as {\it a spherical quandle} (over $K $).
%This operation $\bullet \lhd y$ can be interpreted as a rotation through $180$-degrees with the center $y$.
%Hence the quandle is of type $2$.
In the finite case $K=\F_q$, we denote the quandle by $S_q^n $.
\end{exa}
\noindent
As observed above, quandle consists of, figuratively speaking, `operations itself centered at $y \in X$', which can be described as homogenous spaces (see \cite[\S 7]{Joy} for detail).

%Next, we deal with several quandles $X$ such that the $t_X$-torsions of $\Pi_2(X)$ are non-zero. %compared with Theorem \ref{theorem}.
%Let us discuss the symplectic and spherical quandles over $\F_q$ in Examples \ref{Sympex}, \ref{sphex}. % produce the isomorphisms $\Theta_X \oplus \mathcal{H}_X$.
To describe the theorem, we call $q= p^d \in \mathbb{N}$ {\it exceptional}, if the $q$ is one of $\{ 3, \ 3^2, \ 3^3, \ 5, \ 7 \}$,
that is, $d (p-1) \leq 6$ (cf. the condition in Theorem \ref{factsl22} later).

\begin{thm}\label{thm1b2}
Let $q = p^d$ be odd, and be not exceptional. %Assume that a pair $(q,n)$ is not exceptional.
\begin{enumerate}[(I)]
\item
Let $X$ be the symplectic quandle $\mathsf{Sp}_q^n$ over $\F_q$.
%Let $q \neq 3, \ 3^2, \ 3^3, \ 5. $
Then, the TH-map is
an isomorphism. Furthermore, %if $n>1$, then one has
$$\Pi_2(X) \cong \left\{
\begin{array}{ll}
\Z/(q^2 -1) , &\quad {\mathrm if \ } n>1 ,\\
\Z/(q^2 -1) \oplus (\Z/p)^d, &\quad {\mathrm if \ } n=1.
\end{array}
\right.$$
%On the other hand, if $n=1$, then
%$$ \Pi_2(X) \cong , \ \ \ \ H_3^Q(X) \cong \Z/(q^2 -1) \oplus (\Z/p)^{d(d+1)/2} .$$
\item % Let $()$ be....
Let $X$ be the spherical quandle $S^{n}_{q}$ over $\F_q$. The TH-map is
a $[1/2]$-isomorphism.
Moreover, $$\Pi_2(X) \cong_{[1/2]} \left\{
\begin{array}{ll}
\Z/(q^2 -1) , &\quad {\mathrm if \ } n>2 ,\\
\Z/(q^2 -1) \oplus \Z/(q-\delta_q), &\quad {\mathrm if \ } n=2.
\end{array}
\right.$$Here $\delta_q = \pm 1 $ is according to $q \equiv \pm 1$ $(\mathrm{mod} \ 4)$.
\end{enumerate}
\end{thm}

\begin{proof}%[Proof of the $\Pi_2$-parts of Theorem \ref{thm1b}]
%Let $q = p^d$ be odd and not exceptional, i.e., $ d(p-1) <6 $. % in what follows.
%Furthermore,
The proof essentially relies on some works of Quillen and Friedlander, who had calculated
group homologies of some groups of Lie type over $\F_q$.
We list their works after the proof.

For (I), we will mention some homologies associated with the symplectic quandle $X=\mathsf{Sp}^{n}_q$.
As is shown \cite[\S 4.2]{Nos3}, $\As(X) \cong \Z \times Sp(2n;\F_q)$;
hence Theorems \ref{factsl21} and \ref{factsl22} below tell us the group homology $ H_3^{\rm gr}(\As (X)) \cong \Z/(q^2-1)$.
In addition, when $n \geq 2$ the second quandle homology $H_2^Q(X) $ vanishes; see \cite{Nos3};
Therefore the $P$-sequence \eqref{kihon2} is reduced to be
an epimorphism $\Z/(q^2 -1) \ra \Pi_2(X) \ra 0$.
Since $X$ is of type $p$, Theorem \ref{theorem} with $n \geq 2$ thus concludes the isomorphism $\Z/(q^2 -1) \cong \Pi_2(X) $ as required.

Next, we work out the case $n=1 $.
Note from \cite{Nos3} that the second quandle homology $H_2^Q(X) $ is $(\Z/p)^d$. %; see Proposition \ref{dfg31spk} again.
Thereby the $P$-sequence \eqref{kihon2} is rewritten as
\begin{equation}\label{eesp} \ \Z/(q^2 -1) \lra \Pi_2(X) \lra (\Z/p)^d \lra 0 \ \ \ \ \ \ \ \ \mathrm{(exact)}. \end{equation}
Using Theorem \ref{theorem} again, we have $ \Pi_2(X) \cong \Z/(q^2 -1) \oplus (\Z/p)^d$ as required.

For (II), we similarly deal with the spherical quandle $X=S^{n}_q$ with $n \geq 2 $.
As is shown \cite[\S 4.3]{Nos3},
there is a [1/2]-isomorphism $ H_3^{\rm gr}(\As (X)) \cong H_3^{\rm gr}(O(n+1))$.
Then, it follows from Theorem \ref{factsl21} below that $ H_3^{\rm gr}(\As (X)) \cong \Z/(q^2-1)$ without 2-torsion.
Moreover, it is shown \cite[\S 4.3]{Nos3}
that if $n \geq 3$, $H_2^Q(X)$ is an elementary abelian 2-group, and
that if $n=2$ the quandle homology $H_2^Q(X) $ is $\Z / (q - \delta_q)$.
% (Proposition \ref{dfg31sha}).
Hence, the desired isomorphism $\Pi_2(X) \cong \Z/(q^2 -1)$ results from the $P$-sequence \eqref{kihon2} and Theorem \ref{theorem}.
%an epimorphism $\Z/(q^2 -1) \ra .
%Finally, we deal with the case $n=2$.
%By Proposition \ref{dfg31sha}, the quandle homology $H_2^Q(X) $ is $\Z / (q - \delta_q)$, where $\delta_q = \pm 1 $ is according to $q \equiv \pm 1 $ (mod 4).
%Thereby the $P$-sequence \eqref{kihon2} is
%\begin{equation}\label{eeso} \Z/(q^2 -1) \lra \Pi_2(X) \lra \Z / (q - \delta_q) \lra 0 \ \ \ \ \ \ \ (\mathrm{mod} \ 2 ). \end{equation}
%Using Theorem \ref{theorem} again, we reach the goal $ \Pi_2(X) \cong \Z/(q^2 -1) \oplus \Z / (q - \delta_q)$ mod 2.
\end{proof}

As mentioned above, we review the group homologies of the symplectic groups $Sp(2g,\F_q) $ and the orthogonal groups $O(n;\F_q)$.
There is nothing new until the end of this subsection.
We start by recalling the homologies of $Sp(2,\F_q ) $ and $O(3;\F_q)$.
\begin{prop}\label{factsp2} If $p \neq 2$ and $q \neq 3, \ 9 $, then
the first and second homologies of $Sp(2g;\F_q)$ vanish, i.e., $H_1^{\rm gr} (Sp(2g,\F_q)) \cong H_2^{\rm gr} (Sp(2g,\F_q)) \cong 0.$

Furthermore, the $\ell$-torsions of the third homology of $Sp(2,\F_q)$ are expressed as
$$ H_3^{\rm gr}(Sp(2,\F_q))_{(\ell)} \cong \bigl( \Z /(q^2 -1) \bigr)_{(\ell)}, \ \ \ \ \mathrm{for} \ \ell \neq p . $$

On the other hand, the homologies $H_1^{\rm gr}(O(3,\F_q))$ and $H_2^{\rm gr}(O(3,\F_q))$ are annihilated by 2.
Furthermore, the $\ell$-torsions of the third homology of $O(3,\F_q)$ are expressed by
$$ H_3^{\rm gr}(O(3,\F_q))_{(\ell)} \cong \bigl( \Z /(q^2 -1) \bigr)_{(\ell)}, \ \ \ \ \mathrm{for} \ \ell \neq p, \ 2.\ $$
\end{prop}
\begin{proof} See \cite[VIII. \S 4]{FP} or \cite{Fri}, noting the order $ | O(3,\F_q)|=2q (q^2 -1)$.
\end{proof}

We moreover review the group homologies of $Sp(2g ; \mathbb{F}_{q})$ and $O(n, \F_q)$ as follows:

\begin{thm}[\cite{FP,Fri}]\label{factsl21} Let $q = p^d$ be odd.
The inclusion $Sp(2,\F_q) \hookrightarrow Sp(2n, \F_q)$ induce isomorphisms
$H_3^{\rm gr}(Sp(2,\F_q )) \cong_{(\ell)} H_3^{\rm gr}(Sp(2n,\F_q ) ) \cong_{(\ell)} \Z /(q^2 -1) $ localized at $\ell \neq p $.
Furthermore, for $n \geq 3$, the inclusion $O(3,\F_q) \hookrightarrow O(n, \F_q)$ induces isomorphisms
$H_3^{\rm gr}(O(3,\F_q)) \cong_{(\ell)} H_3^{\rm gr}(O(n, \F_q)) \cong_{(\ell)} \Z /(q^2 -1)$ localized at $\ell \neq p, \ 2$.
\end{thm}
\begin{proof} According to \cite{FP}, the inclusions
induces isomorphisms their cohomology with $\Z/ \ell$-coefficients.
Taking limits as $n \ra \infty$, their homologies are known to be $H_3^{\rm gr}(Sp(\infty,\F_q)) \cong_{(\ell)} \Z/(q^2 -1)$ and
$ H_3^{\rm gr}(O(\infty, \F_q)) \cong_{(\ell)} \Z /( q^2 -1) $; see \cite[Theorem 1.7]{Fri}.
Hence, by Propositions \ref{factsp2}, the induced maps on homologies localized at $\ell$ are isomorphisms.
\end{proof}

In addition, we focus on these $p$-torsion parts, and state the vanishing theorem.
\begin{thm}[{Quillen and \cite[\S 4]{Fri}}]\label{factsl22}
Let $q = p^d$ be odd.
If $ d (p-1) >6 $, then the $p$-torsion parts $H_3^{\rm gr}(Sp(2n,\F_q))_{(p)}$
and $H_3^{\rm gr}(O(n+2,\F_q))_{(p)}$ vanish for any $n \geq 1$.
Furthermore, if $n$ is enough large, the $p$-vanishing holds even for $d (p-1) \leq 6 $.
\end{thm}
\begin{rem}\label{dowl2} As a result in \cite[Corollary 1.8]{Fri}, the inclusions $Sp(2n;\F_q) \hookrightarrow GL(2n;\F_q ) \hookrightarrow GL(\infty ;\F_q )$ induce the isomorphism between
the third homology $ H_3^{\rm gr}(Sp(2n ,\F_q)) $ and the Quillen $K$-group $K_3(\F_q)\cong H_3^{\gr }(GL(\infty ;\F_q)) \cong \Z/(q^2 -1)$,
if $ d (p-1) >6$ or $n$ is enough large.
For instance, following \cite[\S 4]{Fri}, we can see that, when $q =p= 5$ and $n \geq 7$, the third homology $ H_3^{\rm gr}(Sp(2n ,\F_q)) $ is $\Z/(5^2 -1)= \Z/24.$
We later use this result in \S \ref{s87}.
\end{rem}
%In this paper, we confine ourselves the $l=5$ case. As is known (see, e.g., \cite[]{Eb}), The generator $ \Omega_3( Sp(2;\F_5)) \cong \Z/120$ is represented by a continues map
%from the Poincare\'e sphere $\Sigma(2,3,5) $ to $K( Sp(2;\F_5),1)$
%which induces an isomorphism $ \pi_1 ( \Sigma(2,3,5) ) \cong Sp(2;\F_5)$.

\subsection{Connected quandles of order $\leq 8$. }\label{twoto8}

Furthermore, we focus on connected quandles $X$ of order $\leq 8$ as follows:
%In addition, as some non-regular quandles $X$, we study the homomorphism $\Theta_X \oplus \mathcal{H}_X$ and determine the group $\Pi_2(X)$ as follows.
\begin{thm}\label{thm1c}
For any connected quandle $X$ of order $\leq 8$, the TH-map is an isomorphism $\Pi_2(X) \cong H_3^{\rm gr}(\As(X)) \oplus \mathrm{Im}(\mathcal{H}_X).$
\end{thm}
\noindent
%As a result, for such quandles discussed above, we obtain a topological meaning of the quandle homotopy
%(cocycle) invariants from the viewpoint of Corollary \ref{corthm1a}.
%Incidentally,
This theorem is a generalization of the previous paper \cite[\S 4]{Nos1}, which dealt with $\Pi_2(X)$ of only quandles $X$ with $|X| \leq 6$.
%Furthermore, this section does not discuss the thrid homology groups, since the homology groups are known by computer program.

In the subsequent subsections,
we will prove Theorem \ref{thm1c} by computing concretely $\Pi_2(X)$ from the list of connected quandles of order $\leq 8$.
In the proof, we will often use the T-map $\Theta_{\Pi \Omega}: \Pi_2(X) \ra \Omega_3(\Ker (\varepsilon_X))$ in Lemma \ref{key231}.

%and of Theorem \ref{theorem}.

\subsubsection{Connected Alexander quandles of order 4, 8}\label{tetete}

We first determine the $\Pi_2(X)$ of the connected quandle of order $4$ as follows:
%As is known \cite{Cla3},
%a connected Alexander quandles of order 4 or 8 is
%one of the following forms: %$ X=\Z[T ]/(2,T ^3+ T ^2+1)$, $ X=\Z[T ]/(2,T ^3+T +1)$.
\begin{prop}\label{thmqq}
If $X$ is the Alexander quandle of the form $\Z[T]/(2, T ^2+ T +1 )$,
then $\Pi_2(X) \cong \Z/2 \oplus \Z /8.$
\end{prop}
%When $X=\Z[T ]/(2,T ^2+\omega+1)$ as above,
%the generators of $\Pi_2(X) \cong \Z/ 8 \oplus \Z / 2$ are dispected by the two $X$-colorings $C_{3_1}$ and $C_{4_1}$ (see \cite[Figure ??]{Nos1}.

\begin{proof}
%We show (I).
We first recall the fact \cite[Proposition 4.5]{Nos1} which says that $\Pi_2(X)$ is either
$ \Z/ 2 \oplus \Z/ 4$ or $ \Z/ 2 \oplus \Z/ 8$ and that $H_2^Q(X)\cong \Z/2$.
By the main theorem \ref{theorem}, hence it suffices to construct a $[1/3]$-isomorphism $H_3^{\rm gr} (\As(X)) \cong_{[1/3]} \Z/8$.
For this, note $\As (X) \cong Q_8 \rtimes \Z $, where $Q_8$ is
the quaternion group of order $8$
(see \cite[Lemma 4.8]{Nos1} for details).
Noting $\mathrm{Type}(X)=3$, consider the quotient $ Q_8 \rtimes \Z/3 $,
which is isomorphic to $Sp(2; \F_3)$.
By using Proposition \ref{factsp2} and the transfer, we thus have $H_3^{\rm gr} (\As(X)) \cong H_3^{\rm gr} (Sp(2; \F_3)) \cong \Z /8 $ up to $3$-torsion as desired.
%Let $Q_8$ act canonically the sphere $S^3$.
%Notice that the 3-fold covering branched over the trefoil $K$ is $S^3 / Q_8$ (see \cite[\S 10.D]{Rolfsen});
%It can be easily seen that the map $\theta_{X,D}$ in \eqref{ } sends the $X$-coloring $C_{3_1 }$ to the identity map $ \pi_1(S^3 / Q_8 ) \ra Q_8$.
%It is known (see \cite[VI. Examples 9.2]{Bro}) that $\Omega_3(Q_8) \cong H_3^{\mathrm{gr}}(Q_8 ) \cong \Z/8$ generated by the pair of the quotient $S^3 / Q_8$ and the identity.
% $\pi_1( S^3 / Q_8) \ra Q_8$.
%Hence $\Phi_{\Pi \Omega}:\Pi_2(X) \ra \Omega_3( Q_8 ) =\Z/8$ gives the surjectivity.
\end{proof}
\begin{rem}\label{thmqq41}
We here note a relation between $\Pi_2(X)$ and quandle homology groups.
As is known, $H_2^Q(X) \cong \Z/ 2$ and $H_3^Q(X) \cong \Z/ 4$ \cite[Remark 6.10]{CJKLS}.
Namely, the summand $\Z/8$ of $\Pi_2(X)$ is evaluated not by the {\it quandle} cohomology,
but by the {\it group} cohomology $H^3_{\gr}(Q_8;\Z/8)$.
It is therefore sensible to deal with 2-torsion of the groups $ \Pi_2(X) $ in general.
\end{rem}

Next, let us consider two Alexander quandles of order $8$ of the forms $ X=\Z[T ]/(2,T ^3+ T ^2+1)$ and $ X=\Z[T ]/(2,T ^3+T +1)$.
Then the both $\Pi_2(X)$ were shown to be $ \Z/2$ \cite[Table 1]{Nos2}.
We remark that, as is known (see, e.g., \cite{Cla3}),
a connected Alexander quandle $X$ of order 4 or 8 is one of
the three quandles above; hence $\Pi_2(X)$ has been determined.

%Finally we now note a relation between $\Pi_2(X)$ and its quandle homologies.
%The computations $H_2^Q(X) \cong \Z/ 2$ and $H_3^Q(X) \cong \Z/ 4$ are known \cite[Remark 6.10]{CJKLS}.
%We conclude that the summand $\Z/8$ of $\Pi_2(X)$ is evaluated not by the {\it quandle} cohomology,
%but by the {\it group} cohomology $H^3_{\gr}(Q_8;\Z/8)$.
%It is sensible to deal with 2-torsion of the group $ \Pi_2(X) $ in general.

\subsubsection{Two conjugate quandles of order 6}\label{tetedwte}

In this subsection, we will calculate $\Pi_2(X)$ of two quandles $S_6$, $S_6'$ of order 6.
Here the quandle $S_6$ (resp. $S_6'$) is defined to be the set of elements of a conjugacy class in the symmetric group $\mathfrak{S}_4$
including $(12)\in \mathfrak{S}_4$ (resp. $ (1234)\in \mathfrak{S}_4$) with the binary operation $x\lhd y = y^{-1} x y$.

%\vskip -1.6pc

\begin{figure}[htpb]
$$
\begin{picture}(20,80)
\put(-120,25){\pc{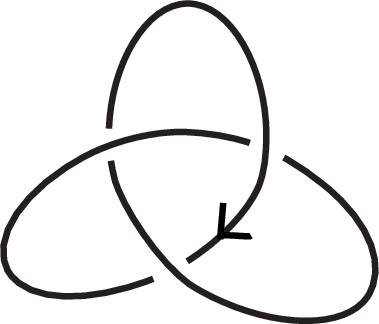}{0.35}}
\put(-143,20){\normalsize $(1432)$}
\put(-65,48){\normalsize $(1342)$}
\put(-52,23){\normalsize $(1423)$}

\put(40,25){\pc{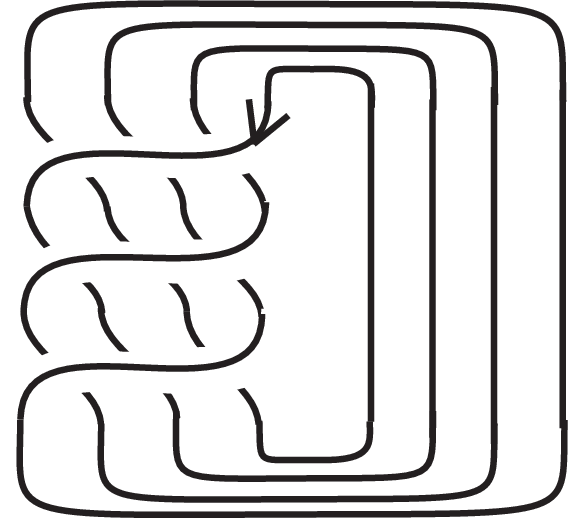}{0.325}}

%\put(40,30){\pc{pic.12.d}{0.45}}

\put(29,55){\normalsize $(12)$}
\put(29,36){\normalsize $(23)$}
\put(29,18){\normalsize $(24)$}
\put(29,-1){\normalsize $(12)$}
%\put(175,-14){\normalsize $(1423)$}
%\put(217,-14){\normalsize $(1234)$}
%\put(259,-14){\normalsize $(1342)$}
\end{picture}
$$
\caption{\label{S'S'6} An $S_6$-coloring $\CC_3$ of the trefoil knot $3_1$, and an $S_6'$-coloring $\CC_4$ of $T_{3,4}$.}
\end{figure}
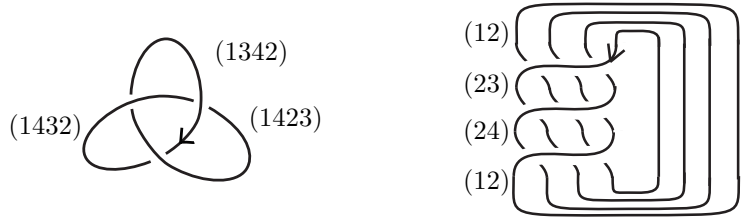

\begin{prop}\label{QSS6}
For the quandle $S_6$, $\Pi_2(S_6) \cong \Z/ 24 \oplus \Z /4$.
The first summand $\Z/24 $ is generated by $\Xi_{S_6, 3_1 }( \CC_3)$, where $\CC_3$ is a coloring of the trefoil knot shown as Figure \ref{S'S'6}.

On the other hand, for another quandle $S'_6$,
we have $\Pi_2(S'_6) \cong \Z/ 12$.
The generator is represented by $\Xi_{S_6' , T_{3,4}}( \CC_4)$, where $\CC_4$ is a coloring of the torus knot $T_{3,4}$ shown as Figure \ref{S'S'6}.
\end{prop}

\begin{proof}
We will show the sequence \eqref{e.s.541790} below.
It follows from the proof of \cite[Proposition 4.9]{Nos1} that $H_2^Q(S_6) \cong \Z/4 $ and $H_2^{\gr }(\As(S_6)) \cong 0$,
and that $H_3^{\rm gr}(\As(S_6)) $ is a quotient of $\Z /24$.
Hence the $P$-sequence \eqref{kihon3} becomes
\begin{equation}\label{e.s.541790}
\Z/24 \lra \Pi_2(S_6) \stackrel{\mathcal{H}}{\lra} H_2^Q(X) \ \bigl ( \cong \Z/4 \bigr) \ \lra 0.
\end{equation}

Next we will show that $\Pi_2 (S_6)$ surjects onto $\Z/2 4$.
It is shown \cite[Lemma 4.10]{Nos1} that
the kernel $ \Ker (\varepsilon_X)$ is the binary tetrahedral group $D_{24}= Sp(2;\F_3)$ whose third homology is $\Z/24 $ (see Proposition \ref{factsp2}).
%the presentation is as follows:
%$$0 \longrightarrow D_{24} \longrightarrow \As ( S_6) \stackrel{\varepsilon_X}{\longrightarrow} \Z \longrightarrow 0 \ \ \ \textrm{(exact)}, $$. where $D_{24}$ is the binary tetrahedral group whose third homology is $\Z/24\Z$ (see, for example, 6.1(3) in \cite{SV}).
Let $D_{24}$ act canonically the 3-sphere $S^3$.
Since the 4-fold covering branched over the trefoil is $S^3 / D_{24} $ (see \cite[\S 10.D]{Rolfsen}),
it can be easily seen that the map $\theta_{X,D}$ in \eqref{sanatana} sends the $X$-coloring $\CC_{3 }$ to an isomorphism $ \pi_1(S^3 / D_{24} ) \stackrel{\sim}{\ra} Sp(2;\mathbb{F}_3)$.
Since $\Omega_3( D_{24}) \cong \Z/24$ is known to be generated by the pair ($S^3 / D_{24} $, $\mathrm{id}_{D_{24}}$)
[see \cite[VI. Examples 9.2]{Bro}], the T-map $\Theta_{ \Pi \Omega}: \Pi_2 (S_6) \ra \Omega_3( D_{24}) $ in Lemma \ref{key231} turns out to be surjective.
% and the identity. $\pi_1( S^3 / Q_8) \ra Q_8$.
%Hence, the desired result follows from Lemma \ref{short} in the case $k=2$.

Finally, for proving the decomposition $\Pi_2(S_6) \cong \Z/4 \oplus \Z/ 24$,
it is enough to show that the class $[\CC_3]\in \Pi_2(S_6) $ is sent to $2 \in H_2^Q(X) \cong \Z/4$ by the map $\mathcal{H}_X.$
By the formula \eqref{kantan},
$$\mathcal{H}_X ( [\CC_3])= \Gamma_{\CC_3}(\mathfrak{l})= e_{(1432)}^{-2} e_{(1432) } e_{(1423)} \in \Ker (\varepsilon_X) \cap \mathrm{Stab}(x_0) \cong \Z /4 . $$
An elementary calculation can show the square $\mathcal{H}( [\CC_3])^2 = 1 $ and $ \mathcal{H}( [\CC_3]) \neq 1 $ in $\As (X)$, although
we will not go into the details.
In the sequel, $\mathcal{H}( [\CC_3]) =2 \neq 0$ as desired.

\

Changing the subject, we will compute another $\Pi_2(S'_6)$.
For this, we now explain the sequence \eqref{e.s.5417} below.
It is shown \cite[Lemma 4.12 and Appendix A.2]{Nos1} that $H_2^Q(S_6' ) \cong H_2^{\gr }(\As(S_6')) \cong \Z/2$,
and to estimate the order $|H_3^{\rm gr}(\As(S_6')) | \leq 12$.
Hence, as a routine work, the $P$-sequence \eqref{kihon3} becomes
\begin{equation}\label{e.s.5417}
H_3^{\rm gr}(\As(S_6')) \lra \Pi_2(S_6') \stackrel{\mathcal{H}}{\lra} \Z/2 \lra \Z/2 \lra 0.
\end{equation}

For the sake of proving $\Pi_2(S_6') \cong \Z/12$, it is sufficient to show the
surjectivity of the T-map $\Theta_{\Pi \Omega} : \Pi_2(S_6') \ra \Z/12$.
As is shown \cite[Lemma 4.12]{Nos1},
the kernel $ \Ker (\varepsilon_X)$ is the alternating group $A_4 $ of order $12$ whose third homology is $\Z/12$.
%the presentation is as follows:
%$$0 \longrightarrow D_{24} \longrightarrow \As ( S_6) \stackrel{\varepsilon_X}{\longrightarrow} \Z \longrightarrow 0 \ \ \ \textrm{(exact)}, $$. where $D_{24}$ is the binary tetrahedral group whose third homology is $\Z/24\Z$ (see, for example, 6.1(3) in \cite{SV}).
Since the quandle $S_6'$ is of type $4$ and
the double cover branched over the knot $T_{3,4}$ is $S^3 / D_{24} $ (see \cite[\S 10. D and E]{Rolfsen}),
we can show that the map $\theta_{X,D}$ in \eqref{sanatana} sends the $X$-coloring $\mathcal{C}_{4 }$ to the epimorphism
$ \pi_1(S^3 / D_{24} )=D_{24} \ra A_4 $, which is a central extension.
Hence, the class $[\theta_{X,D}(\CC_4) ]$ is a generator of $\Omega_3(A_4)\cong \Z/12$.
This means the surjectivity of $\Theta_{\Pi \Omega} : \Pi_2(S_6') \ra \Z/12$; Recalling $ |H_3^{\rm gr}(\As(S_6')) | \leq 12$ and the sequence \eqref{e.s.5417} above,
the T-map $\Theta_{\Pi \Omega}$ is an isomorphism $ \Pi_2(S'_6) \cong \Z/ 12$.
Furthermore, by this process, the generator is represented by the coloring $\CC_4$.
\end{proof}
\subsubsection{The remaining quandle and the proof of Theorem \ref{thm1c}}\label{s87ojf}
Finally, the rest of connected quandle of order $8$ is the (extended) quandle $\X$ explained in \S \ref{s87ab3e}, where
$X$ is the Alexander quandle of the form $X=\Z[T]/(2, T^2 +T+1 ).$
Since $ \As (X) \cong Q_8 \rtimes \Z $ (see the proof of Proposition \ref{thmqq}), we have $|\X|=8$ by definition.

\begin{prop}\label{QSS8} Let $\X$ be the above quandle of order $8$.
Then $\Pi_2(\X) \cong \Z/8$.
\end{prop}
\begin{proof}
We first show that the induced map $p_* : \As (\X) \ra \As (X)$ is an isomorphism. % $\As (\X) \cong \As (X) $.
As is seen in the proof of Proposition \ref{thmqq}, we note
$H_2^{\rm gr}(\As (X)) \cong 0$ and $H_1^{\rm gr}(\As (X)) \cong \Z$.
Since the map $p_*$ is a central extension (see Theorem \ref{extthm}(iii)), $p_*$ is an isomorphism.
Next we will show $ \Pi_2(\X) \cong \Z/8$ as follows.
Note $H_2^Q(\X) \cong 0$ from Theorem \ref{extthm} (iii),
and $H_3^{\rm gr}(\As (\X)) \cong H_3^{\gr} (Q_8 \rtimes \Z ) \cong \Z /8$ from the proof of Proposition \ref{QSS6}.
Therefore the $P$-sequence
is reduced to be an epimorphism $\Z/8 \ra \Pi_2(\X)$.
By Theorem \ref{theorem} again, we conclude $ \Pi_2(\X) \cong \Z/8$. % \ra \Pi_2(\X)$.
\end{proof}

\begin{proof}[Proof of Theorem \ref{thm1c}]
One first deals with quandles $X$ with $|X| = 3,5,7$.
Then $X$ is shown to be an Alexander quandle over the finite field $\F_{|X|}$ \cite{EGS}.
By Theorem \ref{thm1a}, the isomorphism $\Theta_X \oplus \mathcal{H}_X $ holds for such $X$.
Next, a connected quandle of even order with $|X | \leq 8$ is known to be one of the quandles in the previous subsections (see \cite{Cla3}).
%the proceeding subsections already have shown the isomorphism.
Hence the above proposition completes the proof.
\end{proof}
\subsection{Dehn quandle of genus $ \geq 7$}\label{s87}

This subsection deals with Dehn quandles. To discuss this, we fix some notation:

\vskip 0.4pc
\noindent
{\bf Notation.} Denote by $\Sigma_{g,k}$ the closed surface of genus $g$ with $k$ boundaries as usual.
Let $ \mathcal{M}_{g,k}$ denote the mapping class group of $\Sigma_{g,k} $ which is the identity on the $k$-boundaries.
In the case $k=0$, we often suppress the symbol $k$, e.g., $\Sigma_{g,0}=\Sigma_{g}$.

\vskip 0.4pc

We now review Dehn quandles \cite{Y}. Consider the set, $\D$, defined to be
\begin{equation}\label{defdg} \notag \D:= \{ \ \textrm{ isotopy classes of (unoriented) non-separating simple closed curves } \gamma \ {\rm in \ } \Sigma_g \ \}. \end{equation}
%where ``essential" means that no disk in $\Sigma_g$ bounds $\gamma$.
For $ \alpha, \ \beta \in \D $, we define $\alpha \tri \beta \in \D$ by $ \tau_{\beta}(\alpha)$,
where $ \tau_{\beta} \in \M_{g} $ is the positive Dehn twist along $\beta$.
The pair ($\D, \lhd)$ is a quandle, and called {\it (non}-{\it separating) Dehn quandle}.
As is well-known, any two non-separating simple closed curves are
related by some Dehn twists. Hence, the quandle $\D$ is connected, and is not of any type $t_X $.
In addition, since the Dehn twists are transvections in the view of the cohomology $H^1(\Sigma_g ;\F_p )$, for any prime $p$, we have a quandle epimorphism $\mathcal{P}_p$ from $\D$ to
the symplectic quandle $\mathsf{Sp}^g_p$. Concisely, $ \mathcal{P}_p: \D \ra \mathsf{Sp}^g_p. $
The Dehn quandle $\D$ is applicable to study 4-dimensional Lefschetz fibrations (see, e.g., \cite{Y,Nos3}).

We now aim to compute the second homotopy groups $\pi_2(B \DD)$ in a stable range as follows:
\begin{thm}\label{hom3ocdr}
Let $g \geq 7 $.
The group $ \Pi_2( \DD)$ is isomorphic to either $\Z/24 $ or $ \Z/48 $.
Furthermore, a generator of $\Pi_2 (\DD)$ is represented by a $\D$-coloring in Figure \ref{koutssnpn}.
\end{thm}

%\subsection{An upper bound of $\pi_2(B\DD )$ }\label{sds87}
%While Theorem \ref{hom3ocdr2} is shown in \ref{}, the goal in this section is to show Theorem \ref{hom3ocdr}.
%For this, we prepare two lemmas.

%\begin{lem}\label{homohyouka} [Proof of Theorem \ref{hom3ocdr}]
\begin{proof}
We first observe homologies of the associated group $\As( \DD)$.
Note the well-known facts $H_1^{\rm gr} (\M_g) \cong 0 $ and $H_2^{\rm gr} (\M_g) \cong \Z$ (see \cite{FM});
Gervais \cite{Ger} showed the isomorphism $\As (\DD) \cong \Z \times \mathcal{T}_g $,
where $ \mathcal{T}_g $ is the universal central extension of $ \M_g$.
Then, Lemma \ref{Les64t2} below and Kunneth theorem immediately imply $H_2^{\rm gr} (\As (\DD)) \cong 0 $ and $H_3^{\rm gr}(\As (\DD)) \cong \Z/24$.

Next, we study the $P$-sequences in respect to the above epimorphism $\mathcal{P}_5 : \DD \ra \mathsf{Sp}^g_5$ with $p=5.$
Let us use the facts $H_2^Q(\D) \cong \Z /2 $
and $ H_2^Q( \mathsf{Sp}^g_5) \cong 0 $; see \cite[\S 4.2 and \S 4.6]{Nos3}.
%is shown by Proposition \ref{eahges}. %Proposition \ref{eahges}
Then, these $P$-sequences are written in
$$
\xymatrix{ \Z/ 24 \ar[d]_{(\mathcal{P}_5)_* } \ar[rr]& & \Pi_2 (\D ) \ar[d]_{(\mathcal{P}_5)_* } \ar[rr]& & H_2^Q(\D) \bigl(\cong \Z /2 \bigr) \ar[r] \ar[d]_{(\mathcal{P}_5)_* } & 0 & (\mathrm{exact}) \ \\
H_3^{\rm gr}(\As (\mathsf{Sp}^g_5)) \ar[rr]_{\delta_*} & & \Pi_2( \mathsf{Sp}^g_5) \bigl( \cong \Z /24 \bigr)\ar[rr] & & H_2^Q( \mathsf{Sp}^g_5) \bigl( \cong 0 \bigr) \ar[r] & 0 & (\mathrm{exact}).}
$$
Here the proof of Theorem \ref{theorem} says that the bottom $\delta_*$ is an isomorphism $ H_3^{\rm gr}(\As (\mathsf{Sp}^g_5)) \ra \Pi_2( \mathsf{Sp}^g_5) \cong \Z/24$.

Furthermore, we claim that the middle map $(\mathcal{P}_5)_*: \Pi_2 (\D )\ra \Pi_2(\mathsf{Sp}^g_5)$ is surjective.
By combing the T-map $\Theta_{\Pi \Omega}$ in \eqref{sanatana2}
with the epimorphism $ \mathcal{P}_5 : \D \ra \mathsf{Sp}^g_5$ above, we have a composite map
\begin{equation}\label{seqdh55}
\Pi_2 (\D ) \xrightarrow{ \ (\mathcal{P}_5 )_* \ } \Pi_2(\mathsf{Sp}^g_5 ) \xrightarrow{ \ \Theta_{\Pi \Omega} \ } \Omega_3(Sp(2g;\F_5)). \end{equation}
To show the claim, it is enough to prove the surjectivity of this composite.
Let us recall the isomorphisms $\Omega_3(Sp(2g;\F_5)) \cong H^{\rm gr}_3(Sp(2g;\F_5)) \cong \Z/24$ from Remark \ref{dowl2}, and consider the $\DD$-coloring $\mathcal{C}$ illustrated below.
Notice that the quandle $\mathsf{Sp}^g_5$ is of type $5$, and
that the 5-fold cover of $S^3$ branched along the trefoil is the Poincar\'{e} sphere $\Sigma(2,3,5)$ (see \cite[\S 10.D]{Rolfsen}),
whose $\pi_1$ is $ Sp(2;\F_5)$ exactly.
Using the map $ \theta_{X,D} $ in \eqref{sanatana}, we can see that the associated homomorphism $ \theta_{X,D}(\mathcal{C}): \pi_1(\Sigma(2,3,5) ) \ra Sp(2;\F_5)$ is an isomorphism.
Hence, the class $[\theta_{X,D}(\mathcal{C})]$ is a generator of $ \Omega_3( Sp(2;\F_5)) $.
It follows from Theorem \ref{factsl21} that the inclusion $Sp(2;\F_5) \hookrightarrow Sp(2g;\F_5) $ induces an isomorphism between these homologies without $5 $-torsion,
which means the claimed surjectivity of $(\mathcal{P}_5)_* $.% mentioned above, $ is surjective.

\vskip 0.4pc
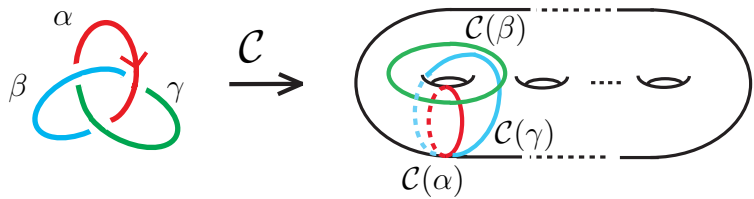
\begin{figure}[h]
$$
\begin{picture}(90,66)
\put(-79,25){\pc{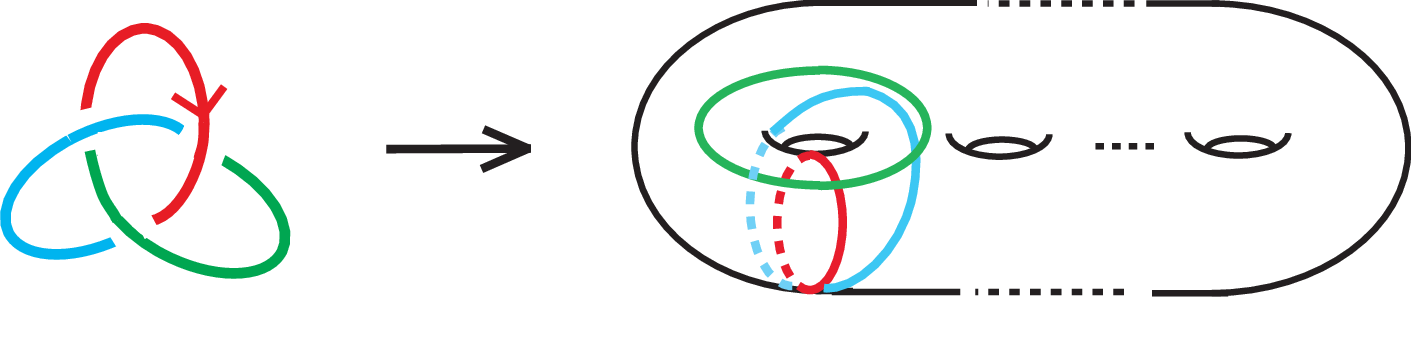}{0.40056}}
\put(-66,56){\large $\alpha $}
\put(-83,29){\large $\beta $}
\put(-23,29){\large $\gamma $}

\put(4.2,42){\LARGE $\mathcal{C}$}

\put(66,-6){\large $\mathcal{C}(\alpha )$}
\put(101,11){\large $\mathcal{C}(\gamma )$}
\put(90,51){\large $\mathcal{C}(\beta ) $}

\end{picture}$$

\caption{\label{koutssnpn} A $\D$-coloring of the trefoil knot}
\end{figure}

However, the left vertical map $(\mathcal{P}_5)_*: \Z/24 \ra H_3^{\rm gr}(\As (\mathsf{Sp}^g_5) )$ is not surjective (see Lemma \ref{agh} below).
Hence, by carefully observing the above commutative diagram, $\Pi_2(\D)$ turns out to be either $\Z/48 $ or $\Z /24$.
\end{proof}
We now show two lemmas which are used in the proof above.

\begin{lem}\label{Les64t2}% For $g \geq 3$, $ H_2(\As (\DD)) \cong 0$.
Let $\mathcal{T}_g $ be the universal central extension on the group $\M_g$.
If $g \geq 3$, then $H_2^{\rm gr} ( \mathcal{T}_g ) $ vanishes. Furthermore, if $g \geq 7$, then $H_3^{\rm gr}( \mathcal{T}_g ) \cong \Z/24$.
\end{lem}
\noindent
We now prove Lemma \ref{Les64t2} by using Quillen plus constructions and Madsen-Tillmann \cite{TM}.

\begin{proof}
We first immediately have $H_2^{\gr }(\mathcal{T}_{g} ) \cong 0$,
since $ \mathcal{T}_{g} $ is the universal central extension of $\M_g$ and the group $\M_g$ is perfect (see, e.g., \cite[Corollary 4.1.18]{Ros}).

We next focus on $H_3^{\rm gr}( \mathcal{T}_{g} ) $ with $g \geq 3$.
Let $ B\M_{g,k}^+ $ denote Quillen plus construction of an Eilenberg-MacLane space of $\M_{g,k} $ (see, e.g., \cite[Chapter 5.2]{Ros} for the definition).
Since $ \M_{g,k}$ is perfect, the space $ B\M_{g,k}^+ $ is simply connected.
As a basic property of plus constructions (see \cite[Theorem 5.2.7]{Ros}), the homotopy group $\pi_3 ( B\M_{g}^+ ) $ is isomorphic to $H_3^{\rm gr}(\mathcal{T}_{g} ) $.

It is therefore sufficient to calculate $\pi_3 ( B\M_{g}^+) $ for $g\geq 7$.
For this, we set up some preliminaries.
Consider the inclusion $\M_{g,1} \ra \M_{g+1,1}$ obtained by gluing the surface $\Sigma_{1,2}$ along one of its boundary components.
Let $\M_{\infty,1 } := \lim_{g \ra \infty}\M_{g,1} $.
Furthermore put an epimorphism $\delta_g: \M_{g,1} \ra \M_{g}$ induced by gluing a disc to the boundary component of $\Sigma_{g,1}$.
According to the Harer-Ivanov stability theorem improved by \cite{RW}, the inclusion $ \iota_{\infty}: \M_{g,1} \ra \M_{\infty,1 }$ induces
an isomorphism $H_j^{\rm gr}( \M_{g,1} ) \cong H_j^{\rm gr}( \M_{\infty,1 } )$, and the map $\delta_g$ does $H_j^{\rm gr}( \M_g ) \cong H_j^{\rm gr}( \M_{g,1} )$, for $j \leq 3$.

Finally, we consider the maps $\delta_g^+: B\M_{g,1}^+ \ra B\M_{g}^+$ and $ \iota^+_{\infty}: B\M_{g,1}^+ \ra B\M_{\infty,1 }^+$
induced by $\delta_g$ and $\iota_{\infty}$, respectively.
By Whitehead theorem, these maps induce isomorphisms
$$ (\delta_g^+)_*: \pi_3 ( B\M_{g,1}^+) \cong \pi_3 ( B\M_{g}^+), \ \ \ \ \ \ \ \ (\iota^+_{\infty})_* : \pi_3 ( B\M_{g,1}^+) \cong \pi_3 ( B\M_{\infty,1 }^+).$$
However the $\pi_3 ( B\M_{\infty}^+) \cong \Z /24 $ was shown by Madsen and Tillmann \cite{TM} (see also \cite{Eb}).
In summary, we have $ H_3^{\rm gr}( \mathcal{T}_{g} ) \cong \pi_3 ( B\M_{\infty ,1 }^+) \cong \Z /24 $ as required.\end{proof}

\begin{lem}\label{agh}
The induced map $(\mathcal{P}_5)_* : H_3^{\rm gr}( \As(\D)) \ra H_3^{\rm gr}( \As(\mathsf{Sp}_5^g) )$ with $g \geq 7$ is not surjective.
\end{lem}
\begin{proof}
As mentioned previous, recall the reduction of $\mathcal{P}_5: \As(\D) \ra \As(\mathsf{Sp}_5^g) $ to $\Z \times \mathcal{T}_{g} \ra \Z \times Sp(2g ;\mathbb{F}_5) $.
We easily see that it factors through $\Z \times \M_g$. However $H_3^{\gr }(\M_g) \cong \Z/12$ is known \cite{TM} (see also \cite[\S 1]{Eb}).
Since $H_3^{\rm gr}( \mathcal{T}_{g} ) \cong H_3^{\rm gr}( Sp(2g ;\mathbb{F}_5 )) \cong \Z/24$ as above,
the map $(\mathcal{P}_5)_*$ is not a surjection.
\end{proof}

\section{Application to third quandle homologies}\label{s87as}
As an application from the study of the homotopy group $\pi_2(BX)$,
we compute some torsion subgroups of third quandle homologies $H_3^Q(X)$ of finite connected quandles $X$.
First, we prove Theorem \ref{thm3homology} owing to the facts explained in \S \ref{sSproof}.
Next, in \S \ref{s87Al}--\ref{s87asl}, we later determine $H_3^Q(X)$ of some quandles.

We briefly explain a basic line to study $ H_3^Q(X)$ in this section.
Let $B(X,X)$ be the rack space associated to the primitive $X$-set.
Note the following isomorphisms:
\begin{equation}
\label{eq.H3RQ} \hspace*{5pc}
H_2 (B(X,X))\cong H^R_3 (X) \cong H_3^Q (X) \oplus H_2^Q (X) \oplus \Z,
\end{equation}
where the first isomorphism is derived from Remark \ref{hayano}, and the second was shown \cite[Theorem 2.2]{LN}.
Composing this \eqref{eq.H3RQ} with the result on $\pi_2(BX) = \pi_2(B(X,X) )$ from Theorem \ref{theorem} can compute
some torsion of the quandle homology $ H_3^Q (X)$ [see Lemma \eqref{clamonoid}].

Following this line, we now prove Theorem \ref{thm3homology} as a general statement, which is rewritten as
\begin{thm}[{Theorem \ref{thm3homology}}]\label{thm3homology2}
Let $X$ be a connected quandle with $|X| < \infty$.
Let $\Ker (\psi_X) $ be the abelian kernel in \eqref{AI}.
%Let $\mathcal{K} $ be the finite abelian subgroup of the kernel $\Ker (\psi_X) $ in \eqref{AI}.
Then an isomorphism $H_3^Q(X) \cong_{(\ell)} H_3^{\rm gr}(\As(X)) \oplus (\Ker (\psi_X) \wedge \Ker (\psi_X) ) $ holds
after localization at any prime $\ell$ which does not divide $2|\mathrm{Inn}(X)|/|X|$.
\end{thm}
\begin{proof}
By Lemma \ref{clamonoid} and the isomorphism \eqref{relezafk} below, we have an isomorphism
\begin{equation}\label{re0ej} \pi_2(BX)_{(\ell)} \oplus H_2^{\rm gr}(\Ker (\psi_X) )_{(\ell)} \cong H_2(B(X,X))_{(\ell)}. \end{equation}
Recall from Theorem \ref{theorem} the isomorphism $\pi_2(BX)_{(\ell)} \cong H_3^{\rm gr}(\As(X))_{(\ell)} \oplus H_2^Q(X)_{(\ell)} \oplus \Z_{(\ell)}.$
Hence, together with \eqref{eq.H3RQ} above, the isomorphism \eqref{re0ej} is rewritten in
$$ H_3^{\rm gr}(\As(X))_{(\ell)} \oplus H_2^Q(X)_{(\ell)} \oplus \Z_{(\ell)} \oplus H_2^{\rm gr}(\Ker (\psi_X) )_{(\ell)} \cong \Z_{(\ell)} \oplus H_3^Q(X)_{(\ell)} \oplus H_2^Q(X)_{(\ell)}. $$
Since
the second group homology $ H^{\rm gr}_2( \Ker (\psi_X))$ is the exterior product $ \bigwedge^2 (\Ker (\psi_X)) $ [see \cite[\S V.6]{Bro}],
by a reduction of the both hand sides, we reach at the conclusion. %de the required $H_3^Q(X) \cong H_3^{\rm gr}(\As(X)) \oplus \mathcal{K} \oplus (\bigwedge^2 \mathcal{K} ) $.
\end{proof}

\subsection{Preliminaries}\label{sSproof}

\baselineskip=16pt
We now recall basic properties of the rack space $B(X,Y)$ introduced in \S \ref{SS3j3out},
which are used in the preceding proof. To begin, we review the following:

\begin{prop}[{\cite[Theorem 3.7 and Proposition 5.1]{FRS1}}]\label{propFRS}
Let $X$ be a quandle, and $Y$ an $X$-set. Decompose $Y$
into the orbits as $Y = \sqcup_{i \in I}Y_i$. For $i \in I$ and an element $y_i \in Y_i$, denote by $\mathrm{Stab}(y_i ) \subset \As(X)$ the stabilizer of $y_i$.
Then, the subspace $B(X,Y_i) \subset B(X,Y)$ is path-connected, and the
natural projection $B(X,Y_i) \ra BX$ is a covering.
Furthermore, the $\pi_1(B(X,Y_i))$ is
the stabilizer $\mathrm{Stab}(y_i ) $, by observing the covering transformation group.\end{prop}

We next observe the spaces $B(X,Y)$ in some cases of $Y$, where $X$ is assumed to be connected.
First, since $ \pi_1(BX) \cong \As (X)$ from the 2-skeleton of $BX$,
the projection $ B(X,Y)\ra BX$ with $Y= \As(X)$ is the universal covering.
Next, we let $Y$ be the inner automorphism group Inn$(X)$, and be acted on by $\As(X)$ via \eqref{AI}. %When $X$ is connected,
Considering the surjections Inn$(X) \ra X \ra \{ \mathrm{pt} \}$ as $X$-sets,
they then yield a sequence of the coverings
\begin{equation}\label{relfk} B(X, \mathrm{Inn}(X)) \lra B(X,X) \lra BX. \end{equation}
\begin{rem}\label{setrans}
According to Proposition \ref{propFRS},
$\pi_1( B(X,X)) $ is the stabilizer $\mathrm{Stab}(x_0) \subset \As(X) $, and
$\pi_1 (B(X, \mathrm{Inn}(X))) $ is the abelian kernel $\Ker (\psi_X )$ of $\psi_X:\As(X) \ra \mathrm{Inn}(X)$ in \eqref{AI}.
\end{rem}

%Furthermore, as is known (\cite{FRS1}. See also \cite[]{Cla}), the action of $\pi_1(BX)$ on the homoropy groups $\pi_i(BX)$ is trivial.

We further observe homologies of the space $B(X,X)$.
Let $\ell$ be a prime which does not divide the order $|\mathrm{Inn}(X)|/|X|$.
As is known, the action of $\pi_1(BX)$ on the homology group $H_*(BX)$ is trivial (see \cite{Cla}).
Then the first covering in \eqref{relfk} induces an isomorphism between their homologies localized at $\ell$. To be precise
\begin{equation}\label{relezafk} H_*( B(X, \mathrm{Inn}(X) ))_{(\ell)} \cong H_* ( B(X,X))_{(\ell)}. \end{equation}
Actually the transfer map of the covering is an inverse mapping (see \cite[Proposition 4.2]{Cla}).
%and the second isomorphism is obtained by Remark \ref{hayano}.
%In summary, for a regular quandle $X$, the $H_*(BX)_{(\ell)}$ is a direct divisor of $H_*(B(X, \mathrm{Inn}(X)))_{(\ell)}$.

Finally, we will observe a relation between the homotopy and homology groups of the rack space $B(X,\mathrm{Inn}(X) ) $. % in Lemma \ref{}.
Refer to the fact
that Clauwens \cite[\S 2.5]{Cla} gave a topological monoid structure on $B(X,\mathrm{Inn}(X) )$.
%Hence
\begin{lem}\label{clamonoid} Let $X$ be a connected quandle. Let $B X_G $ denote the rack space $B(X,\mathrm{Inn}(X))) $ for short.
Then the Hurewicz homomorphism $\pi_2(BX) = \pi_2( BX_G )\ra H_2(B X_G )$ gives a $[1/2]$-splitting. In particular
%In particular, we have a decomposition $H_3(BX) \cong_{(\ell)} \pi_2(BX) \oplus H_2^{\rm gr}(\Ker (\psi))$, where $\Ker (\psi)$
%denotes the abelian kernel in \eqref{}.
$$ H_2(B X_G ) \cong_{[1/2]} \pi_2(BX) \oplus H_2^{\rm gr}(\Ker (\psi_X) ). $$
\end{lem}

\begin{proof} As is known, the second $k$-invariants of path-connected topological monoids with CW-structure are
annihilated by $2$ \cite{Sou,AP}. Namely, the Hurewicz map is a $[1/2]$-splitting $\mathcal{H}$.
% splits modulo 2-torsion.
Noting $\pi_1( B X_G ) \cong \Ker (\psi_X) $ by Remark \ref{setrans},
we have $\Coker(\mathcal{H}) \cong H_2^{\rm gr}(\Ker (\psi_X) ) $,
which implies the required decomposition.
%By Remark \ref{},
\end{proof}

\subsection{Example 1; Alexander quandles.}\label{s87Al}

\large
\baselineskip=16pt
%We now prove Theorems \ref{Ale3homology}, \ref{3homology} based on the properties of rack spaces explained above:
We will compute the third quandle homologies of some quandles in more details than Theorem \ref{thm3homology2} shown above.
%To begin, we will prove Theorem \ref{Ale3homology} which is in Alexander case.
Notice that Theorem \ref{thm3homology} is of use with respect to quandles $X$ such that the order $|\mathrm{Inn}(X)|/|X| $ is small.
As such examples, %in Alexander case,
we consider the third quandle homologies of Alexander quandles.
Actually, notice that the order $|\mathrm{Inn}(X)|/|X| $ equals $\mathrm{Type}(X)$ exactly.
% (see \cite[Lemma 5.6]{Nos2}).
% the describe third homologies .
\begin{thm}\label{Ale3homology}
Let $X$ be a regular Alexander quandle of finite order.
%Assume that the order $|X|$ is $< \infty$, and is relatively prime to $t_X $.
%Let $ \mathcal{K}$ be the quotient module $ X\otimes_{\Z} X /(x \otimes y -Ty \otimes x)_{x,y \in X} \ (\cong H_2^Q (X))$.
Then, there is the [1/2]-isomorphism $H_3^Q(X) \cong H_3^{\rm gr}(\As(X)) \oplus \bigl( \bigwedge^2 \Ker (\psi_X) \bigr) $.

Moreover, if the order of $X$ is odd, then the 2-torsion subgroups of the both sides are zero. \end{thm}
Here remark that, the associated group $\As(X)$ and the kernel $\Ker (\psi_X)$ have been completely calculated by Clauwens \cite{Cla2}.
% (see also Appendix \ref{Sin2}).
In particular, the group $\As(X)$ is known to be a nilpotent group of degree $2$;
% (see \eqref{lower} for the lower central series);
hence the group (co)homology is not simple, e.g., it contains some Massey products (see \cite[\S 4]{Nos4}).
%Hence, if $H_2^Q(X)$ is far from zero, the $H_3^Q(X)$ is not simple in general.

\begin{proof}[Proof of Theorem \ref{Ale3homology}]
%Let $X$ be a regular Alexander quandle of finite order.
Let $\mathcal{K}$ denote the abelian $\Ker (\psi_X)$ for short.
Consider the rack space $B( X, \mathrm{Inn}(X))$
whose $\pi_1$ is $\Ker(\psi_X) = \Z \times \mathcal{K}$ by Remark \ref{setrans}.
It is shown that the torsion subgroups of the homology $ H_2( B(X, \mathrm{Inn}(X ))) $
and $ H_2( B(X,X)) $ are annihilated by $|X|^3$ (see \cite[Lemma 5.7]{Nos2} and \cite[Theorem 1.1]{LN}).
%is annihilated by $|X|$, and that the .
Hence, $H_3^Q(X)$ and $H_3^{\gr}(\As(X)) $ are annihilated by $|X|^3$, by repeating the proof of Theorem \ref{thm3homology2}.
Noticing $t_X = |\mathrm{Inn}(X)|/|X| $ as above,
%\eqref{re0ej} holds for $\ell$ prime to $2 t_X $ as well.
the purpose $H_3^Q(X) \cong_{[1/2]} H_3^{\rm gr}(\As(X)) \oplus \mathcal{K} \oplus (\wedge^2 \mathcal{K} ) $ immediately follows from the regularity and Theorem \ref{thm3homology2}.

Finally, we work out the case where $|X|$ is odd.
By repeating the above discussion, the homologies $ H_2( B(X,X))$ and $ H_2(B(X, \mathrm{Inn}(X) ) ) $ have no $2$-torsion;
so does the $H_3^Q(X) $ as well. % in all torsion similarly.
\end{proof}

\subsection{Example 2; Symplectic and spherical quandles.}\label{s87asl}

We next compute $H_3^Q(X)$ for the symplectic and orthogonal quandles over $\F_q$,
where we exclude the exceptional cases of $q$, i.e., $q\neq 3, 3^2, 3^3, 5, 7 $.
%The reader should reread the statement of Theorem \ref{3homology}.
We remark that the orders $|\mathrm{Inn}(X)|/|X| $ are not simple, in contract to Theorem \ref{thm3homology2}.

\begin{thm}\label{thm14b}
Let $q = p^d$ be odd, and be not exceptional. %Assume that a pair $(q,n)$ is not exceptional.
\begin{enumerate}[(I)]
\item
Let $X$ be the symplectic quandle $\mathsf{Sp}_q^n$ over $\F_q$.
%Let $q \neq 3, \ 3^2, \ 3^3, \ 5. $
Then,
$$H_3^Q(X) \cong \left\{
\begin{array}{ll}
0 , &\quad \mathrm{if} \ n>1 \\
\Z/(q^2 -1) \oplus (\Z/p)^{d(d+1)/2} , &\quad \mathrm{if} \ n=1
\end{array}
\right.
$$
\item % Let $()$ be....
Let $X$ be the spherical quandle $S^{n}_{q}$ over $\F_q$. Then, there are $[1/2]$-isomorphisms
$$ H_3^Q(X) \cong_{[1/2]} \left\{
\begin{array}{ll}
0 , &\quad \mathrm{if} \ n>2 \\
\Z/(q^2 -1) \oplus \Z /(q - \delta_q) , &\quad \mathrm{if} \ n=2
\end{array}
\right.
$$
Here $\delta_q = \pm 1 $ is according to $q \equiv \pm 1$ $(\mathrm{mod} \ 4)$.
\end{enumerate}
\end{thm}
This theorem mostly settles a problem posed by Kabaya \cite{ILDT} for computing the homology $H_3^Q(\SQ_{q}^n ) $ with $n=1$.

\begin{proof}
(I) %Let $X= \mathsf{Sp}^{n}_q$ be the symplectic quandle over $\F_q$.
Recalling $\As(X) \cong \Z \times Sp(2n ;\F_q)$ from the proof of Theorem \ref{thm1b2},
% (see Proposition \ref{dfg31sp}).
we particularly see the kernel $\Ker (\psi_X) \cong \Z$.

First, we deal with the case $n=1$.
Notice $\mathrm{Inn}(X) \cong Sp(2;\F_q)$ and $|\mathrm{Inn}(X)| /|X|=q(q^2-1 )/q^2-1 = q$.
Therefore the purpose $H_3^Q(X) \cong_{[1/p]} \Z/ q^2 -1$ without $p$-torsion follows from Theorem \ref{thm3homology}.
So, we now focus on the $p$-torsion of $H_3^Q(X) $ with $n=1$. Consider the rack space $B(X,X)$ associated to the primitive $X$-set,
whose $P$-sequence is given by
$$ H_3^{\rm gr}(\mathrm{Stab}(x_0)) \lra \pi_2(BX ) \stackrel{\mathcal{H} }{\lra} H_2(B(X,X)) \lra H_2^{\rm gr}(\mathrm{Stab}(x_0))\lra 0 . $$
We can easily see the stabilizer $\mathrm{Stab}(x_0) \cong \Z \times (\Z/p)^d$.
% by Proposition \ref{dfg31sp}.
Since $\pi_2(BX) \cong \Z \oplus \Z/(q^2-1) \oplus (\Z/p)^d$ by Theorem \ref{thm1b2}, the exact sequence is rewritten as
\begin{equation}\label{Pd1l2} H_3^{\rm gr} (\Z \times (\Z/p)^d) \lra \Z \oplus \Z/(q^2-1) \oplus (\Z/p)^d \stackrel{\mathcal{H} }{\lra} H_2(B(X,X)) \lra H_2^{\rm gr}(\Z \times (\Z/p)^d)\lra 0 .
\end{equation}
We here claim that the Hurewicz map $\mathcal{H}_{(p)}$ is a splitting injection.
Indeed, it follows from \eqref{eesp}
that the composite $P_* \circ \mathcal{H} : \pi_2(BX ) \ra H_2(BX) \cong \Z \oplus (\Z/p)^d$ is surjective, where $P:B(X,X) \ra BX$ is the covering.
Consequently, this claim made the previous sequence above into %means
$$ H_2(B(X,X)) \cong_{(p)}\pi_2 (BX) \oplus H_2^{\rm gr} (\Z \times (\Z/p)^d) \cong_{(p)} \Z_{(p)} \oplus (\Z/p)^{2d} \oplus (\Z/p)^{d(d- 1)/2} . $$
Hence, compared with the isomorphism \eqref{eq.H3RQ}, we have $ H_3^Q(X) \cong_{(p)} (\Z/p)^{d(d+1) /2} $ as desired.

Next, when $n \geq 2$, we will show $H_3^Q(X) =0$ as stated.
It is shown \cite[\S 4.2]{Nos3}
that the second group homology of the stabilizer $\mathrm{Stab}(x_0)= \pi_1(B(X,X))$ is zero,
and that $H_2(BX) \cong \Z$.
Therefore, the $P$-sequencesarising from the projection $P: B(X,X ) \ra BX$ are written as
$$
\xymatrix{ H_3^{\rm gr}( \mathrm{Stab}(x_0 )) \ar[d]_{P_*} \ar[r]_{ \ \ \ \ \ \delta_*}& \pi_2(B(X,X)) \ar[d]^{\cong}_{P_*} \ar[r] & H_2(B(X,X) ) \ar[d]_{P_*} \ar[r] & 0 \\
H_3^{\rm gr}(\Z \times Sp(n;\F_q )) \ar[r] & \pi_2(BX) \ar[r] & H_2(BX ) \ \ (\cong \Z) \ar[r] & 0. }
$$
By the stability theorem \ref{factsl21}, the left vertical map $P_*$ surjects onto $\Z/(q^2-1)$.
Since $\pi_2(BX) \cong \Z \oplus \Z/(q^2-1)$ by Theorem \ref{thm1b2}, the delta map $\delta_*$ is
surjective in torsion part. Therefore a diagram chasing can show $H_2(B(X,X))=\Z$.
Using \eqref{eq.H3RQ} again, we have the goal $H_3^Q(X) =0$.

\

\noindent
(II)
The calculations of the third homology $H_3^Q(X) $ for
the spherical quandle $X=S^n_q$ over $\F_q$
can be shown in a similar way to the symplectic case.
The point is that the homology $H_i^{\rm gr} (\mathrm{Stab}(x_0))$
is $[1/2]$-isomorphic to $H_i^{\rm gr} (O(n;\F_q )) $ for $i \leq 3$ (see Proposition \ref{factsp2}).
Furthermore, we employ the computation of $H_2^Q(X) $ shown in \cite[\S 4.3]{Nos3}.
%; see Proposition \ref{dfg31sha} for details.
To summarize, using results of Quillen and Friedlander explained in \S \ref{asssym},
we can complete the proof similarly, although we omit writing the details.
\end{proof}

%\begin{proof}
% $$
% \xymatrix{ H_3^{\rm gr}(\Z \times O(n-1;\F_q )) \ar[d] \ar[r]& \pi_2(BX) \ar[d] \ar[r] & H_2(X,X ) \ar[d] \ar[r] & H_2(\Z \times O(n-1;\F_q )) \ar[d] \ar[r] & 0 \\
% H_3^{\rm gr}(\Z \times O(n;\F_q )) \ar[r] & \pi_2(BX) \ar[r] & H_2(X ) \ar[r] & H_2(\Z \times O(n;\F_q ) ) \ar[r] & 0}
%$$
% By the sequence we have the isomorphism $ H_3^Q(X) \cong \Z/(q^2 -1) \oplus (\Z/p)^{d(d+1)/2} $. \end{proof}

\subsection{Example 3; the extended quandles. }\label{s87asl}
We now discuss the groups $H_3^Q (\X) $ and $\Pi_2 (\X)$ of extended quandles $\X$.
%The disucssions and results in this subsection will be used in a subsequent paper \cite{Nos4}.
\begin{thm}\label{3hom9}
Let $X$ be a connected quandle of type $t_X $. %Let $|X|$ be $<\infty $.
Let $p: \X \ra X$ be the universal covering mentioned in \S \ref{vanishes}.
If $H_3^{\gr} (\As (X))$ is finitely generated,
then there are $[1/t_X]$-isomorphisms
$$ H_3^Q(\X ) \cong_{[1/t_X]} \Pi_2(\X) \cong_{[1/t_X]} H_3^{\rm gr}(\As (X)). $$
Here the second isomorphism $ \Pi_2(\X) \cong H_3^{\rm gr}(\As (X))$ is obtained from the composite $\Theta_{X } \circ p_* $.
%2-torsion groups.
\end{thm}
\begin{proof}
We now construct the first isomorphism $\Pi_2(\X) \cong_{[1/t_X]} H_3^{Q}(\X) $.
% modulo $t_X $.
Consider the rack space $ B(\X,\X)$ associated with the primitive $\X$-set.
Notice from Remark \ref{setrans} and Theorem \ref{extthm} (ii) that this $\pi_1$ is an abelian group $\Z \times \Ker (p_*) $.
Then the Postnikov tower is expressed by
$$ H_3^{\gr}( \Z \times \Ker (p_*) ) \lra \pi_2(B\X ) \xrightarrow{\ \mathcal{H}_{\X} \ } H_2(B(\X,\X) ) \lra H^{\rm gr}_2( \Z \times \Ker (p_*) ) \ra 0 \ \ \ \ \ \ ({\rm exact}). $$
Notice from Theorem \ref{extthm} (iii)(iv) that $H_i^{\gr}( \Z \times \Ker (p_*) )$ is annihilated by $t_X $ for $i \geq 2 $.
Hence, the Hurewicz map $\mathcal{H}_{\X}$ amounts for the claimed $[1/t_X]$-isomorphism $\Pi_2(\X) \cong H_3^{Q}(\X) $.

To prove the $[1/t_X]$-isomorphism $\Pi_2(\X) \cong H_3^{\rm gr}(\As (X)) $ and the latter part,
we first note $H_2^Q(\X) \cong_{[1/t_X]} 0$ by Theorems \ref{theorem} and \ref{extthm}.
Considering the $[1/t_X]$-isomorphisms $p_*: H_3^{\gr }( \As(\X)) \ra H_3^{\gr }(\As(X))$ and the $T$-map $\Theta_X: \Pi_2(\X) \ra H_3^{\rm gr}(\As (\X)) $ in
Theorem \ref{extthm} (v) and Theorem \ref{theorem} respectively,
we obtain the composite $[1/t_X]$-isomorphism $ p_* \circ \Theta_X: \Pi_2(\X) \ra H_3^{\rm gr}(\As (X)) $ as desired.
\end{proof}
Furthermore, let us show a lemma which will be used in a subsequent paper \cite{Nos4}.
\begin{lem}\label{k1pg}
Let $X$ be a connected Alexander quandle of type $t_X $.
Let $p: \X \ra X$ be the universal covering.
Then the induced map $p_*: H_3^Q(\X) \ra H_3^Q(X)$ is injective up to $2t_X $-torsion.
\end{lem}
\begin{proof} By the proof of Theorem \ref{Ale3homology} (see \S \ref{sSproof}),
the Hurewicz map $\mathcal{H}_{X}: \pi_2(B(X,X)) \ra H_2(B(X,X))$ is injective up to $2t_X $-torsion.
Consider the Postnikov tower with respect to the $p: \X \ra X$:
$$
\xymatrix{ H_3^{\gr }(\Z \times \Ker (p_*))_{(\ell)} \ar[r]^{\ \ \ \ \ \ 0} \ar[d]_{p_*} & \pi_2 (B\X)_{(\ell)} \ar[r]^{\!\!\!\! \!\! \mathcal{H}_{\X} } \ar[d]_{p_*} & H_2(B(\X,\X))_{(\ell)} \ar[d]_{p_*} \ar[r] & H_2^{\gr}( \Z \times \Ker (p_*))_{(\ell)} = 0 \ar[d]_{p_*} \\
H_3^{\gr}( \mathrm{Stab}(x_0))_{(\ell)} \ar[r]^{\ \ \ \ \ \ 0} & \pi_2(BX)_{(\ell)} \ar[r]^{\!\!\!\! \!\! \mathcal{H}_{X} } & H_2(B(X,X))_{(\ell)} \ar[r] & H_2 ^{\gr}( \mathrm{Stab}(x_0))_{(\ell)} .}
$$
Here a prime $\ell$ is relatively prime to $2t_X$. Since the upper $ \mathcal{H}_{\X}$ is a $[1/2t_X]$-isomorphism (see the previous proof of Theorem \ref{3hom9}),
the vertical map $ p_* : H_2(B(\X,\X))_{(\ell)} \ra H_2(B(X,X))_{(\ell)}$
is injective. Hence, by \eqref{eq.H3RQ} as usual, the map $p_* : H_3^Q(\X) \ra H_3^Q(X)$ turns out to be a $[1/2t_X]$-injection.
\end{proof}
%\begin{lem}\label{k1pg2}
%Let $X$ be a connected Alexander quandle of type $t_X $.\end{lem} \begin{proof}

%\end{proof}

\appendix

\section{Proof of Theorem \ref{ehyouj2i3}}\label{kokoroni}

We will show Theorem \ref{ehyouj2i3} as a result of Proposition \ref{ehyouji}.
This proposition provides an algorithm to compute the first rack homology as follows:
\begin{prop}\label{ehyouji}
Let $X$ be a quandle, and $Y$ an $X$-set.
Decompose $Y $ into the orbits as $Y= \sqcup_{i \in I}Y_i$.
For $i \in I $, choose an arbitrary element $y_i \in Y_i$, and denote by $\mathrm{Stab}(y_i) \subset \As (X)$ the stabilizer subgroup of $y_i$.
Then $H_1^R (X,Y ) $ is isomorphic to the direct sum of the abelianizations of $ \mathrm{Stab}(y_i) $.
Precisely, $H_1^R (X,Y ) \cong \oplus_{i \in I} ( \mathrm{Stab}(y_i) )_{\mathrm{ab}}.$
\end{prop}
\begin{proof} %[Proof of Theorem \ref{ehyouji}]
%We investigate the first homology of the space $ B(X,Y_i)$ for each $i \in I$. %, from covering theory.
Recall from Proposition \ref{propFRS},
that each connected component of the $B(X,Y)$ is $B(X,Y_i )$,
and $\pi_1(B(X,Y))$ is the stabilizer $\mathrm{Stab}(y_i )$. % by observing the covering transformation group (see for details).
Thereby $H_1(B(X,Y_i ) ) \cong \pi_1(B(X,Y_i ) )_{\rm ab} \cong \mathrm{Stab}(y_i)_{\rm ab} $.
Hence, we conclude
\[H_1^R(X,Y) \cong H_1 (B(X,Y) ) \cong \bigoplus_{i \in I} H_1 (B(X,Y_i ) ) \cong \bigoplus_{i\in I} \mathrm{Stab}(y_i)_{\mathrm{ab}}. \]
\vskip -1pc
\vskip -0.7pc
\end{proof}
\begin{proof}[Proof of Theorem \ref{ehyouj2i3}]
We first show \eqref{1111} below.
Let $Y=X$ be the primitive $X$-set.
For each $x_i \in X_i$, we have $e_{x_i} \in \mathrm{Stab}(x_i)$ since $x_i\tri x_i =x_i$.
Hence, the restriction of $\varepsilon_i: \As (X) \ra \Z$ on $\mathrm{Stab}(x_i)$ is also surjective, and
permits a section $\mathfrak{s}:\Z \ra \mathrm{Stab}(x_i)$ defined by $\mathfrak{s}(n)=e_{x_i}^n$.
Here we remark that the action of $\Z$ on $\mathrm{Stab}(x_i) \cap \Ker (\varepsilon_i) $ induced by the section is trivial.
Indeed, the equality \eqref{hasz2} implies $g^{-1} e_{x_i} g =e_{x_i} \in \As(X)$ for any $g \in \mathrm{Stab}(x_i) $.
We therefore have $ \mathrm{Stab}(x_i)_{\mathrm{ab}} \cong \bigl( \mathrm{Stab}(x_i) \cap \Ker (\varepsilon_i) \bigr)_{\mathrm{ab}} \oplus \Z $.
Hence it follows from Proposition \ref{ehyouji} that
\begin{equation}\label{1111} H_1^R(X,X ) \cong \bigoplus_{i\in \OX} \mathrm{Stab}(x_i) _{\mathrm{ab}} \cong \Z^{\oplus \OX } \oplus \bigoplus_{i\in \OX } \bigl( \mathrm{Stab}(x_i) \cap \Ker (\varepsilon_i)\bigr)_{\mathrm{ab}}. \end{equation}

Finally, it is sufficient to show that $H_2^Q(X)$ is isomorphic to the last summand.
Recall $ H^R_2 (X) \cong H^R_1(X,X) $ in Remark \ref{setrans}.
It is known \cite[Theorem 2.1]{LN} that $ H^R_2 (X) \cong H^Q_2 (X ) \oplus \Z^{\oplus \OX}$, and that a basis of the $\Z^{\oplus \OX}$ is represented by $(x_i,x_i) \in C_2^R(X)$ for $i \in \OX$.
By comparing the basis with the isomorphisms in \eqref{1111}, we complete the proof.
\end{proof}

\section{Proof of Proposition \ref{prop.homcol}}\label{ds1epr23}\begin{proof}[Proof of Proposition \ref{prop.homcol}.]
We first construct an $X$-coloring
from any element $(x_1, \dots, x_{\# L }, f )$ in \eqref{1v4}.
Let us denote by $\gamma_i$ the oriented arc associated to the meridian $\mathfrak{m}_i $,
and color the $\gamma_i$ by the $x_i \in X $.
For each $i $, we consider the path $ \mathcal{P}_i$ along the longitude $ \mathfrak{l}_i$ as illustrated in the figure below.
Furthermore, let $\alpha_0(=\gamma_i) , \alpha_1 , \dots, \alpha_{N_i-1 }, \alpha_{N_i}=\alpha_0 $ be the over-paths on this $ \mathcal{P}_i $,
and let $ \beta_{k} \in \pi_1(S^3 \setminus L) $ be the meridian corresponding to
the arc that divides the arcs $\alpha_{k-1 }$ and $\alpha_k$.
Then, we define a map $ \mathcal{ C}: \{$over arcs of $D \ \} \ra X $ %the color on the arc $ \alpha_{k } $
by the formula
$$ \mathcal{ C}(\alpha_{k }) := x_i \cdot \bigl( f( \beta_{1}^{\epsilon_1}) \cdots f( \beta_{k-1}^{\epsilon_{k-1}} ) \bigr) \in X . $$
Here $\epsilon_j \in \{ \pm 1\}$ is the sigh of the crossing of $ \alpha_{j}$ and $ \beta_{j} $.
Note $ \mathcal{ C}(\alpha_{N_i }) = \mathcal{ C}(\alpha_0 ) =x_i $
since the longitude $\mathfrak{l}_i $ equals the product $ \beta_{1}^{\epsilon_1} \beta_2^{\epsilon_2} \cdots \beta_{N_i -1 }^{\epsilon_{N_i -1 }} $ by definition.
Here we claim that this $ \mathcal{ C} $ is an $X$-coloring.
For this purpose, using \eqref{hasz2}, we notice equalities
\[ e_{\mathcal{ C} ( \alpha_{k })} = e_{x_i \cdot f( \beta_{1}^{\epsilon_1}) \cdots f( \beta_{k-1}^{\epsilon_{k-1}} ) }\]
\begin{equation}\label{AI2} = \bigl( f( \beta_{1}^{\epsilon_1}) \cdots f( \beta_{k-1}^{\epsilon_{k-1}} ) \bigr)^{-1} f(\mathfrak{m}_i) \bigl( f( \beta_{1}^{\epsilon_1}) \cdots f( \beta_{k-1}^{\epsilon_{k-1}} ) \bigr) = f( \alpha_{k })\in \mathrm{As}(X).
\end{equation}% \label{AI}
Hence, with respect to the crossing between $\alpha_k$ and $\beta_k$ with $k \leq N_i$, we have the following equality in $X $:
$$ \mathcal{ C}(\alpha_{k }) \lhd \mathcal{ C}(\beta_{k }) = \mathcal{ C}(\alpha_{k }) \cdot e_{\mathcal{ C}(\beta_{k })}^{\epsilon_k}= \mathcal{ C}(\alpha_{k }) \cdot f( \beta_{k }^{\epsilon_k })
= x_i \cdot \bigl( f( \beta_{1}^{\epsilon_1}) \cdots f( \beta_{k-1}^{\epsilon_{k-1}} ) \bigr) \cdot f( \beta_{k }^{\epsilon_k}) = \mathcal{ C}(\alpha_{k +1}) . $$
%Notice that, from the definitions of the $f$, for each $\beta_k $ we have $f(\beta_k )=e_{y_k }$ for some $y_k \in X$ because of Wirtinger presentation and the equality \eqref{ }.
%Hence, the $ \mathcal{ C}(\alpha_{k }) $ is reformulated as
This equality means that this $ \mathcal{ C} $ turns out to be an $X$-coloring as desired.
Here note the equality $f= \Gamma_{ \mathcal{ C} }$ for such an $X$-coloring $\mathcal{ C} $ coming from the original $ f $ (see the previous \eqref{AI2}).

To summarize, we obtain a map from the set \eqref{1v4} to the set $\mathrm{Col}_X(D)$
which carries such $(x_1, \dots, x_{\# L }, f )$ to the above $ \mathcal{ C} $.
Moreover, by construction,
it is the desired inverse mapping of the $\Gamma_{\bullet }$, which proves the desired bijectivity.
%Conversely, if $f$ is $\Gamma_{C}$ for some coloring $ \mathcal{ C} $, the coloring constructed from $f$ is equal to $\mathcal{C}$. But we leave the details to the reader.
\end{proof}

\begin{figure}[htpb]
\begin{center}
\begin{picture}(50,74)
\put(-100,50){\large $\alpha_0 = \gamma_i $}
\put(-13,24){\large $\alpha_1 $}
\put(14,26){\large $\alpha_2 $}

\put(96,66){\Large $\mathcal{P}_i $}

\put(-66,33){\pc{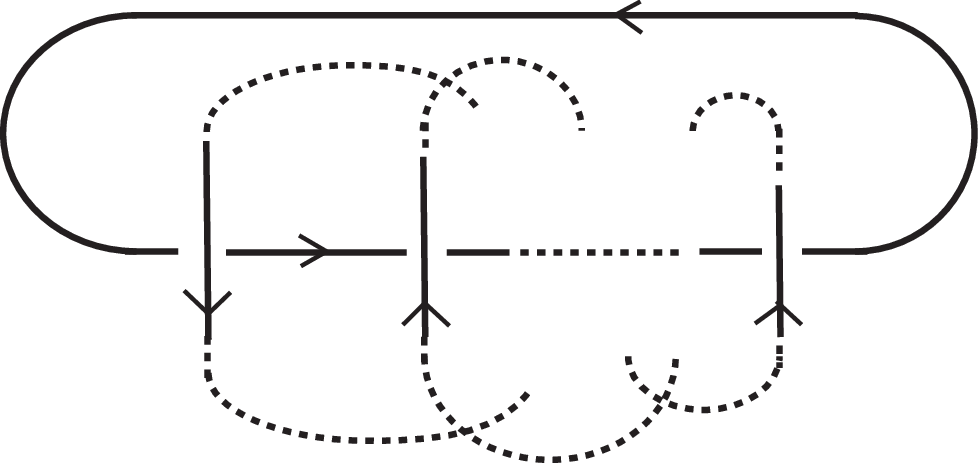}{0.34}}

\put(-26,43){\large $\beta_1 $}
\put(11,43){\large $\beta_2 $}
\put(69,41){\large $\beta_{N_i} $}
\put(33,43){\large $\cdots $}
\end{picture}
\end{center}
\vskip -0.95pc
%\vskip -1.7pc
%\vskip -0.5pc
%\caption{\label{koutenpn131} Positive and negative crossings.}
\end{figure}

\begin{cor}\label{prop.homcol2}Let $X$ be a quandle
such that the map $X \ra \As(X)$ sending $x$ to $e_x$ is injective. Let $D$ be a diagram of an oriented link $L$.
We fix a meridian $\mathfrak{m}_i \in \pi_1(S^3 \setminus L)$ of each link-component.
Then, the set of $X$-colorings of $D$ is bijective to the following set:
\begin{equation}\label{ffu2} \{ \ f \in \Hom_{\gr} (\pi_1(S^3 \setminus K), \As(X)) \ | \ f(\mathfrak{m}_i)= e_{x_i} \mathrm{\ for \ some \ } x_i \in X \ \}.
\end{equation}
\end{cor}
\begin{proof}
By Proposition \ref{prop.homcol}, it is enough to show that
such an $f$ in \eqref{ffu2} satisfies the equality $ x_i \cdot f(\mathfrak{l}_i) =x_i $.
Actually, noting $ e_{x_i}=f(\mathfrak{m}_i)= f(\mathfrak{l}_i)^{-1 }f(\mathfrak{m}_i)f(\mathfrak{l}_i)= f(\mathfrak{l}_i)^{-1 }e_{x_i}f(\mathfrak{l}_i)= e_{x_i \cdot f(\mathfrak{l}_i)}$ by \eqref{hasz2},
the injectivity $X \hookrightarrow \As(X) $ concludes the desired $ x_i \cdot f(\mathfrak{l}_i) =x_i $.
\end{proof}
\noindent
For example, every quandle in Section \ref{twoto} satisfies the injectivity of $X \ra \As(X)$.

\subsection*{Acknowledgment}
The author expresses his gratitude to Tomotada Ohtsuki and Seiichi Kamada for comments and advice.
He also thanks the referee for kindly reading this paper.
%%roup cohomologies and 3-manifolds. He is particularly grateful to Yuichi Kabaya for useful discussions and making several suggestions for improvement.

\vskip 1pc

\normalsize

Faculty of Mathematics, Kyushu University, 744, Motooka, Nishi-ku, Fukuoka, 819-0395, Japan

\

E-mail address: {\tt nosaka@math.kyushu-u.ac.jp}


\begin{thebibliography}{99} 
\normalsize



\ifx\undefined\bysame
\newcommand{\bysame}{\leavevmode\hbox to3em{\hrulefill}\,}
\fi


%\bibitem[AG]{AG} N. Andruskiewitsch, M. Gra\~{n}a. {\it From racks to pointed Hopf algebras}. Adv. Math., {\bf 178} (2003), 177--243.


\bibitem[AP]{AP}
D. Arlettaz, N. Pointet-Tischler, {\it Postnikov invariants
of H-spaces}, Fund. Math. {\bf 161} (1-2) (1999), 17--35.

\bibitem[Bro]{Bro}
K. S. Brown, 
{\it Cohomology of Groups}, 
Graduate Texts in Mathematics, {\bf 87}, Springer-Verlag, New York, 1994.

%\bibitem[CE]{CE} H. Cartan, S. Eilenberg, 
%{\it Homological algebra}. Princeton University Press, Princeton, N. J. 1956.
 
\bibitem[CEGS]{CEGS} 
J. S. Carter, J. S. Elhamdadi, M. Gra\~na, M. Saito,
{\it Cocycle knot invariants from quandle modules and 
generalized quandle homology}, 
Osaka J. Math. {\bf 42} (2005), 499--541.

\bibitem[CF]{CF} P. E. Conner, E. N. Floyd,  {\it Differentiable Periodic Maps}, Ergeb. Math. Grenzgeb. N. F., vol 33, Springer-Verlag, Berlin, G\"{o}ttingen, Heidelberg 1964.


\bibitem[CJKLS]{CJKLS} 
J. S. Carter, D. Jelsovsky, S. Kamada, L. Langford, M. Saito, 
{\it Quandle cohomology and state-sum invariants of knotted curves and surfaces}, 
Trans. Amer. Math. Soc. {\bf 355} (2003), 3947--3989. 


\bibitem[CKS]{CKS}
J. S. Carter, S. Kamada, M. Saito,
{\it Geometric interpretations of quandle homology}, 
J. Knot Theory Ramifications {\bf 10} (2001), 345--386.

\bibitem[Cla1]{Cla} 
F.J.-B.J. Clauwens, 
{\it The algebra of rack and quandle cohomology}, 
J. Knot Theory Ramifications {\bf 11} (2011), 1487--1535.


\bibitem[Cla2]{Cla2} \bysame, 
{\it The adjoint group of an Alexander quandle}, arXiv:math/1011.1587

\bibitem[Cla3]{Cla3} \bysame, 
{\it Small simple quandles}, arXiv:math/1011.2456 


\bibitem[DW]{DW}
R. Dijkgraaf, E. Witten, {\it Topological gauge theories and group cohomology},
Comm. Math. Phys. {\bf 129} (1990), 393--429.


\bibitem[E1]{Eis1} 
M. Eisermann, 
{\it Knot colouring polynomials},
Pacific Journal of Mathematics {\bf 231} (2007), 305--336.

\bibitem[E2]{Eis2} 
\bysame, {\it Quandle coverings and their Galois correspondence}, Fund. Math. {\bf 225} (2014), 103--167.


%\bibitem[E2]{Eis3} 
%\bysame, {\it Homological characterization of the unknot}, J. Pure Appl. Algebra {\bf 177} (2003), no. 2, 131--157. 


\bibitem[Eb]{Eb} 
J. F. Ebert, {\it The icosahedral group and the homotopy of
the stable mapping class group}, M\"{u}nster Journal of Mathematics {\bf 3} (2010), 221--232.


\bibitem[EGS]{EGS} P. Etingof, R. Guralnick, A. Soloviev, 
{ \it Indecomposable set-theoretical solutions to the quantum Yang-Baxter equation on a set with a prime number of elements}, J. Algebra {\bf 242} (2001) 709--719.

%\bibitem[EN]{EN} 
%H. Endo, S. Nagami, {\it Signature of relations in mapping class groups and non-holomorphic Lefschetz fibrations,}
%Trans. Amer. Math. Soc. {\bf 357} (2005), no. 8, 3179--3199.

\bibitem[FM]{FM}
B. Farb, D. Margalit, {\it A primer on mapping class groups}, PMS {\bf 50}, Princeton
University Press, 2011.


%\bibitem[FR]{FR]} R. Fenn, C. Rourke, 
%{\it Racks and links in codimension two}, 
%J. Knot Theory Ramifications {\bf 1} (1992) 343--406. 


\bibitem[FRS1]{FRS1} 
R. Fenn, C. Rourke, B. Sanderson,
{\it Trunks and classifying spaces}, 
Appl. Categ. Structures {\bf 3} (1995) 321--356.

\bibitem[FRS2]{FRS2} 
\bysame, 
{\it The rack space}, Trans. Amer. Math. Soc. {\bf 359} (2007), no. 2, 701--740.


\bibitem[FP]{FP}
Z. Fiedorowicz, S. Priddy, {\it Homology of classical groups over finite fields and their associated infinite loop spaces},
 Lecture Notes in Mathematics, {\bf 674}, Springer, Berlin, 1978. 




%\bibitem[FP2]{FP2}\bysame, {\it Homology of classical groups over finite fields},
%Lecture Notes in Math, Springer-Verlag, {\bf 674}, 1978.

\bibitem[Fri]{Fri}
E. M. Friedlander, {\it Computations of $K$-theories of finite fields}, Topology {\bf 15} (1976), no. 1, 87--109.
%B. Farb, D. Margalit, {\it A primer on mapping class groups}, PMS {\bf 50}, PrincetonUniversity Press, 2011.


%\bibitem[Ga]{Ga}
%S. Galatius, {\it Mod $p$ homology of the stable mapping class group}, Topology {\bf 43} (2004), 1105--1132.

\bibitem[Ger]{Ger} S. Gervais, {\it Presentation and central extension of mapping class groups}, Trans. Amer. Math. Soc. {\bf 348} (1996) 3097--3132.

%\bibitem[GMT]{GMT} S. Galatius, I. Madsen, U. Tillmann, {\it Divisibility of the stable Miller-Morita-Mumford classes}, Journal of Amer. Math. Soc., {\bf 19} (2006), 759--779. 

%\bibitem[Har]{Har} J. L. Harer, {\it Stability of the homology of the mapping class groups of
%orientable surfaces}, Ann. of Math., {\bf 121} (2), (1985) 215--249.

%\bibitem[Hut]{Hut} K. Hutchinson, {\it A refined Bloch group and the third homology of $SL_2$ of a field}, 
%available at, arXiv/math:1101.3279v2 


\bibitem[HN]{HN}
E. Hatakenaka, T. Nosaka, 
{\it Some topological aspects of 4-fold symmetric quandle invariants of 3-manifolds},
to appear Internat. J. Math.

%\bibitem[Hum]{Hum} 
%S. Humphries, {\it Generators for the mapping class group}, from: Topology of
%Low-dimensional Manifolds, LNM 722 (1979) 44--47.

%\bibitem[IK]{IK}  
%A. Inoue, Y. Kabaya, 
%{\it Quandle homology and complex volume}, arXiv:math/1012.2923.

\bibitem[Iva]{Iva} N. V. Ivanov, {\it Stabilization of the homology of Teichm\"uller modular groups}, Algebrai Analiz {\bf 1} (1989), 110--126 (Russian); translation in Leningrad Math. J. {\bf 1} (1990) 675--691.


%\bibitem[Knu]{Knu} K. P. Knudson, {\it Homology of linear groups}, Progress in Mathematics, {\bf 193}, Birkh\"{a}user Verlag, Basel, 2001.

\bibitem[ILDT]{ILDT} T. Ohtsuki (ed), {\it Problems on Low-dimensional Topology,} in a conference ``Intelligence of Low-dimensional Topology" 2010, Kokyuroku RIMS {\bf 1716}, 119--128.


\bibitem[JW]{JW} 
J. D. S. Jones and B. W. Westbury. {\it Algebraic $K$-theory, homology spheres, and the $\eta$-invariant},
Topology, {\bf 34} (1995), 929--957.

\bibitem[Joy]{Joy} 
D. Joyce, 
{\it A classifying invariant of knots, the knot quandle},
J. Pure Appl. Algebra {\bf 23} (1982) 37--65.

%\bibitem[Kam]{Kam} S. Kamada, {\it Braid and knot theory in dimension four,} Mathematical Surveys and Monographs,
%{\bf 95}. American Mathematical Society, Providence, RI, 2002.

%\bibitem[JW]{JW} J. D. S. Jones, B. Westbury, {\it Algebraic $K$-theory, homology spheres, and the
%$\eta$-invariant}, Topology {\bf 34}(4) (1995), 929--957.


\bibitem[Kab]{Kab} 
Y. Kabaya, 
{\it Cyclic branched coverings of knots and quandle homology},  Pacific Journal of Mathematics, {\bf 259} (2012), No. 2, 315--347.

%
%\bibitem[Kor]{Kor} 
%M. Korkmaz, {\it Low-dimensional homology groups of mapping class groups: a survey}, Turkish J. Math. {\bf 26} (2002), no. 1, 101--114.

\bibitem[LN]{LN} R. Litherland, S. Nelson, {\it The Betti numbers of some finite racks,} J. Pure Appl. Algebra {\bf 178} (2003), no. 2, 187--202. 

%\bibitem[Kab]{Kab} Y. Kabaya, {\it Cyclic branched coverings of knots and quandle homology,} arXiv:math/1012.3729

%\bibitem[Mat]{Mat}
%M. Matsumoto, {\it A presentation of mapping class groups in terms of Artin
%groups and geometric monodromy of singularities}, Math. Ann. {\bf 316} (2000) 401.
%418.


\bibitem[McC]{McC} 
J. McCleary, 
{\it A user's guide to spectral sequences}, 
Second edition. Cambridge Studies in Advanced Mathematics, {\bf 58}. 
Cambridge University Press, Cambridge, 2001.

%\bibitem[MR]{MR}
%G. Masbaum and J. Roberts, {\it On central extensions of mapping class groups}, Math. Annalen, {\bf 302} (1995), 131--150.


%\bibitem[Mat]{Matsumoto}
%Y. Matsumoto, {\it Lefschetz fibrations of genus two -- a topological approach}, Proceedings of the 37th Taniguchi
%Symposium on Topology and Teichmuller Spaces, ed. Sadayoshi Kojima et al., World Scientific (1996), 123--148.

%\bibitem[Mil]{Milnor} J. W. Milnor, 
%{\it On the 3-dimensional Brieskorn manifolds $M(p, q, r)$, 
%Knots, groups and
%3-manifolds}, Ann. of Math. Studies {\bf 84}, Princeton Univ. Press, 1975, 175--225.


\bibitem[Moc]{Moc2} T. Mochizuki,
{\it The 3-cocycles of the Alexander quandles $\mathbb{F}_q [T]/(T- \omega)$}, Algebraic and Geometric Topology. {\bf 5} (2005), 183--205.

\bibitem[MT]{TM}
I. Madsen, U. Tillmann, {\it The stable mapping class group and $Q(\mathbb{CP}_+^{\infty})$}, Invent.
Math. {\bf 145} (2001), no. 3, 509--544

\bibitem[N1]{Nos1}
T. Nosaka, 
{\it On homotopy groups of quandle spaces and the quandle homotopy invariant of links}, Topology and its Applications {\bf 158} (2011), 996--1011.


\bibitem[N2]{Nos2}
\bysame,
{\it Quandle homotopy invariants of knotted surfaces}, %???
Mathematische Zeitschrift 2013, {\bf 274}, 341--365.


%\bibitem[N3]{Nos5} \bysame,　{\it Computations of adjoint groups and second homologies of quandles}, preprint

\bibitem[N3]{Nos4}
\bysame,
{\it On third homologies of groups and of quandles 
via the Dijkgraaf-Witten invariant and Inoue-Kabaya map}, to appear Algebraic and Geometric Topology.

\bibitem[N4]{Nos3}
\bysame,
{\it Computations of adjoint groups and second homologies of quandles}, preprint

% \bibitem[PR]{PR}L. Paris, D. Rolfsen, {\it 
% Geometric subgroups of mapping class groups,} J. Reine Angew. Math. {\bf 521} (2000), 47--83.

%\bibitem[NP]{NP}
%M. Niebrzydowski, J. H. Przytycki, {\it The quandle of the trefoil as the Dehn
%quandle of the torus}, Osaka Journal of Mathematics {\bf 46} (2009) 645--659.

%\bibitem[MR]{MR}
%G. Masbaum, J. D. Roberts, {\it On central extensions of mapping class groups,} Math. Ann. {\bf 302} (1995), no. 1, 131--150.

\bibitem[Ros]{Ros}
J. Rosenberg, {\it Algebraic $K$-theory and its applications}. Graduate Texts in Mathematics,
{\bf 147}. Springer-Verlag, New York, (1994)


\bibitem[Rol]{Rolfsen}
D. Rolfsen,
{\it Knots and links}, Math. Lecture Series, 7,
Publish or Perish, Inc., Houston, Texas, 1990, Second printing, 
with corrections.


\bibitem[RS]{RS} 
C. Rourke, B. Sanderson, 
{\it A new classification of links and some calculation using it},
arXiv:math.GT/0006062.
%\bibitem[PR]{PR}L. Paris, D. Rolfsen, {\it 
%Geometric subgroups of mapping class groups,} J. Reine Angew. Math. {\bf 521} (2000), 47--83.

\bibitem[RW]{RW}O. Randal-Williams, {\it Resolutions of moduli spaces and homological stability},
arXiv:0909.4278.

%\bibitem[Sie]{Sie} A. J. Sieradski, {\it Combinatorial squashings, 3-manifolds, and the third homology of groups}, Invent. Math., {\bf 84} (1986) 121--139.


\bibitem[Sou]{Sou} C. Soul\'{e}, {\it Op\'{e}rations en K-th\'{e}orie alg\'{e}brique}, Canad. J. Math. {\bf 37} (1985) 488--550.

\bibitem[Y]{Y}
D. N. Yetter, {\it Quandles and monodromy}, J. Knot Theory Ramifications {\bf 12} (2003), 523--541.


%\bibitem[Zab]{Zab} J. Zablow, {\it On relations and homology of the Dehn quandle}. Algebr. Geom.
%Topol. {\bf 8} (2008), no. 1, 19--51.



\end{thebibliography}
\end{document}